\documentclass[reqno]{amsart}
\usepackage[margin=1.5in]{geometry}
\usepackage{palatino}
\usepackage{mathpazo}
\usepackage{amsmath, amssymb, amsthm, mathrsfs}
\usepackage{tikz-cd}
\usepackage{hyperref}
\usepackage{enumerate}

\newtheorem{theorem}{Theorem}[section]
\newtheorem{introthm}{Theorem}

\newtheorem{lemma}[theorem]{Lemma}
\newtheorem{proposition}[theorem]{Proposition}

\newtheorem{conjecture}[theorem]{Conjecture}
\newtheorem*{theorem*}{Theorem}

\theoremstyle{definition}
\newtheorem{definition}[theorem]{Definition}
\newtheorem{setup}[theorem]{Setup}
\newtheorem{example}[theorem]{Example}

\theoremstyle{remark}
\newtheorem{remark}[theorem]{Remark}

\newcommand{\LL}{\mathbb{L}}
\newcommand{\TT}{\mathbb{T}}
\newcommand{\Gm}{\mathbb{G}_m}
\newcommand{\cO}{\mathcal{O}}
\newcommand{\cL}{\mathcal{L}}

\newcommand{\K}{\mathbb{K}}
\newcommand{\F}{\mathbb{F}}
\newcommand{\DR}{\mathbf{DR}}
\newcommand{\Spec}{\mathrm{Spec\,}}

\newcommand{\bfem}[1]{\textbf{\textit{#1}}}

\newcommand{\dCrit}{\mathrm{dCrit}}
\newcommand{\vdim}{\mathrm{vdim}}
\newcommand{\Het}{\mathbb{H}}
\newcommand{\LFc}{\mathfrak{L}\mathcal{F}_c}
\newcommand{\Hom}{\mathrm{Hom}}
\newcommand{\rk}{\mathrm{rk}\,}

\makeatletter
\g@addto@macro{\endabstract}{\@setabstract}
\newcommand{\authorfootnotes}{\renewcommand\thefootnote{\@fnsymbol\c@footnote}}%
\makeatother

\begin{document}
\title[Monodromic Perverse Sheaves on Contact Stacks]{Monodromic
Perverse Sheaves on Shifted Contact Stacks}
\author{Efe \.{I}zbudak}
\date{\today}

\begin{center}
  \LARGE
  Monodromic Perverse Sheaves on Shifted Contact Stacks \par \bigskip

  \normalsize
  \authorfootnotes
  Efe \.{I}zbudak\footnote{efe.izbudak@studium.uni-hamburg.de}\par \bigskip

  Department of Mathematics, Universit\"{a}t Hamburg\par \bigskip
\end{center}


\begin{abstract}
  Applying the
  BBDJS minimal model to the derived symplectification of a $-1$-shifted contact
  derived Artin stack and descending algebraically along the structural free
  $\mathbb{G}_m$-action, we construct an $\ell$-adic perverse sheaf
  on any oriented
  such stack, and use Verdier's specialization equivalence for
  monodromic sheaves
  to equip it with a tame twisted monodromy operator $\theta$. We show that the
  local fundamental class of a Legendrian $L$ satisfies $\theta \circ \mu_L =
  (-1)^{\vdim L}\mu_L$, so that on Legendrians of odd virtual dimension, the class vanishes due to $\theta$-invariance. Moreover, both parities occur
  already on the $A_1$ chart while an odd Legendrian can carry a nonzero local
  class. We further formulate a contact analogue of Joyce's conjecture for a
  graded orientation, in which the orientation datum is twisted by the parity of
  the virtual dimension. Under the assumption of a monodromic
  refinement of the
  symplectic conjecture, we construct the categorified Legendrian 2-categories
  $\mathfrak{L}\mathcal{F}_c(X)$ and $LLeg_0$ via $\ell$-adic
  pull-push functors.
  Finally, we show that the contact Behrend function is identically
  $1$, so that the
  associated Donaldson--Thomas invariant is the compactly supported
  \'etale Euler
  characteristic of the classical truncation, and that the higher traces of
  $\theta$ recover the singularity type that the first trace discards. As an
  application, we show that the symplectic invariant of a derived intersection of conic
  Lagrangians in a cotangent bundle vanishes identically, while the contact
  invariant computes the Euler characteristic of the projectivized
  intersection, with an explicit formula for conormal bundles.
\end{abstract}
\newpage
\tableofcontents
\newpage
\section{Introduction and Summary}
\label{sec:intro}

An oriented $-1$-shifted symplectic derived Artin stack carries a perverse
sheaf of vanishing cycles \cite[Theorem 1.3]{BenBassat2015}, whose pointwise
Euler characteristic is the Behrend function \cite{Behrend} and whose
hypercohomology computes the Donaldson--Thomas invariant when the stack is
proper \cite{JoyceSong}. Amorim and Ben-Bassat extend the assignment to Lagrangian
correspondences in \cite[Theorem 6.11]{AmorimBassat}, where the $2$-category of
such correspondences is linearized by pull-push between these sheaves.

As opposed to the symplectic case, $-1$-shifted contact derived Artin stack carries no vanishing cycles, and the
construction has to be performed on its symplectification $p \colon
\widetilde{Z} \to Z$, a $\Gm$-torsor whose total space is $-1$-shifted
symplectic and whose structural action scales the symplectic form with weight
$1$ \cite[Theorem 3.9]{kib2}. The perverse sheaf of $\widetilde{Z}$ is then
monodromic along the orbits, and the pushforward $Rp_*$ retains of it only the
generalized $1$-eigenspace of the orbit monodromy, which is zero whenever that
eigenvalue does not occur. We descend this sheaf along Verdier's
specialization equivalence \cite{Verdier} instead, and work on $Z$ with the descended
perverse sheaf together with a tame automorphism.

The descent machinery differs from the symplectic construction in that the 
$\Gm$-action does not preserve the orientation datum of a chart. The orientation
local system has monodromy $(-1)^{n+1}$ along the orbits in a chart with base
of dimension $n$, so the geometric monodromy $T$ of the vanishing cycles and the
twisted operator $\theta$ that glues across the atlas differ by a global sign.
The same sign appears in the fundamental classes of Legendrians as the parity
of the virtual dimension, and also in the invariants of Sections \ref{sec:dt} and
\ref{sec:applications}, which are defined as traces of $\theta$.

The shifted contact structures are those of Berktav \cite{kib1, kib2}, and the
categories linearized below are the non-linear derived Legendrian categories
$\mathcal{F}_c(X)$ and $Leg_n$ constructed in \cite{IzbudakBerktav2} via the ``contact--symplectic dictionary'' established in \cite{IzbudakBerktav2026}. Section
\ref{sec:review} recollects the material used in the proofs. The local input is
the contact Darboux atlas of \cite[Theorem 3.7]{kib2}, collected in
\S\ref{sec:charts} in the form used throughout. A chart has a smooth base $U$ of
dimension $n$ and a potential $s$, and its symplectification is the derived
critical locus of $\widetilde{f}(x,t) = t \cdot s(x)$ over $U \times \Gm$
(Theorem \ref{thm:atlas}). The equivariant refinement of the atlas, describing the local
model at a point with linearly reductive stabilizer, is the subject of the
companion paper \cite[Theorem 3.8(2)]{Izbudak_Darboux} and is not used here.

The sheaf-theoretic input is the perverse sheaf of vanishing cycles of Brav,
Bussi, Dupont, Joyce and Szendr\H{o}i \cite{BBDJS}. Their results are stated
over $\mathbb{C}$, and loc.\ cit.\ leaves the $\ell$-adic case open, the
obstruction being that the Thom--Sebastiani theorem was not available there in
the required generality. Both ingredients are now in place. The perverse sheaf on an oriented d-critical
stack is constructed in \cite[Theorem 1.3]{BenBassat2015} for any theory of
perverse sheaves, in particular for Laszlo--Olsson's $\ell$-adic perverse
sheaves over $\K$, and the Thom--Sebastiani theorem for Donaldson--Thomas
perverse sheaves is proved in \cite[\S 4.3]{KPS}, with \cite[Remark 7.20]{KPS}
recording that the arguments there apply to $\ell$-adic perverse sheaves over
$(-1)$-shifted symplectic stacks over $\K$. On stacks the statements we need are
proved directly for $\ell$-adic coefficients by Descombes, monodromic
Thom--Sebastiani in \cite[Theorem 4.2]{Descombes} and the glued
Donaldson--Thomas sheaves together with smooth pullback functoriality in
\cite[Corollary 5.9]{Descombes}, whose standing hypotheses are those fixed in
the conventions below. We cite \cite{BBDJS} for the constructions and their
local models, and \cite{BenBassat2015, KPS, Descombes} for their validity in
the present setting.

The descended sheaf, the two operators it carries, and the comparison between
them are the content of the first main theorem.

\begin{introthm}[Monodromic Perverse Sheaf] \label{thm:A}
  Let $Z$ be an oriented $-1$-shifted contact derived Artin stack over $\K$ with
  symplectification $p \colon \widetilde{Z} \to Z$. Then
  \begin{enumerate}[(i)]
  \setlength\itemsep{0.1in}
  \item There is an $\ell$-adic perverse sheaf $\mathcal{P}_Z$ on $Z$ carrying
  two commuting tame automorphisms, the geometric monodromy $T$ and the twisted
  monodromy operator $\theta$, such that the pair $(\mathcal{P}_Z, T)$
  corresponds to the BBDJS perverse sheaf $\mathcal{P}_{\widetilde{Z}}$ of
  $\widetilde{Z}$ under Verdier's specialization equivalence along $p$ (cf.
  Theorem \ref{thm:monodromic_sheaf}).
  \item In a contact Darboux chart with base of dimension $n$ the operator
  $\theta$ is $(-1)^n$ times the geometric monodromy of the homogeneous
  potential along the $\Gm$-orbits, and the chartwise operators commute with the
  comparison isomorphisms of the atlas, hence glue (cf. Definition
  \ref{def:twisted}, Proposition \ref{prop:theta_glues}).
  \item The two operators differ by the orientation sign character $\varpi$,
  which is identically $-1$ for contact orientation data, so that $T = -\theta$
  globally (cf. Proposition \ref{prop:sign_character}).
  \item Consequently $Rp_* \mathcal{P}_{\widetilde{Z}}$ retains only the part on
  which $\theta$ acts with generalized eigenvalue $-1$, while the pair
  $(\mathcal{P}_Z, \theta)$ retains all of it (cf. Lemma
  \ref{lem:monodromic_descent}(1)).
  \end{enumerate}
\end{introthm}

The second main theorem computes the action of $\theta$ on the fundamental
class of a Legendrian and analyzes the obstruction that the resulting eigenvalue
imposes.

\begin{introthm}[Parity of the local fundamental class] \label{thm:B}
  Let $\varphi \colon L \to X$ be a proper oriented Legendrian in a $-1$-shifted
  contact derived scheme in Darboux form over a chart with base of dimension
  $n$. Then
  \begin{enumerate}[(i)]
  \setlength\itemsep{0.1in}
  \item There is a local class $\mu_L \colon \F_L[\vdim L] \to \varphi^!
  \mathcal{P}_X$, independent of the choices in its construction, with
  \[ \theta \circ \mu_L = (-1)^{n + \rk E}\, \mu_L , \]
  where $E$ is the bundle of a local presentation of $\varphi$ in the sense of
  Theorem \ref{thm:JS_local} (cf. Proposition \ref{prop:local_verification}).
  \item The exponent reduces modulo $2$ to the virtual dimension, so that
  $\theta \circ \mu_L = (-1)^{\vdim L} \mu_L$, and on a Legendrian of odd
  virtual dimension $\theta$-invariance forces $\mu_L = 0$ (cf. Proposition
  \ref{prop:parity_formula}).
  \item Both parities occur over a single contact chart, already for $s = x^2$
  on $\mathbb{A}^1$ (cf. Example \ref{ex:A1_parity}).
  \item There is a Legendrian of odd virtual dimension whose local class is
  nonzero and generates a one-dimensional $\Hom$-group, so the obstruction is
  not vacuous (cf. Example \ref{ex:odd_nonzero}, Proposition
  \ref{prop:odd_nonvanishing}).
  \item Grading the orientation datum by the virtual dimension absorbs the sign,
  and for the graded datum the local class is invariant (cf. Definition
  \ref{def:graded_orientation}, Remark \ref{rem:local_graded}).
  \end{enumerate}
\end{introthm}

We note that Conjecture \ref{conj:joyce}, the contact analogue of Joyce's conjecture for
Lagrangians, is stated in that graded form, and part (v) above is its local
verification. The linearization of the Legendrian categories is independent of
this conjecture and instead depends on the symplectic conjecture of Amorim and
Ben-Bassat and on a monodromic refinement (the assumptions (J1)--(J3) of
Setup \ref{setup:joyce}) in which the same grading of the orientation datum
enters. That refinement is not a consequence of the symplectic conjecture. The
attractor case of (J1)--(J2) is a theorem of Kinjo, Park and Safronov
\cite[Theorem 7.23]{KPS}, proved for a $\Gm$-action preserving both the
d-critical structure and the orientation, whereas the action here scales the
potential with weight one and moves the orientation by $(-1)^{n+1}$. What
supports (J3) is instead the unconditional local verification of parts (i) and
(v) of Theorem \ref{thm:B}, together with the grading that Example
\ref{ex:odd_nonzero} shows to be forced.

\begin{introthm}[Perverse linearization] \label{thm:C}
  Assume the hypotheses (J1)--(J3) of Setup \ref{setup:joyce}, and let $X$ be a
  $1$-shifted contact derived stack. Then
  \begin{enumerate}[(i)]
  \setlength\itemsep{0.1in}
  \item There is a bicategory $\LFc(X)$ of Legendrian correspondences in $X$,
  enriched over graded $\F$-vector spaces equipped with an automorphism, whose
  $2$-morphism spaces are the monodromic \'etale hypercohomologies of the
  sheaves of Theorem \ref{thm:A} (cf. Theorem \ref{thm:linear_category}).
  \item There is a linear weak $2$-category $LLeg_0$ whose objects are
  $0$-shifted contact derived stacks and whose $1$-morphisms are proper spans
  carrying graded orientations (cf. Theorem \ref{thm:global_2cat}).
  \item The non-linear $2$-category $Leg_0$ of Legendrian correspondences of
  \cite{IzbudakBerktav2} maps to $LLeg_0$ by a $2$-functor
  \[ F \colon Leg_0^{\mathrm{pr}} \longrightarrow LLeg_0 \]
  which is the identity on objects and $1$-morphisms and sends a $2$-morphism to
  the image of its fundamental class, where $Leg_0^{\mathrm{pr}}$ is the
  sub-$2$-category on all objects whose $1$- and $2$-morphisms are proper and
  carry graded orientations (cf. Theorem \ref{thm:linearization_functor}).
  \item Every class in the image of $F$ is $\theta$-invariant, and $F$ is
  monoidal for the contact product (cf. Theorem
  \ref{thm:linearization_functor}).
  \end{enumerate}
\end{introthm}

The results that follow below are again unconditional. The traces of $\theta$ on the
hypercohomology of $\mathcal{P}_Z$ are invariants of a proper oriented contact
stack, and the fourth main theorem evaluates the local function whose integral
gives the first of these traces, and computes the higher ones on an
$A_k$-singularity. The contrast with the symplectic theory is sharp here. There
the Behrend function takes the value $k$ at an $A_k$-point and the invariant is
a virtual count, while the contact function is constant and the singularity type
is recorded by the higher traces instead. Remark \ref{rem:not_virtual} exhibits
a family in which the two behave differently, the symplectic invariant being
deformation invariant and the contact one not.

\begin{introthm}[The contact Behrend function] \label{thm:D}
  Let $Z$ be an oriented $-1$-shifted contact derived Artin stack over $\K$.
  Then
  \begin{enumerate}[(i)]
  \setlength\itemsep{0.1in}
  \item The $\ell$-adic contact Behrend function of $Z$ is identically $1$ (cf.
  Proposition \ref{prop:behrend_one}).
  \item If $Z$ is proper, then
  \[ \mathrm{DT}(Z) = \chi_{et, c}(Z^{\mathrm{cl}},
  \overline{\mathbb{Q}}_\ell) , \]
  the compactly supported \'etale Euler characteristic of the classical
  truncation (cf. Theorem \ref{thm:integration}, Proposition
  \ref{prop:behrend_one}).
  \item The higher traces $\mathrm{DT}_m$ of $\theta$ are again constructible
  integrals (cf. Theorem \ref{thm:integration} and
  \S\ref{subsec:highertraces}).
  \item For an $A_k$-singularity with $k \geq 2$ they leave $\{\pm 1\}$ first at
  $m = k+1$, while for $k = 1$ they are constantly $1$. In either case the
  sequence $(\mathrm{DT}_m)_{m \geq 1}$ recovers $k$ (cf. Example
  \ref{ex:Ak_higher}).
  \item For $Z$ proper the higher traces assemble into a rational zeta function,
  which for every stratification adapted to $\mathcal{P}_Z$ is the product of
  the local zeta functions of the strata, weighted by their Euler
  characteristics (cf. Proposition \ref{prop:zeta}).
  \end{enumerate}
\end{introthm}

The last main theorem applies these invariants to the derived intersections of
conic Lagrangians in a cotangent bundle of a smooth variety, a class of global
objects on which the invariant of the symplectic theory vanishes identically.

\begin{introthm}[Conic Lagrangian intersections]\label{thm:E}
  Let $M$ be a smooth variety, let $\Lambda_1, \Lambda_2$ be conic Lagrangians
  in $T^*M \setminus 0$, and let $Z$ be the derived intersection of the
  Legendrians $\mathbb{P}\Lambda_i \subseteq \mathbb{P}T^*M$. Then
  \begin{enumerate}[(i)]
  \setlength\itemsep{0.1in}
  \item $Z$ is a $-1$-shifted contact derived stack whose symplectification is
  the derived intersection $\widetilde{Z}$ of $\Lambda_1$ and $\Lambda_2$ (cf.
  Theorem \ref{thm:conic}(1)).
  \item If $\widetilde{Z}$ is of finite type, every $\Gm$-invariant
  constructible function on it integrates to zero. In particular the
  Behrend-weighted Euler characteristic of $\widetilde{Z}$, which is the
  invariant of the symplectic theory, vanishes identically (cf. Lemma
  \ref{lem:torsor_vanishing}, Theorem \ref{thm:conic}(2)).
  \item If $Z$ is proper and admits contact orientation data, its invariant is
  the Euler characteristic of the classical projectivized intersection (cf.
  Theorem \ref{thm:conic}(3)).
  \item For $M$ proper and $\Lambda_i$ the conormal bundles of two smooth
  subvarieties meeting cleanly along $W$ with excess $e$, the invariant is
  $e \cdot \chi_{et, c}(W)$, and it vanishes for a transverse intersection (cf.
  Proposition \ref{prop:conormal}).
  \item The self-intersection of the conormal of a subvariety of codimension at
  least $2$ admits contact orientation data if and only if $\dim M$ is even (cf.
  Proposition \ref{prop:conormal_orientation}).
  \end{enumerate}
\end{introthm}

Overall, Theorems \ref{thm:A}, \ref{thm:B}, \ref{thm:D} and \ref{thm:E} are
unconditional, and are proved in Sections \ref{sec:perverse}, \ref{sec:joyce},
\ref{sec:dt} and \ref{sec:applications} respectively, apart from the last
assertion of Theorem \ref{thm:D}, which is proved in \S\ref{subsec:zeta}. The
assumptions (J1)--(J3) enter only in Theorem \ref{thm:C} and Section
\ref{sec:linearization}.

\medskip
\paragraph{\bfem{Conventions and notations}.}
Let $\K$ denote an algebraically closed field of characteristic zero.
All commutative differential graded algebras (cdgas) are concentrated
in nonpositive degrees and defined over $\K$ \cite{PTVV}. Classical
$\K$-schemes are locally of finite type, and all derived $\K$-schemes
and stacks $X$ are locally finitely presented. Derived Artin stacks are in
addition assumed quasi-separated and to have affine stabilizers, as required by
the contact Darboux atlas of Theorem \ref{thm:atlas} and by the stacky
Thom--Sebastiani theorem of \cite[\S 4.3]{KPS}. For constructible and
perverse sheaves on an Artin stack $X$, we work in the
derived category $D^b_c(X_{et}, \F)$ of algebraically constructible complexes on
the lisse-\'etale site, where $\F$ is a finite extension of $\mathbb{Q}_\ell$ or
$\overline{\mathbb{Q}}_\ell$ for a prime $\ell$. We abbreviate the subscript to
$X_{et}$ throughout, and for schemes and Deligne--Mumford stacks this agrees
with the small \'etale site. We use the $\ell$-adic sheaf theories of
Laszlo--Olsson
\cite{LO1, LO2, LO3} and Liu--Zheng \cite{LZ}. See Section \ref{subsec:elladic}.

We fix a topological generator $\gamma = (\zeta_N)_N$ of
$\pi_1^{\mathrm{et}}(\Gm) \simeq \hat{\mathbb{Z}}(1)$, a compatible system of
primitive roots of unity in $\K$. All monodromy operators are the actions of
$\gamma$. We fix once and for all an isomorphism $\F(1) \simeq \F$ and suppress Tate twists except where a
monodromy weight is at issue. Cohomological shifts follow the conventions of
\cite[Section 2]{BBDJS}. Virtual dimensions of symplectifications are normalized
by
\[ \vdim \widetilde{Y} = \vdim Y + 1, \]
the shift accounting for the $\Gm$-fiber of the torsor $\widetilde{Y} \to Y$. In
particular $\vdim \widetilde{L} = \vdim L + 1$ for a Legendrian $L$ and
$\vdim \widetilde{X} = \vdim X + 1$ for the ambient contact stack. We write
$X^{\mathrm{cl}}$ for the classical truncation of a derived stack $X$, and
$f^{\mathrm{cl}}$ for the induced morphism of classical truncations.

\section{Recollection of Prior Work}
\label{sec:review}

We review the material used in the proofs, namely shifted symplectic and contact
structures, $\ell$-adic sheaf theory on Artin stacks, vanishing cycles and the
BBDJS perverse sheaf, monodromic complexes, the local models of Lagrangians and
the contact Darboux charts. The companion paper covers the same shifted
contact geometry in \cite[\S\S 2.1--2.3]{Izbudak_Darboux}. We keep only the
statements that the later sections invoke by number, in the form in which they
are used here, and we collect the local models in \S\ref{sec:charts} beside the
charts they present rather than separately. The remaining subsections are the
$\ell$-adic material, which the companion paper does not use.

\subsection{Shifted Symplectic Structures and Lagrangian Intersections}
\label{subsec:symplectic}
All constructions take place in derived algebraic geometry over $\K$ in the
sense of To\"en--Vezzosi \cite{ToenHAG}; see \cite{Toen2014} for a survey. We
write $\mathcal{A}^{p(, \mathrm{cl})}(X, n)$ for the space of \bfem{(closed)
$p$-forms of degree $n$} on a derived Artin stack $X$, $L_{qcoh}(X)$ for its
stable $\infty$-category of quasi-coherent complexes, $\LL_X \in L_{qcoh}(X)$
for the cotangent complex and $\TT_X \simeq
\mathbb{R}\underline{\mathcal{H}om}(\LL_X, \cO_X)$ for its dual. An
\bfem{$n$-shifted symplectic structure} is a closed $2$-form
$\omega \in \mathcal{A}^{2, \mathrm{cl}}(X, n)$ whose underlying $2$-form is
non-degenerate, in the sense that $\Theta_\omega \colon \TT_X \to \LL_X[n]$ is
an equivalence \cite{PTVV}, and a \bfem{Lagrangian structure} on $L \to X$ is
an isotropic structure whose induced map on tangent complexes is an equivalence
onto the annihilator \cite{AmorimBassat}. A derived intersection of Lagrangians
in an $n$-shifted symplectic stack is $(n-1)$-shifted symplectic \cite{PTVV},
and for $k < 0$ such a stack admits smooth atlases in Darboux form
\cite[Theorem 2.10]{BenBassat2015}.

Amorim
and Ben-Bassat use the perverse sheaf of vanishing cycles on these intersections
to construct, for a $1$-shifted symplectic $S$, a weak $2$-category
$\mathfrak{L}Lag(S)$ whose $2$-morphism spaces are the \'etale hypercohomologies
\[ \mathcal{H}om(N_0, N_1) := \Het^\bullet\bigl(N_0 \times^h_{L_{01}} N_1,\,
\mathcal{P}_{N_{01}}[-\vdim N_0]\bigr) \]
\cite[Theorem 1.3]{AmorimBassat}. The contact constructions of Sections
\ref{sec:joyce} and \ref{sec:linearization} are modeled on this one.

\subsection{The Symplectification Torsor}
\label{subsec:contact}
An \bfem{$n$-shifted contact structure} on a derived Artin stack $X$ is a line
bundle $\cL$ together with a non-degenerate Maurer--Cartan element defining a
contact form $\alpha \in \mathcal{A}^1(X, \cL, n)$, and we write
$(X, \cL, \alpha)$ for the datum \cite[Definition 3.6]{kib1}. Nothing below is
computed on $X$ itself; everything is computed on the $\Gm$-torsor over it, for
which we follow \cite{kib1, kib2, IzbudakBerktav2026, IzbudakBerktav2}.

\begin{theorem}[Symplectification \cite{kib1, kib2}]\label{thm:symplectification}
  An $n$-shifted contact derived stack $(X, \cL, \alpha)$ is the base of a
  principal $\Gm$-bundle
  $p \colon \widetilde{X} \longrightarrow X$,
  its \bfem{derived symplectification}, whose total space carries an
  $n$-shifted symplectic form $\omega_{\widetilde{X}}$ scaled with weight $1$ by
  a free structural $\Gm$-action \cite[Theorem 3.9]{kib2}, \cite[Theorem
  4.7]{kib1}. The passage back down is the weight-one quotient of
  \cite[Theorem 3.7]{IzbudakBerktav2026}, which recovers $\alpha$ as the
  weight-zero reduction of the contraction of $\omega_{\widetilde{X}}$ with the
  fundamental vector field of the action.
\end{theorem}

\begin{definition}[Legendrian Morphism
  \cite{IzbudakBerktav2026}]\label{def:legendrian}
  A morphism $f \colon L \to X$ into an $n$-shifted contact stack is
  \bfem{Legendrian} when it carries an isotropic structure
  $f^*\alpha \simeq 0$ for which
  \[ \cO_L \longrightarrow \LL_{L/X} \otimes^{\mathbb{L}}_{\cO_L} f^*\cL[n-1]
    \longrightarrow \TT_L \]
  is a stable fiber sequence \cite[Definition 2.39]{IzbudakBerktav2026}. What
  the condition buys is the statement every computation below relies on: the
  symplectification $\widetilde{L} \to \widetilde{X}$ of a Legendrian is a
  $\Gm$-equivariant Lagrangian \cite[proof of Theorem
  4.1]{IzbudakBerktav2026}.
\end{definition}

Note further that intersecting two Legendrians stays within contact geometry as per the following result.

\begin{theorem}[Legendrian intersections {\cite[Theorem
  4.1]{IzbudakBerktav2026}, \cite[Theorem 5.1]{IzbudakBerktav2}}]
  \label{thm:review_intersection}
  For Legendrian morphisms $L_1, L_2 \to X$ into an $n$-shifted contact stack,
  the derived fiber product $L_{01} = L_1 \times_X^h L_2$ is $(n-1)$-shifted
  contact, which is what keeps the constructions of Sections
  \ref{sec:linearization} and \ref{sec:applications} on the contact side. The
  statement is relative as well: for Legendrians $N_1 \to L_{01}$ and
  $N_2 \to L_{12}$ over a third Legendrian $L_1$, the fiber product
  $N_1 \times^h_{L_1} N_2$ is Legendrian over $L_{02}$.
\end{theorem}

These intersections carry the non-linear $2$-categories of contact geometry.
For an $n$-shifted contact stack $X$, the weak $2$-category
$\mathcal{F}_c(X)$ has Legendrian objects $L \to X$, Legendrians
$N \to L_{01}$ in the intersection as $1$-morphisms, and Legendrian spans
descending from the symplectification as $2$-morphisms \cite[Theorem
1.1]{IzbudakBerktav2}, its composition operations descending to the contact
quotients and satisfying the coherence conditions there by \cite[Theorem
6.1]{IzbudakBerktav2}.

The Cartesian product $X_1 \times X_2$ of contact stacks carries no contact
structure, as the contact line bundles and the conformal scaling factors do
not agree across the factors. The product of the symplectifications, however,
does descend to one. Give
\[ \widetilde{X}_1^- \times \widetilde{X}_2 \]
the diagonal $\Gm$-action. Each factor form has weight $1$, and so does
$\omega_1^- \boxplus \omega_2$, and by \cite[Theorem 3.7]{IzbudakBerktav2026} the
stack quotient
\begin{equation}\label{eq:contact_product}
  X_1 \circledast X_2 := \bigl[ (\widetilde{X}_1^- \times \widetilde{X}_2) /
  \Gm \bigr]
\end{equation}
carries an $n$-shifted contact structure whose symplectification is
$\widetilde{X}_1^- \times \widetilde{X}_2$. We call it the \bfem{contact
product}. This is the derived analogue of the classical contact product, and its
virtual dimension is $\vdim \widetilde{X}_1 + \vdim \widetilde{X}_2 - 1$. 

A Legendrian correspondence between $n$-shifted contact stacks $X_1$ and $X_2$ is
then a Legendrian morphism $L \to X_1 \circledast X_2$, equivalently a span
$X_1 \leftarrow L \rightarrow X_2$ whose associated principal $\Gm$-bundles
form a $\Gm$-equivariant Lagrangian span
$\widetilde{X}_1 \leftarrow \widetilde{L} \rightarrow \widetilde{X}_2$ inside
$\widetilde{X}_1^- \times \widetilde{X}_2$. This construction generates the
global $(\infty, 2)$-category of Legendrian correspondences, denoted $Leg_n$
\cite[Corollary 8.1]{IzbudakBerktav2}.

\subsection{\texorpdfstring{$\ell$}{l}-adic Sheaf Theory on Artin
Stacks}\label{subsec:elladic}

The category $D^b_c(X_{et}, \F)$ fixed in the conventions is the one of
Laszlo--Olsson, constructed together with the six operations $f^*$, $Rf_*$,
$Rf_!$, $f^!$, $\otimes$, $R\mathcal{H}om$ for finite coefficients in
\cite{LO1} and for adic coefficients in \cite{LO2}, and enhanced
$\infty$-categorically by Liu--Zheng \cite{LZ}. Its perverse t-structure is constructed in \cite[Theorem
5.1]{LO3}, and in \cite[Theorem 4.2]{LO3} for $X$ of finite type. The heart
$\mathrm{Perv}(X_{et}, \F)$ is the category of perverse sheaves on $X$, and
perverse sheaves form a stack for the smooth topology
\cite[Proposition 7.1]{LO3}.

The six operations, with their adjunctions, base change and Verdier duality, are
available in this setting by \cite{LO1, LO2, LZ}, and we use them without
further comment. Proper base change is \cite[Corollary 6.2.2]{LZ} and
\cite[Theorem 12.1]{LO2}, the K\"unneth formula is \cite[Theorem 6.2.1]{LZ},
$Rf_! \simeq Rf_*$ for proper $f$ is \cite[Proposition 6.2.11]{LZ}, and the
gluing of perverse sheaves is \cite[Proposition 7.1]{LO3}. These four are
invoked by name in what follows.
\vspace{1em}

\subsection{Vanishing Cycles, d-Critical Loci, and the BBDJS Perverse Sheaf}
\label{subsec:vanishing}

Let $U$ be a smooth $\K$-scheme, $f \in \cO(U)$ a regular function, and
$U_0 = f^{-1}(0)$. The \'etale vanishing cycles functor
$\phi_f \colon D^b_c(U_{et}, \F) \to D^b_c((U_0)_{et}, \F)$ and its monodromy
automorphism are constructed in \cite[Exp. XIII]{SGA7II}. We write
$\phi^p_f := \phi_f[-1]$ for the perverse-normalized functor, and write
\[ \mathcal{PV}_{U, f} := \phi^p_f(\F[\dim U]), \]
a perverse sheaf supported on $\mathrm{Crit}(f) \cap U_0$, equipped with the
monodromy automorphism given by the action of the generator $\gamma$.

\begin{theorem}[Thom--Sebastiani \cite{BBDJS}]\label{thm:TS}
  For regular functions $f \in \cO(U)$ and $g \in \cO(V)$ on smooth $\K$-schemes
  there is a canonical isomorphism
  \[ \mathcal{PV}_{U \times V,\, f \boxplus g} \simeq
  \mathcal{PV}_{U, f} \boxtimes \mathcal{PV}_{V, g} \]
  over $\mathrm{Crit}(f) \times \mathrm{Crit}(g)$ \cite[Theorem
  2.13]{BBDJS}. Its availability over $\K$ with $\ell$-adic coefficients, and
  on stacks, is the sheaf-theoretic input recorded in Section
  \ref{sec:intro}, the precise references being
  \cite[Proposition 4.3, Corollary 4.4 and Remark 7.20]{KPS} and
  \cite[Theorem 4.2, Corollary 5.9]{Descombes}. The functor $\phi^p_f$ commutes with proper pushforward
  \cite[Theorem 2.11(ii)]{BBDJS}. The interaction of these isomorphisms with the
  monodromy operators is the subject of Section \ref{sec:perverse}.
\end{theorem}

\begin{example}[Non-degenerate quadratic forms \cite{BBDJS}]\label{ex:quadratic}
  For the quadratic form $q(z) = z_1^2 + \cdots + z_r^2$ on
  $\mathbb{A}^r$ one has
  $\mathrm{Crit}(q) = \{0\}$ and $\mathcal{PV}_{\mathbb{A}^r, q}
  \cong \F_{\{0\}}$,
  an isomorphism canonical only up to sign, since it depends on a
  generator of the
  rank-one vanishing cohomology. The monodromy operator acts by $(-1)^r$
  \cite[Example 2.14]{BBDJS}, \cite[Exp. XV]{SGA7II}.

  The relative form of this statement is the one we use. Let $\Phi \colon U
  \hookrightarrow V$ be a closed embedding of smooth $\K$-schemes with
  $f = g \circ \Phi$ and $\Phi|_{\mathrm{Crit}(f)}$ an isomorphism onto
  $\mathrm{Crit}(g)$,
  \[
    \begin{tikzcd}[row sep=large, column sep=large]
      \mathrm{Crit}(f) \arrow[r, hook] \arrow[d, "\sim"',
      "\Phi|_{\mathrm{Crit}(f)}"] &
      U \arrow[d, hook, "\Phi"] \arrow[dr, "f"] & \\
      \mathrm{Crit}(g) \arrow[r, hook] & V \arrow[r, "g"'] & \mathbb{A}^1 ,
    \end{tikzcd}
  \]
  and let $q_{UV}$ be the induced non-degenerate quadratic form
  on the normal bundle $N_{UV}$, of rank $r = \dim V - \dim U$. Then
  $\det(q_{UV})$ determines a principal $\mu_2$-bundle $P_\Phi$ on
  $\mathrm{Crit}(f)$, and there is a natural isomorphism of perverse sheaves
  \[ \Theta_\Phi \colon \mathcal{PV}_{U, f} \longrightarrow
    \Phi|^*_{\mathrm{Crit}(f)}\bigl( \mathcal{PV}_{V, g} \bigr)
  \otimes_{\mu_2} P_\Phi \]
  compatible with Verdier duality and with the monodromy operators
  \cite[Definition 5.2, Lemma 5.3 and Theorem 5.4]{BBDJS}. These
  \bfem{stabilization isomorphisms} generate, together with \'etale comparisons
  of charts, the descent data of the global construction below
  \cite[Theorem 6.9]{BBDJS}, \cite[Section 4]{BenBassat2015}.
\end{example}

Joyce introduces \bfem{algebraic d-critical loci} $(X, s)$, classical spaces
recording local presentations as critical loci of regular functions on smooth
schemes, together with a canonical line bundle $K_{X, s}$ on
$X^{\mathrm{red}}$ restricting on a critical chart $(U, f)$ to
$K_U^{\otimes 2}|_{\mathrm{Crit}(f)^{\mathrm{red}}}$ \cite{Joyce_dCrit}. An
\bfem{orientation} is a square root of $K_{X, s}$
\cite[Section 2]{Joyce_dCrit}. The truncation of a $-1$-shifted symplectic
derived Artin stack carries an induced d-critical structure \cite{BenBassat2015}.

\begin{theorem}[The BBDJS perverse sheaf \cite{BBDJS, BenBassat2015}]\label{thm:BBDJS_sheaf}
  Let $(X, \omega)$ be an oriented $-1$-shifted symplectic derived Artin stack
  over $\K$. There is a perverse sheaf $\mathcal{P}_X \in
  \mathrm{Perv}(X_{et}, \F)$, \'etale-locally isomorphic on a Darboux chart
  $\dCrit(f) \to X$ to $\mathcal{PV}_{U, f} \otimes_{\mu_2} Q_{U, f}$, where
  $Q_{U, f}$ is the $\mu_2$-local system comparing the orientation with the
  square root $K_U$ of $K_{X,s}|_{\dCrit(f)}$. The local models glue along the
  comparison isomorphisms of \cite[Theorem 6.9]{BBDJS}. Over
  $\mathbb{C}$ this is
  \cite[Theorem 6.9 and Corollary 6.11]{BBDJS}. For the $\ell$-adic perverse
  sheaves of \cite{LO1, LO2, LO3} over $\K$ it is \cite[Theorem 1.3]{BenBassat2015},
  which requires only that the characteristic be different from $2$ and
  moreover records the compatibility $t^*(\mathcal{P}_X)[n] \simeq
  \mathcal{P}_T$
  for $t \colon T \to X$ smooth of relative dimension $n$.
\end{theorem}

In classical Donaldson--Thomas theory over $\mathbb{C}$, the \bfem{Behrend
function} $\nu_X \colon X(\mathbb{C}) \to \mathbb{Z}$ is a canonical
constructible function \cite[Section 1.2]{Behrend}. For
$X = \mathrm{Crit}(f) \subseteq U$ it is given by
$\nu_X(x) = (-1)^{\dim U}\bigl(1 - \chi(MF_f(x))\bigr)$, where $MF_f(x)$
denotes the Milnor fiber. Behrend proves that the Donaldson--Thomas invariant
of a proper moduli scheme with symmetric obstruction theory equals the Euler
characteristic weighted by $\nu_X$ \cite{Behrend}. Section \ref{sec:dt}
constructs the contact analogue.

\subsection{Milnor Fibers and the \texorpdfstring{$A_k$}{Ak}-Singularities}
\label{subsec:milnor}

The invariants of Sections \ref{sec:dt} and \ref{sec:applications} are
evaluated on hypersurface singularities, and the $A_k$-singularities recur
throughout. We collect the classical facts used.

Let $f \colon U \to \mathbb{A}^1$ be a function on a smooth variety of
dimension $n$ with $f(x) = 0$, and let $F_x$ denote its Milnor fiber at $x$.
Over $\mathbb{C}$ this is 
\[
f^{-1}(\epsilon) \cap B_\delta(x)\quad \text{for}\quad 0 < |\epsilon| \ll \delta \ll 1,
\]
and its reduced cohomology is the stalk at
$x$ of the vanishing cycles $\phi_f \overline{\mathbb{Q}}_\ell$, in the sense
that 
\[\widetilde{H}^j(F_x) = \mathcal{H}^j(\phi_f
\overline{\mathbb{Q}}_\ell)_x.\]
Over $\K$ we take the right side as the definition, the comparison
theorem and the Lefschetz principle transporting every statement below. The
Milnor fiber carries the monodromy automorphism $T_x$, the action of the
generator $\gamma$ fixed in the conventions of the introduction, which is
quasi-unipotent
by the monodromy theorem \cite[Exp.~I]{SGA7}.

When $x$ is an isolated critical point of $f$, the Milnor fiber is homotopy
equivalent to a bouquet of $\mu$ spheres of dimension $n - 1$
\cite[\S\S 6--7]{Milnor}, where
\[ \mu = \mu_f(x) := \dim_\K \cO_{U, x} \big/ \bigl( \partial_1 f, \dots,
  \partial_n f \bigr) \]
is the \bfem{Milnor number}. Its reduced cohomology is therefore concentrated
in degree $n - 1$, of rank $\mu$, and $\chi(F_x) = 1 + (-1)^{n-1}\mu$. The
Behrend function of $\mathrm{Crit}(f)$ at $x$ is then
\[ \nu_{\mathrm{Crit}(f)}(x) = (-1)^{n}\bigl( 1 - \chi(F_x) \bigr) =
  (-1)^n (-1)^n \mu = \mu , \]
the Milnor number. Two further facts concern the monodromy. The Lefschetz
number of $T_x$ on $H^\bullet(F_x)$ vanishes whenever $x$ is a singular point
of the fiber $f^{-1}(0)$, that is whenever $f(x) = 0$ and $df(x) = 0$
\cite{ACampo}, and the \bfem{monodromy zeta function}
\[ \zeta_{f, x}(t) := \prod_{j} \det\bigl( 1 - t\,T_x \mid
  \widetilde{H}^j(F_x) \bigr)^{(-1)^{j+1}} \]
is computed by A'Campo from an embedded resolution of $f^{-1}(0)$
\cite{ACampo}. Both are used in Section \ref{sec:dt} through the twisted
operator $\theta$, which acts on the stalks of $\mathcal{P}_Z$ as $(-1)^n$
times the inverse of $T_x$ by Proposition \ref{prop:behrend_one}.

\begin{example}[The $A_k$-singularity in one variable]\label{ex:Ak_classical}
  Let $f(x) = x^{k+1}$ on $U = \mathbb{A}^1$, $k \geq 1$. The Milnor number is
  $\mu = \dim \K[x]/(x^k) = k$, and the Milnor fiber $F_0 = \{ x^{k+1} =
  \epsilon \}$ consists of the $k + 1$ roots of $\epsilon$, so that
  $\widetilde{H}^0(F_0)$ has rank $k$, in agreement with the bouquet of $k$
  zero-spheres. The monodromy $T_0$ permutes the roots cyclically, $x \mapsto
  \zeta_{k+1} x$ for a primitive $(k+1)$-th root of unity $\zeta_{k+1}$. On
  $H^0(F_0)$ it is a permutation without fixed points, so its trace is $0$,
  and on the reduced cohomology its trace is $-1$. Its characteristic
  polynomial on $\widetilde{H}^0(F_0)$ is
  \[ \frac{x^{k+1} - 1}{x - 1} = \prod_{\zeta^{k+1} = 1,\ \zeta \neq 1} (x -
    \zeta) , \]
  and more generally, for $m \geq 1$, the trace of $T_0^m$ on
  $\widetilde{H}^0(F_0)$ is $k$ if $k + 1$ divides $m$, since $T_0^m$ is then
  the identity, and $-1$ otherwise, since $T_0^m$ is then a permutation without
  fixed points. The monodromy zeta function is
  \[ \zeta_{f, 0}(t) = \det\bigl( 1 - t\,T_0 \mid \widetilde{H}^0(F_0)
    \bigr)^{-1} = \frac{1 - t}{1 - t^{k+1}} , \]
  by the identity $\prod_{\zeta^{k+1} = 1}(1 - t\zeta) = 1 - t^{k+1}$. The
  Behrend function of $\mathrm{Crit}(f) = \{0\}$ takes the value $k$ at the
  origin. For $k = 1$ this is the non-degenerate quadratic singularity, whose
  monodromy on $\widetilde{H}^{n-1}(F_0)$ in $n$ variables is $(-1)^n$ by the
  Picard--Lefschetz formula \cite[Exp.~XV]{SGA7II}, the case $n = 1$ being the
  transposition of two points.
\end{example}

Adding a non-degenerate quadratic form in further variables does not change
the singularity type. The $A_k$-singularity in $n$ variables is $x_1^{k+1} +
x_2^2 + \cdots + x_n^2$, its Milnor number is again $k$, and its Milnor fiber
and monodromy are obtained from those of $x_1^{k+1}$ by the Thom--Sebastiani
theorem, Theorem \ref{thm:TS}. A function whose critical locus is
positive-dimensional has \bfem{transversal type $A_k$} along a stratum $S$
of its critical locus if its restriction to a transversal slice at each point
of $S$ is an $A_k$-singularity in the transversal variables. The invariants of
Section \ref{sec:applications} are computed stratum by stratum from the
transversal type.

\subsection{Monodromic Complexes and Tame Monodromy}\label{subsec:monodromic}

Over $\K$ the \'etale fundamental group $\pi_1^{\mathrm{et}}(\Gm, 1) \simeq
\hat{\mathbb{Z}}(1)$ is topologically generated by the element $\gamma$ fixed
in the conventions, and every local system on $\Gm$ is tame.

For $m \geq 1$ and
a character $\chi \colon \mu_m(\K) \to \F^\times$, the \bfem{Kummer local
system} $\mathscr{K}_\chi$ is the rank-one summand of $[m]_* \F$ on which
$\mu_m$ acts through $\chi$, where $[m] \colon \Gm \to \Gm$ is the $m$-th power
map. Kummer local systems are the rank-one local systems with finite monodromy,
and $H^\bullet_{\mathrm{et}}(\Gm, \mathscr{K}_\chi) = 0$ for $\chi$ nontrivial.
We write $\mathscr{K} := \mathscr{K}_{\chi_2}$ for the quadratic Kummer system.
Its monodromy is $-\mathrm{id}$ and $\mathscr{K}^{\otimes 2}$ is canonically
trivial.

\begin{definition}[\cite{Verdier}]\label{def:monodromic}
  Let $Y$ be an Artin stack locally of finite type over $\K$ and let
  $q \colon \widetilde{Y} \to Y$ be a principal $\Gm$-bundle. A complex
  $N \in D^b_c(\widetilde{Y}_{\mathrm{et}}, \F)$ is \bfem{monodromic} if its
  cohomology sheaves are locally constant with tame monodromy along
  the fibers of
  $q$. For the
  trivial bundle
  $\widetilde{Y} = Y \times \Gm$ this is the condition that the
  cohomology sheaves
  be locally constant along the $\Gm$-factor, tameness being
  automatic over $\K$.
  
  We write $\mathrm{Perv}_{\mathrm{mon}}(\widetilde{Y}_{\mathrm{et}}, \F)$
  for the full subcategory of monodromic perverse sheaves. 
\end{definition}

The following theorem collects well-known facts in related literature about monodromic complexes.
\vspace{2em}

\begin{theorem}\label{thm:monodromy_facts}
  \begin{enumerate}
  \vspace{0em}
  \setlength{\itemsep}{0.1in}
    \item[]
    \item The monodromy of \'etale vanishing cycles is quasi-unipotent
      \cite[Exp. I]{SGA7}.
    \item The category $\mathrm{Perv}_{\mathrm{mon}}((Y \times
      \Gm)_{et}, \F)$ is
      equivalent to the category of pairs $(P, T)$ of a perverse
      sheaf on $Y$ with an
      automorphism, and under the equivalence the pushforward along
      the projection
      corresponds to $\mathrm{Cone}(T - \mathrm{id} \colon P \to P(-1))$, in
      the normalization in which $P$ is the shift by $[-1]$ of the restriction
      of the monodromic sheaf to a fiber
      \cite{Verdier} (for Artin
      stacks via \cite{LO1, LO3, LZ}).
    \item (Picard--Lefschetz.) For an ordinary quadratic singularity in $n$
      variables, the vanishing cohomology is of rank one, concentrated in degree
      $n - 1$, and $\gamma$ acts on it by $(-1)^n$ \cite[Exp. XV]{SGA7II}.
  \end{enumerate}
\end{theorem}

Descent for monodromic complexes along $\Gm$-torsors is developed
in Section \ref{sec:perverse} (Definition \ref{def:slicepair}, Lemma
\ref{lem:monodromic_descent}).

\subsection{Local Models of Lagrangians}\label{subsec:lagmodels}

Joyce and Safronov prove a Lagrangian neighborhood theorem for Lagrangians in
$k$-shifted symplectic derived schemes in Darboux form \cite{JS}. We use the
case $k = -1$, in the following form and with the following sign convention.

\begin{theorem}[{\cite[Theorem 3.7, Example 3.6]{JS}}]\label{thm:JS_local}
  Let $U$ be a smooth affine $\K$-scheme, $s \in \cO(U)$, and let
  $\varphi \colon L \to \dCrit(s)$ be a Lagrangian. \'Etale-locally on $L$ there
  exist a smooth morphism $\Psi \colon V \to U$ of smooth $\K$-schemes, a vector
  bundle $E \to V$ with a non-degenerate quadratic form $q$, and a section
  $\sigma \in \Gamma(V, E)$ with
  \[ q(\sigma) = -\, s \circ \Psi, \]
  such that $L$ is equivalent to the derived zero locus $Z(\sigma) \subseteq V$
  compatibly with $\varphi$ and the Lagrangian structure. The cotangent complex
  of $L$ is
  \[ \LL_L \simeq \bigl[\, T_{V/U} \longrightarrow E^\vee \longrightarrow
  T^\vee_V \,\bigr]\big|_L \]
  in degrees $-2, -1, 0$, so that
  \begin{equation}\label{eq:vdimJS}
    \vdim L = 2 \dim V - \dim U - \rk E .
  \end{equation}
  An orientation of $L$ induces a trivialization of $\det(E)|_L$
  \cite[Proposition 5.20]{AmorimBassat}.
\end{theorem}

Joyce and Safronov state the compatibility as $Q(s,s) + 4\,\pi^*(\Phi) = 0$
\cite[Example 3.6]{JS}. Rescaling $q$ by $\tfrac14$ gives the form above.

Amorim and Ben-Bassat write the relation without the sign
\cite[Proposition 5.20]{AmorimBassat}. The two conventions differ by
replacing $\phi_f$ with $\phi_{-f}$, and the canonical identification
$\phi_{-f} \simeq \phi_f$ commutes with the monodromy operator $T$, since
multiplication by $-1$ on the base acts trivially on the tame fundamental
group, so the difference is immaterial for the statements below. The identity \eqref{eq:vdimJS} is obtained by taking the
alternating sum
of the ranks in $\LL_L$. The degree $-2$ term contributes the
relative dimension
$\dim V - \dim U$ of $\Psi$, and $T^\vee_{V/U} \oplus T^\vee_U$ has rank
$\dim V$.

We call $(V, \Psi, E, q, \sigma)$ a \bfem{local presentation} of $L$. Its shape
is summarized by
\[
  \begin{tikzcd}[row sep=large, column sep=large]
    & E \arrow[d, "\pi"] & \\
    L \simeq Z(\sigma) \arrow[r, hook] & V \arrow[u, bend left=45, "\sigma"]
    \arrow[r, "\Psi"'] & U ,
  \end{tikzcd}
  \qquad q(\sigma) = - \, s \circ \Psi .
\]
Section \ref{sec:joyce} applies Theorem \ref{thm:JS_local} to the
$\Gm$-equivariant Lagrangian $\widetilde{L} \to \widetilde{X} \simeq
\dCrit(t \cdot s)$ obtained by symplectification, with all data pulled back
along the torsor.

\subsection{Contact Darboux Charts}
\label{sec:charts}

This subsection assembles the local objects the rest of the paper computes
with. Nothing here is new, and nothing here is reproved. A chart is a derived
discriminant locus inside a 1-jet space, extended in degree $-2$; the atlas
producing them is \cite[Theorem 3.7]{kib2} in the form of \cite[Theorem
3.8(1)]{Izbudak_Darboux}, and its symplectification is the trivial-group case
of \cite[Proposition 3.3]{Izbudak_Darboux}, whose equivariant refinement we do
not need. We record the statements in the form used below and cite the
companion paper for the proofs.

Let $V$ be a smooth affine $\K$-scheme with \'etale coordinates $x_1, \dots,
x_n$. The \bfem{1-jet space} $J^1(V) := T^*V \times \mathbb{A}^1$ carries the
coordinates $(x_i, y_i, z)$ and the canonical $0$-shifted contact form
\[ \alpha_{\mathrm{can}} = -d_{\DR}z + \sum_{i=1}^n y_i\, d_{\DR}x_i \]
with contact line bundle $\cO_{J^1(V)}$ \cite[Section 4.1]{kib2}. The
\bfem{1-jet prolongation} of a regular function $s \in \cO(V)$ is the section
\[ j^1 s \colon V \longrightarrow J^1(V), \qquad
x \longmapsto (x,\, d s(x),\, s(x)), \]
and $z_V \colon V \to J^1(V)$, $x \mapsto (x, 0, 0)$, denotes the zero section.
Both are Legendrian in the sense of Definition \ref{def:legendrian}, since
$(j^1 s)^* \alpha_{\mathrm{can}} = -d_{\DR}s + d_{\DR}s = 0$ and
$z_V^* \alpha_{\mathrm{can}} = 0$.

\begin{definition}[Derived discriminant locus, {\cite[Definition
  2.6]{Izbudak_Darboux}}]\label{def:discriminant}
  The \bfem{derived discriminant locus} of $s \in \cO(V)$ is the derived
  intersection of the 1-jet prolongation with the zero section,
  \[ \Delta\mathrm{loc}(s) := V \times^h_{j^1 s,\; J^1(V),\; z_V} V . \]
  By Theorem \ref{thm:review_intersection} it carries a canonical $-1$-shifted
  contact structure.
\end{definition}

\begin{lemma}[Koszul--Tate model, {\cite[Lemma 2.7]{Izbudak_Darboux}}]\label{lem:KT_model}
  Let $A$ be the quasi-free cdga generated over $\cO(V)$ by degree
  $-1$ generators
  $y_1, \dots, y_n$ and $z$, with differential
  \[ d(z) = s, \qquad d(y_i) = \frac{\partial s}{\partial x_i}, \]
  and let $\alpha := d_{\DR}z + \sum_{i=1}^n y_i\, d_{\DR}x_i \in
  \mathcal{A}^1(\Spec(A), \cO, -1)$. Then there is an equivalence
  \[ \Spec(A) \;\simeq\; \Delta\mathrm{loc}(s) \]
  of $-1$-shifted contact derived schemes under which $\alpha$
  corresponds to the
  contact form induced by $\alpha_{\mathrm{can}}$. The classical truncation of
  $\Delta\mathrm{loc}(s)$ is the closed subscheme
  $\{ s = 0 \} \cap \mathrm{Crit}(s) \subseteq V$.
\end{lemma}

The derived discriminant locus differs from the derived critical locus by the
additional equation $s = 0$, resolved by the generator $z$. Equivalently,
$\Delta\mathrm{loc}(s)$ is the derived zero locus of the $1$-jet of $s$, and $\dCrit(s)$ is the derived zero locus of $d s$. The generator $z$ carries the
contact direction, and Theorem \ref{lem:chart_sympl} identifies the
symplectification of $\Delta\mathrm{loc}(s)$ with $\dCrit(t \cdot s)$. Over a
quotient stack both sides acquire further Koszul--Tate generators in degree
$-2$. See \cite[Definition 3.1]{Izbudak_Darboux}.

\begin{definition}[Contact Darboux chart, {\cite[Definition
  3.1(i)]{Izbudak_Darboux}}]\label{def:chart}
  A \bfem{contact Darboux chart} is $\Spec(A)$ for $A$ the
  Koszul--Tate algebra of
  Lemma \ref{lem:KT_model} attached to a regular function $s$ on a smooth affine
  $\K$-scheme $U$ with \'etale coordinates $x_1, \dots, x_n$,
  extended by finitely
  many generators in degree $-2$, together with the form
  $\alpha = d_{\DR}z + \sum_i y_i\, d_{\DR}x_i$. A chart with no generators in
  degree $-2$ is $\Delta\mathrm{loc}(s)$.
\end{definition}

\begin{theorem}[Contact Darboux atlas, \cite{kib2, Izbudak_Darboux}]\label{thm:atlas}
  Every quasi-separated, locally finitely presented $-1$-shifted contact derived
  Artin stack $Z$ over $\K$ with affine stabilizers admits a smooth atlas by
  contact Darboux charts, each carrying as many generators in degree $-2$ as its
  relative dimension. If $Z$ admits an \'etale atlas, for instance if $Z$ is a
  derived scheme or a derived Deligne--Mumford stack, the charts are the contact
  Darboux schemes $\Delta\mathrm{loc}(s)$ of Lemma \ref{lem:KT_model}.

  Moreover, let $\Spec(A)$ be a chart with base $U$ of dimension $n$ and
  potential $s$, write $\widetilde{A} := A[t, t^{-1}]$ with $\Gm$ scaling $t$
  with weight $1$, and set
  \[ \widetilde{y}_i := t y_i, \qquad \zeta := z . \]
  Then $\widetilde{A}$ is the standard-form Koszul--Tate model of
  $\dCrit(\widetilde{f})$ over $\widetilde{U} := U \times \Gm$, where
  $\widetilde{f}(x,t) := t \cdot s(x)$ is homogeneous of $\Gm$-weight $1$, with
  degree $0$ generators $x_i, t$ and degree $-1$ generators $\widetilde{y}_i,
  \zeta$. The Liouville form $\lambda := t\alpha$ has $\Gm$-weight $1$ and
  satisfies $d_{\DR}\lambda = \omega$ for the standard symplectic form $\omega$
  of $\dCrit(\widetilde{f})$. The generators $x_i, y_i, z$ carry $\Gm$-weight
  $0$ and $t$ carries weight $1$, so $A = \widetilde{A}^{\Gm}$, and
  $\Spec(\widetilde{A})$ is the symplectification of
  $\Spec(A)$.\label{lem:chart_sympl}
\end{theorem}

The first assertion is \cite[Theorem 3.7]{kib2} in the form of \cite[Theorem
3.8(1)]{Izbudak_Darboux}, and the second is \cite[Proposition
3.3]{Izbudak_Darboux} for the trivial group, the bookkeeping of generators,
degrees and weights being tabulated in \cite[Remark 3.4]{Izbudak_Darboux}.

\section{Contact Orientations and \texorpdfstring{$\ell$}{l}-adic
Monodromic Perverse Sheaves}
\label{sec:perverse}

Throughout this section and \S\ref{sec:joyce}, $Z$ denotes an oriented
$-1$-shifted contact derived Artin stack over $\K$, with symplectification
$p \colon \widetilde{Z} \to Z$, and charts are the contact Darboux charts of
Theorem \ref{thm:atlas}. Statements requiring $Z$ proper say so.

\subsection{Contact Orientation Data}

The Donaldson--Thomas sheaf of a $-1$-shifted symplectic stack depends on a
choice of square root of its canonical bundle. On a contact stack the square
root has to be taken upstairs, on the symplectification, and has to be compatible with the $\Gm$-action, which is what carries the contact structure. This forces the following definition.

\begin{definition}\label{def:orientation}
  We define \bfem{contact orientation data} on a $-1$-shifted
  contact derived stack $Z$ to be $\mathbb{G}_m$-equivariant
  symplectic orientation data on its symplectification
  $\widetilde{Z}$, which corresponds to a choice of a square root
  $R_Z$ of the canonical bundle $K_Z$ \cite[Section
  2.5]{Joyce_dCrit}. Following the analogous definition for
  Lagrangians \cite[Definition 5.3]{AmorimBassat}, an orientation
  for a Legendrian morphism $f \colon L \to Z$ is defined on the
  symplectification, where $\widetilde{f} \colon \widetilde{L} \to
  \widetilde{Z}$ is Lagrangian. An orientation is a $\Gm$-equivariant isomorphism of line
  bundles $K_{\widetilde{L}} \simeq \widetilde{f}^*(p^* R_Z)$ whose tensor
  square is the canonical isomorphism $(K_{\widetilde{L}})^{\otimes 2} \simeq
  \widetilde{f}^*(K_{\widetilde{Z}})$. For Legendrians of odd virtual dimension
  this notion requires a monodromic refinement, given in Definition
  \ref{def:graded_orientation}.
\end{definition}

The definition is made upstairs. The corresponding statement downstairs is twisted. Taking determinants in the Legendrian fiber sequence of
\cite[Definition 2.39]{IzbudakBerktav2026} we obtain
\[ (K_L)^{\otimes 2} \simeq f^*\bigl( K_Z \otimes
  \cL^{\otimes -(\vdim L + 1)} \bigr) . \]
Note that the twist disappears on $\widetilde{Z}$, where $p^*\cL$ is trivialized by
the tautological section, which has weight $1$. This is the same weight
discrepancy computed in Proposition \ref{prop:sign_character}.

The exponent $\vdim L + 1$ is the source of the parity obstruction of Section
\ref{sec:joyce}. It reappears as the residual twist of Step 2 of the proof of
Proposition \ref{prop:local_verification}, and forces the grading of Definition
\ref{def:graded_orientation}.

Contact orientation data need not exist. The equivariance requirement is not an
extra condition on the square root, and the obstruction is the same one that
governs orientability downstairs.

\begin{proposition}[\'Etale Obstruction to
  Orientability]\label{prop:orientation_obstruction}
  The projection $p \colon \widetilde{Z} \to Z$ is a principal
  $\mathbb{G}_m$-bundle associated to the contact line bundle $\cL$
  \cite[Definition 4.3]{kib1}. The topological obstruction to
  orientability resides in the \'etale cohomology group
  $H^2_{\mathrm{et}}(Z, \mu_2)$, the \'etale analogue of the
  vanishing of a second Stiefel--Whitney class.
\end{proposition}

\begin{proof}
  The relative cotangent sequence for $p \colon \widetilde{Z} \to Z$
  gives the exact triangle \[ p^*\LL_Z \to \LL_{\widetilde{Z}} \to
  \LL_{\widetilde{Z}/Z} \to p^*\LL_Z[1].\] As the fiber is
  $\mathbb{G}_m$, the relative cotangent complex
  $\LL_{\widetilde{Z}/Z}$ is equivalent to the rank $1$ trivial
  bundle $\cO_{\widetilde{Z}}$. The exact triangle implies that the
  determinant line bundles satisfy \[\det(\LL_{\widetilde{Z}}) \simeq
  p^*(\det(\LL_Z)) \otimes \det(\LL_{\widetilde{Z}/Z}) \simeq
  p^*(\det(\LL_Z)) \otimes \cO_{\widetilde{Z}} \simeq
  p^*(\det(\LL_Z))\]. A $\mathbb{G}_m$-equivariant square root of
  $\det(\LL_{\widetilde{Z}})$ is thus given by a square root of the
  line bundle $\det(\LL_Z)$ over the base $Z$. As $p$ is a principal
  $\mathbb{G}_m$-bundle, the pullback $p^*$ is an equivalence between line
  bundles on $Z$ and $\mathbb{G}_m$-equivariant line bundles on
  $\widetilde{Z}$, so the topological obstruction to orientability
  resides in the \'etale cohomology group $H^2_{\mathrm{et}}(Z,
  \mu_2)$ via the boundary map of the Kummer exact sequence $1 \to
  \mu_2 \to \mathbb{G}_m \xrightarrow{2} \mathbb{G}_m \to 1$.
\end{proof}

In the situation of the contact Darboux atlas the obstruction vanishes on each
chart, and there is a preferred choice.

\begin{remark}\label{rem:canonical_orientation}
  A derived discriminant locus $\Delta\mathrm{loc}(s)$ attached to a regular
  function $s$ on a smooth $\K$-scheme $V$ of dimension $N$ is canonically
  oriented. By Lemma \ref{lem:KT_model} its cotangent complex has
  $\Omega^1_V$ in
  degree $0$ and, in degree $-1$, the rank $N+1$ module spanned by
  $d y_1, \dots, d y_N, dz$, namely $T_V \oplus \cO_V$. Hence
  \[ \det\bigl(\LL_{\Delta\mathrm{loc}(s)}\bigr)
    = \det(\Omega^1_V) \otimes \det(T_V \oplus \cO_V)^{-1}
  = \omega_V \otimes \omega_V \]
  restricted to $\Delta\mathrm{loc}(s)$, and $R :=
  \omega_V|_{\Delta\mathrm{loc}(s)}$
  is a square root. The two determinants cancel because the contact
  line is trivial
  here, $s$ being a function rather than a section. For a nontrivial
  contact line
  the obstruction of Proposition \ref{prop:orientation_obstruction} is not
  formally resolved in this way.
\end{remark}

\subsection[Slice Pairs and Descent Along a Gm-Torsor]{Slice Pairs
and Descent Along a \texorpdfstring{$\Gm$}{Gm}-Torsor}

An orientation of $\widetilde{Z}$ produces a perverse sheaf upstairs, and the
question is what of it survives on $Z$. The projection $p$ is a $\Gm$-torsor, so
a complex on $\widetilde{Z}$ that is locally constant along the orbits is
determined by its restriction to a slice together with the monodromy around the
orbit. We fix that language here and record the descent statements the rest of
the paper uses.

\begin{definition}[Slice pairs and the monodromic sign
  twist]\label{def:slicepair}
  Let $p \colon \widetilde{Y} \to Y$ be a trivialized principal $\Gm$-bundle,
  $\widetilde{Y} \simeq Y \times \Gm$, with unit section $u \colon Y
  \to \widetilde{Y}$,
  and let $N \in D^b_c(\widetilde{Y}_{\mathrm{et}}, \F)$ be
  monodromic in the sense
  of Definition \ref{def:monodromic}. The \bfem{slice pair} of $N$ is
  \[ (B_N, T_N), \qquad B_N := u^* N[-1], \]
  where $T_N \in \mathrm{Aut}(B_N)$ is the automorphism induced by the fiberwise
  monodromy. A different section of $p$ differs from $u$ by a translation, which
  changes this identification by a power of $T_N$. Any construction
  invariant under
  $T_N$ is therefore independent of the trivialization.

  Let $\mathscr{K}$ denote the Kummer local system on $\Gm$ attached
  to the quadratic
  character, and set $\Lambda := \mathrm{pr}_2^* \mathscr{K}$ on
  $\widetilde{Y}$, the
  \bfem{monodromic sign twist}. For monodromic $N$, the complex $N
  \otimes \Lambda$ is
  monodromic with slice pair $(B_N, -T_N)$, and $\Lambda^{\otimes 2}$
  is canonically
  trivial.
\end{definition}

\begin{remark}\label{rem:pullback_slice}
  For $A \in D^b_c(Y_{\mathrm{et}}, \F)$, the pullback $p^*A[1]$ is monodromic
  with slice pair $(A, \mathrm{id})$. Its restriction to a fiber is the constant
  complex with value $A_y[1]$, so the cohomology sheaves are constant along the
  fibers and the monodromy automorphism is the identity, while the slice is
  $u^*p^*A[1][-1] = A$. In particular, a monodromic complex whose monodromy
  automorphism is nontrivial is not of the form $p^*A[1]$.
\end{remark}

What a monodromic complex records beyond its slice is the monodromy
automorphism, and the pushforward along $p$ sees only the part of it where that
automorphism is trivial. The following lemma makes this precise and provides
the descent and base change statements used throughout.

\begin{lemma}[Monodromic sheaves on a $\Gm$-torsor]\label{lem:monodromic_descent}
  Let $p \colon \widetilde{Y} \to Y$ be as in Definition
  \ref{def:slicepair} and let
  $N$ be monodromic with slice pair $(B, T)$.
  \begin{enumerate}
    \item There is a canonical exact triangle
      \[ B \xrightarrow{\;T - \mathrm{id}\;} B(-1) \longrightarrow Rp_* N
      \xrightarrow{\;+1\;}. \]
    \item For every $A \in D^b_c(Y_{\mathrm{et}}, \F)$ there is a
      short exact sequence,
      functorial in $A$ and $N$,
      \[ 0 \to \Hom_{Y}(A, B[-1](-1))_{T} \to
        \Hom_{\widetilde{Y}}(p^*A[1], N) \xrightarrow{\;\mathrm{sl}\;}
      \Hom_{Y}(A, B)^{T} \to 0, \]
      where $(-)^T$ and $(-)_T$ denote invariants and coinvariants of
      post-composition
      with $T$, and the \bfem{slice map} $\mathrm{sl}$ is induced by
      the morphism
      $Rp_*N \to B[1]$ of the triangle in (1). The image of
      $\mathrm{sl}$ is independent
      of the trivialization of the torsor.
  
    \item For a morphism $g \colon W \to Y$ of finite type with induced Cartesian
      morphism of torsors $\widetilde{g} \colon \widetilde{W} \to \widetilde{Y}$,
      the complex $\widetilde{g}^! N$ is monodromic with slice pair
      $(g^! B, g^! T)$, and $\mathrm{sl}$ intertwines $\widetilde{g}^!$ with
      $g^!$.
  \end{enumerate}
\end{lemma}

\begin{proof}~
\begin{enumerate}
    \item For a tame monodromic complex $M$ on $\Gm$ over $\K$ there
      is a canonical
      identification $R\Gamma((\mathbb{G}_{m})_{\mathrm{et}}, M)
      \simeq [\, M_1 \xrightarrow{T -\mathrm{id}} M_1(-1) \,]$ in
      degrees $0$ and $1$, computing the cohomology of the
      tame fundamental group $\pi_1^{\mathrm{et}}(\Gm) \simeq
      \hat{\mathbb{Z}}(1)$
      \cite{Verdier}. Applying this fiberwise to $N$, whose
      restriction to the fiber over
      $y$ is $B_y[1]$ with monodromy $T$, gives
      $(Rp_*N)_y \simeq \mathrm{Cone}(T - \mathrm{id} \colon B_y \to
      B_y(-1))$, naturally
      in $y$. The triangle follows. For $N = p^*A[1]$ the monodromy is the
      identity by Remark \ref{rem:pullback_slice}, so $T - \mathrm{id}$ vanishes
      and the triangle splits as
      \[ Rp_* p^* A[1] \simeq A[1] \oplus A(-1) , \]
      the two summands being the degree $0$ and degree $1$ cohomology of the
      $\Gm$-fiber.

    \item Apply $\Hom_Y(A[1], -)$ to the triangle in (1) and use the adjunction
      $\Hom_{\widetilde{Y}}(p^*A[1], N) \simeq \Hom_Y(A[1], Rp_*N)$.
      The resulting long
      exact sequence yields the stated short exact sequence. A change
      of trivialization
      modifies the slice identification of $(B, T)$ by a power of $T$. This acts
      trivially on $\Hom_Y(A, B)^T$ and on the $T$-coinvariants, so
      the sequence and the
      image of $\mathrm{sl}$ are independent of the choice.
  
    \item Write $\widetilde{g} := g \times_Y \widetilde{Y}$, so that
      \[
        \begin{tikzcd}[row sep=large, column sep=large]
          \widetilde{W} \arrow[r, "\widetilde{g}"] \arrow[d, "p_W"'] &
          \widetilde{Y} \arrow[d, "p"] \\
          W \arrow[r, "g"'] \arrow[u, bend left=45, "u_W"] & Y \arrow[u,
          bend right=45, "u_Y"'] .
        \end{tikzcd}
      \]
      For monodromicity, work locally on $Y$. After pullback along a finite Kummer
  covering $[k] \colon \Gm \to \Gm$ trivializing the finite-order part of the
  monodromy, a tame monodromic complex is an iterated extension of
  external products
  $M \boxtimes \mathscr{K}_\chi$ with $\mathscr{K}_\chi$ a Kummer
  local system, and
  \[(g \times \mathrm{id})^!(M \boxtimes \mathscr{K}_\chi) \simeq g^! M \boxtimes
  \mathscr{K}_\chi,\] so $\widetilde{g}^! N$ is monodromic. For the
  slice pair,
  note that $u_Y$ is a section of a smooth morphism of relative
  dimension $1$ and
  $N$ has locally constant cohomology sheaves in the fiber direction, so
  $u_Y^! N \simeq u_Y^* N[-2](-1)$, and similarly for $u_W$. Then
  \[ u_W^* \widetilde{g}^! N[-1] \simeq u_W^! \widetilde{g}^! N[1](1)
  = g^! u_Y^! N[1](1) \simeq g^!(u_Y^* N[-1]) = g^! B, \]
  using $\widetilde{g} \circ u_W = u_Y \circ g$ and the composition
  of exceptional
  pullbacks. The identification carries the monodromy automorphism to
  $g^! T$, and
  the compatibility of $\mathrm{sl}$ follows from the functoriality
  of the triangle in part (1) under $g^!$.
  \end{enumerate}
\end{proof}

\subsection{The Twisted Monodromy Operator}

The monodromy of the vanishing cycles of a chart does not glue across the atlas. A change of chart alters the dimension of the base. Twisting by the
parity of that dimension repairs this, and the twisted operator is the one that
descends.

\begin{definition}[Twisted monodromy operator of a chart]\label{def:twisted}
  Let $(U, s)$
  be a contact Darboux chart of Theorem \ref{thm:atlas}, with smooth base $U$ of
  dimension $n$, potential $s \in \cO(U)$ and symplectified chart
  $\widetilde{U} = U \times \Gm$ carrying the homogeneous potential
  $\widetilde{f} = t \cdot s$. Write
  $\mathcal{PV}_{U, \widetilde{f}} := \phi^p_{\widetilde{f}}(\F[\dim
  \widetilde{U}])$ for the
  perverse vanishing cycles local model, and $T_{U, \widetilde{f}}$ for its
  geometric monodromy along the $\Gm$-orbits. The \bfem{twisted monodromy
  operator} of the chart is
  \[ \theta_{U, \widetilde{f}} := (-1)^{n}\, T_{U, \widetilde{f}}
  = (-1)^{\dim \widetilde{U} - 1}\, T_{U, \widetilde{f}}. \]
  We let $\theta_{U, \widetilde{f}}$ act as the identity on the orientation
  $\mu_2$-local systems $Q_{U, \widetilde{f}}$ of Theorem \ref{thm:BBDJS_sheaf}, and on the twisted local model
  $\mathcal{PV}_{U, \widetilde{f}} \otimes_{\mu_2} Q_{U, \widetilde{f}}$ as
  $\theta_{U, \widetilde{f}} \otimes \mathrm{id}$.
\end{definition}

Brav, Bussi, Dupont, Joyce and Szendr\H{o}i write $M_{U, f}$ for the geometric
monodromy operator on $\mathcal{PV}_{U, f}$ and define the twisted monodromy
operator $\tau_{U, f} = (-1)^{\dim U} M_{U, f}$ \cite[\S 2.4]{BBDJS}. The
corresponding operator on a d-critical stack is the second isomorphism of
Theorem \ref{thm:BBDJS_sheaf}. In this paper $T$ denotes that operator, transported
to $\mathcal{P}_{\widetilde{Z}}$.

The symplectified chart has dimension
$\dim \widetilde{U} = n + 1$, so $\tau_{\widetilde{U}, \widetilde{f}} =
(-1)^{n+1} M$ while $\theta_{U, \widetilde{f}} = (-1)^{n} M$, and the two
differ by a sign. That sign is computed intrinsically in Proposition
\ref{prop:sign_character}, where it appears as the character
$\varepsilon_U = (-1)^{n+1}$ of the $\Gm$-action on the orientation, so the
relation $T = -\theta$ arises twice by independent routes.

\begin{proposition}[Gluing of the twisted operator]\label{prop:theta_glues}
  The chartwise operators $\theta_{U, \widetilde{f}}$ commute with
  the comparison
  isomorphisms entering the construction of $\mathcal{P}_{\widetilde{Z}}$ in
  Theorem \ref{thm:BBDJS_sheaf}, and therefore glue to a canonical tame
  automorphism $\theta$ of $\mathcal{P}_{\widetilde{Z}}$, compatible with the
  monodromic structure and independent of all choices of charts. Under the slice
  correspondence of Lemma \ref{lem:monodromic_descent} it induces a
  canonical tame
  automorphism $\theta$ of $\mathcal{P}_Z$ commuting with the
  descended monodromy
  automorphism.
\end{proposition}

\begin{proof}
  The comparison isomorphisms between the twisted local models are generated by
  two moves, described in Theorem \ref{thm:BBDJS_sheaf} and \cite[Section 4]{BenBassat2015}. These are \'etale morphisms of
  charts of equal base dimension, and equivariant stabilizations
  $(U', s') = (U \times \mathbb{A}^r, s \boxplus q)$ with $q$ a
  non-degenerate quadratic form of rank $r$, with homogeneous potentials
  $\widetilde{f} = t \cdot s$ and $\widetilde{f}' = \widetilde{f} + t \cdot q$
  \cite[Section 4]{BenBassat2015}.

  For an \'etale morphism of charts, vanishing cycles commute with \'etale
  pullback and the orbit monodromy is functorial, so the comparison intertwines
  the operators $T$. The base dimensions agree, so it intertwines the operators
  $\theta$.

  For a stabilization, the potential $t \cdot q$ is a quadratic form in the new
  variables with invertible coefficient matrix $t \cdot (q_{ij})$
  over $\widetilde{U}$, and
  the stabilization isomorphism $\Theta_\Phi$ of \cite[Theorem 5.4]{BBDJS},
  recalled in Example \ref{ex:quadratic}, identifies
  $\mathcal{PV}_{U', \widetilde{f}'}$ with
  $\mathcal{PV}_{U, \widetilde{f}}$ twisted by the $\mu_2$-local
  system of square
  roots of $\det(t \cdot q_{ij}) = t^r \det(q_{ij})$. Along a $\Gm$-orbit this
  local system has monodromy $(-1)^r$. Equivalently, the transversal
  Milnor fiber
  of $t \cdot q$ is the affine quadric $\{q = \epsilon/t\}$, on which
  the $t$-loop
  acts through the antipodal map of degree $(-1)^r$, as in the computation of
  Section \ref{sec:dt}. It follows that the fiberwise monodromies satisfy
  $T_{U', \widetilde{f}'} = (-1)^r \, T_{U, \widetilde{f}}$ under the
  stabilization isomorphism, and
  \[ \theta_{U', \widetilde{f}'} = (-1)^{n + r}\, T_{U', \widetilde{f}'}
    = (-1)^{n + r} (-1)^r\, T_{U, \widetilde{f}}
  = (-1)^{n}\, T_{U, \widetilde{f}} = \theta_{U, \widetilde{f}}. \]
  Note that the orientation factors $Q$ are $\mu_2$-local systems, the comparisons are
  $\mu_2$-equivariant, and $\theta$ acts trivially on them, so they
  do not interact.

  Perverse sheaves on the lisse-\'etale site form a stack
  \cite[Proposition 7.1]{LO3} and their automorphisms form a sheaf,
  so an automorphism commuting with the descent data
  of $\mathcal{P}_{\widetilde{Z}}$ glues. Tameness holds chartwise by
  quasi-unipotence \cite{SGA7}. Finally, $\theta$ commutes with the fiberwise
  monodromy of $\mathcal{P}_{\widetilde{Z}}$ (chartwise it is a scalar multiple of it), and induces an automorphism of the slice pair $(\mathcal{P}_Z, T)$
  commuting with $T$.
\end{proof}

The twisted operator exists globally. It remains to compare it with
the geometric monodromy, which also acts on $\mathcal{P}_{\widetilde{Z}}$ and
does not agree with it.

\begin{proposition}[Comparison with the geometric monodromy]
  \label{prop:sign_character}
  Let $T$ be the geometric orbit monodromy automorphism of
  $\mathcal{P}_{\widetilde{Z}}$. There is a unique locally constant function
  $\varpi$ from the reduced support of $\mathcal{P}_{\widetilde{Z}}$ to $\mu_2$,
  the \bfem{orientation sign character}, with
  \[ T = \varpi \cdot \theta. \]
  In a contact Darboux chart with base dimension $n$ one has
  $\varpi = (-1)^n \varepsilon_U$, where $\varepsilon_U \in \{\pm 1\}$ is the
  monodromy of the orientation local system $Q_{U, \widetilde{f}}$ along the
  $\Gm$-orbits. In particular $\varpi$ depends only on the contact orientation
  data, and the same relation holds for the descended automorphisms of
  $\mathcal{P}_Z$.

  For contact orientation data in the sense of Definition \ref{def:orientation}
  one has $\varepsilon_U = (-1)^{n+1}$ in every contact Darboux chart, and
  therefore
  \[ \varpi \equiv -1, \qquad \text{that is} \qquad T = - \theta \]
  globally on the support of $\mathcal{P}_{\widetilde{Z}}$.
\end{proposition}

\begin{proof}
  In a chart, $\mathcal{P}_{\widetilde{Z}}$ is identified with
  $\mathcal{PV}_{U, \widetilde{f}} \otimes_{\mu_2} Q_{U,
  \widetilde{f}}$, and the
  fiberwise monodromy of a tensor product is the tensor product of the fiberwise
  monodromies. Hence $T = T_{U, \widetilde{f}} \otimes \varepsilon_U =
  (-1)^n \varepsilon_U \cdot \theta$ on the chart. The ratio of the two global
  automorphisms $T$ and $\theta$ is therefore, chartwise, multiplication by the
  locally constant sign $(-1)^n \varepsilon_U$. Globality of $T$ and
  $\theta$ makes
  this sign independent of the chart, and uniqueness holds on the support.

  It remains to compute $\varepsilon_U$ for contact orientation data.
  Work in the
  coordinates of Theorem \ref{lem:chart_sympl}, where
  $\widetilde{U} = U \times \Gm$ carries the homogeneous potential
  $\widetilde{f} = t \cdot s$ and $\widetilde{X} \simeq
  \dCrit(\widetilde{f})$ has
  Koszul--Tate generators $x_1, \dots, x_n, t$ in degree $0$ and
  $\widetilde{y}_1, \dots, \widetilde{y}_n, \zeta$ in degree $-1$, with weights
  \[ \mathrm{wt}(x_i) = 0, \quad \mathrm{wt}(t) = 1, \quad
    \mathrm{wt}(\widetilde{y}_i) = \mathrm{wt}(t y_i) = 1, \quad
  \mathrm{wt}(\zeta) = \mathrm{wt}(z) = 0 . \]
  By Definition \ref{def:orientation}, contact orientation data is
  $\Gm$-equivariant of weight $0$, so we must compare the Darboux
  generators with
  weight-$0$ ones. In degree $0$ the Darboux generator of the
  determinant line is
  $d x_1 \wedge \cdots \wedge d x_n \wedge d t$, while the weight-$0$
  generator is
  $d x_1 \wedge \cdots \wedge d x_n \wedge d t / t$. The two differ
  by the factor
  $t$. In degree $-1$ the Darboux generator is $d \widetilde{y}_1 \wedge \cdots
  \wedge d \widetilde{y}_n \wedge d \zeta$ and the weight-$0$ generator is
  $d y_1 \wedge \cdots \wedge d y_n \wedge d z$. As
  $d \widetilde{y}_i \equiv t \, d y_i$ modulo terms of degree $0$, which do not
  affect the determinant of the graded pieces, the two differ by $t^{n}$. Taking
  the alternating combination,
  \[ \det \LL_{\widetilde{X}} = \det(\text{degree } 0) \otimes
  \det(\text{degree } -1)^{-1} \]
  shows that the Darboux and weight-$0$ generators of
  $K_{\widetilde{X}}$ differ by
  $t \cdot t^{-n} = t^{1-n}$. On the other side,
  $K_{\widetilde{U}}^{\otimes 2}$ has
  Darboux generator $(d x \wedge d t)^{\otimes 2} = t^2 (d x \wedge d
  t/t)^{\otimes 2}$,
  so its Darboux and weight-$0$ generators differ by $t^{2}$. The canonical
  isomorphism $K_{\widetilde{X}} \simeq K_{\widetilde{U}}^{\otimes 2}$ of
  Theorem \ref{thm:BBDJS_sheaf} carries one Darboux generator to the other, and in the weight-$0$ trivializations it is multiplication by
  \[ t^{2} \cdot t^{-(1-n)} = t^{\, n+1} . \]
  Therefore $Q_{U, \widetilde{f}}$ is the $\mu_2$-local system of
  square roots of
  $t^{\,n+1}$ up to a weight-$0$ unit, whose monodromy along the $\Gm$-orbits is
  $(-1)^{n+1}$. Thus $\varepsilon_U = (-1)^{n+1}$ and
  $\varpi = (-1)^n (-1)^{n+1} = -1$.

  This value is consistent with the stabilizations of Proposition
  \ref{prop:theta_glues}. A stabilization of rank $k$ replaces $n$ by $n+k$ and $\varepsilon_U$ by $(-1)^{n+k+1} = (-1)^k \varepsilon_U$, matching the factor
  $(-1)^k$ by which the chartwise monodromy $T_{U, \widetilde{f}}$
  changes, so that
  the product $T = T_{U, \widetilde{f}} \otimes \varepsilon_U$ is unchanged.
\end{proof}

The two operators carry the same information on a chart and differ by
a global sign. Only one of them survives the pushforward, which is what the
following theorem records.

\begin{theorem} \label{thm:monodromic_sheaf}
  The monodromic
  \'etale perverse sheaf $\mathcal{P}_{\widetilde{Z}}$ descends to an
  $\ell$-adic
  perverse sheaf $\mathcal{P}_Z$ on $Z_{et}$ equipped
  with two commuting tame automorphisms, the geometric monodromy $T$ and the
  twisted monodromy operator $\theta$ of Proposition
  \ref{prop:theta_glues}, related
  by the orientation sign character, $T = \varpi \cdot \theta$ (Proposition
  \ref{prop:sign_character}). In every contact Darboux chart with base
  dimension $n$, the operator $\theta$ acts as $(-1)^n$ times the
  vanishing-cycle
  monodromy of the homogeneous potential, and $T = -\theta$ by Proposition
  \ref{prop:sign_character}. The operator $\theta$ is the one that acts on
  fundamental classes of Legendrians, after the grading correction of
  \S\ref{subsec:gradedjoyce}.
\end{theorem}

\begin{proof}
  \vspace{2pt}
  \textbf{\textit{Step 1: The sheaf upstairs is monodromic.}}
  The $\Gm$-action on $\widetilde{Z}$ equips the $\ell$-adic BBDJS perverse
  sheaf $\mathcal{P}_{\widetilde{Z}}$ of Theorem \ref{thm:BBDJS_sheaf} with the
  structure of a monodromic perverse sheaf along $p \colon \widetilde{Z} \to Z$
  in the sense of Definition \ref{def:monodromic}.

In a local Darboux chart over $\K$, the \'etale vanishing cycles
  functor $\phi_{\widetilde{f}}$ \cite[Section 2.3]{BBDJS} is applied
  to the homogeneous potential $\widetilde{f}(x,t) = t \cdot s(x)$
  evaluated along the special fiber $\widetilde{f} = 0$. The
  potential $\widetilde{f}$ is homogeneous of weight $1$ with respect
  to the $\mathbb{G}_m$-action. The cohomology sheaves of the
  vanishing cycles are locally constant along the
  $\mathbb{G}_m$-orbits, and the canonical monodromy automorphism $T$
  of the \'etale vanishing cycles corresponds to the geometric
  monodromy along the $\mathbb{G}_m$-orbits \cite[Section 2.4]{BBDJS}.

  By the $\ell$-adic monodromy theorem \cite[Exp.~I]{SGA7}, the
  monodromy $T$ is quasi-unipotent, and tame automatically since
  $\mathrm{char}\, \K = 0$. This operator corresponds to the
  topological generator of the \'etale fundamental group
  $\pi_1^{\mathrm{et}}(\mathbb{G}_m) \simeq \hat{\mathbb{Z}}(1)$. The
  complex $\mathcal{P}_{\widetilde{Z}}$ restricts to a locally
  constant system along the orbits, twisted by the tame monodromy
  $T$. Consequently, $\mathcal{P}_{\widetilde{Z}}$ is monodromic.

  \vspace{5pt}
  \textbf{\textit{Step 2: The pushforward loses the twist.}}
  The derived pushforward $Rp_* \mathcal{P}_{\widetilde{Z}}$ computes
  the cohomology of the $\mathbb{G}_m$-fibers. For a monodromic
  perverse sheaf, this pushforward is given locally over a
  trivializing chart $U \to Z$ by the complex $[\mathcal{P}_U
  \xrightarrow{T - \mathrm{id}} \mathcal{P}_U(-1)]$ in degrees $0$ and $1$,
  that is by the mapping cone $\mathrm{Cone}(T -
  \mathrm{id} \colon \mathcal{P}_U \to \mathcal{P}_U(-1))$, as in Lemma
  \ref{lem:monodromic_descent}(1). When $1$ is not an eigenvalue of the monodromy
  operator $T$ acting on the stalks, this mapping cone is acyclic,
  and the pushforward vanishes. The derived pushforward along the
  $\mathbb{G}_m$-orbits retains only the generalized $1$-eigenspace
  of the monodromy.

  Verdier's specialization equivalence for monodromic perverse sheaves
  \cite{Verdier}, recorded in Theorem \ref{thm:monodromy_facts}(2), extends to
  Artin stacks through the lisse-\'etale theory of Laszlo--Olsson \cite[Theorem 2.2.3 and Example
2.2.6]{LO1} and the $\infty$-categorical \'etale descent of Liu--Zheng
  \cite[Sections 4, 5, 7, and 8]{LZ}.

  Because $\widetilde{Z} \to Z$ is a $\mathbb{G}_m$-torsor, Verdier's
  equivalence applies locally over a trivializing chart $U \to Z$
  (where $\widetilde{U} \simeq U \times \mathbb{G}_m$), yielding the
  commutative diagram:
  \begin{center}
    \begin{tikzcd}[row sep=large, column sep=large]
      \mathrm{Perv}_{\mathrm{mon}}(\widetilde{U}_{\mathrm{et}},
      \mathbb{F}) \arrow[r, "\sim"] \arrow[d, "Rp_*"'] &
      \mathrm{Perv}(U_{\mathrm{et}}, T) \arrow[d, "{\mathrm{Cone}(T -
      \mathrm{id})}"] \\
      D^b_c(U_{\mathrm{et}}, \mathbb{F}) \arrow[r, equal] &
      D^b_c(U_{\mathrm{et}}, \mathbb{F})
    \end{tikzcd}
  \end{center}

  The abelian category of monodromic $\ell$-adic perverse sheaves on
  the trivial $\mathbb{G}_m$-torsor $\widetilde{U} \to U$ is
  equivalent to $\mathrm{Perv}(U_{\mathrm{et}}, T)$, the category of
  \'etale perverse sheaves on the chart $U$ equipped with an automorphism.

  Given the monodromic perverse sheaf $\mathcal{P}_{\widetilde{Z}}$,
  Verdier's equivalence locally assigns an object
  \[ (\mathcal{P}_U, T) \in \mathrm{Perv}(U_{\mathrm{et}}, T), \]
  where $T$ is the monodromy automorphism. In the notation of
  Definition \ref{def:slicepair}, $(\mathcal{P}_U, T)$ is the slice
  pair of $\mathcal{P}_{\widetilde{Z}}$ over the chart. The BBDJS
  local gluing data for $\mathcal{P}_{\widetilde{Z}}$ in Theorem
  \ref{thm:BBDJS_sheaf} depends on the choice of orientation. The choice of
  contact orientation data on $Z$ provides $\Gm$-equivariant
  orientation data on $\widetilde{Z}$, so the gluing isomorphisms are the
  comparisons of Proposition \ref{prop:theta_glues}, which commute
  with $\theta$. The
  geometric monodromy $T$ is intrinsic to the global $\Gm$-action and
  requires no
  gluing. The local perverse sheaves $(\mathcal{P}_U, T)$
  descend and glue over $Z$, as perverse sheaves on the lisse-\'etale site form a stack for the smooth topology
  \cite[Proposition 7.1]{LO3}. The result is a global $\ell$-adic
  perverse sheaf $\mathcal{P}_Z$ with a monodromy automorphism $T$.
\end{proof}

\begin{remark}[The exact case]\label{rem:exact_case}
  The observation of this paragraph is due to P. Safronov (personal
  communication). Let $X$ be a $(-1)$-shifted symplectic derived stack with its canonical exact
  structure. Restricting that structure to the classical truncation $X^{\mathrm{cl}}$
  gives the d-critical structure underlying the symplectic form, which is how
  that d-critical structure is constructed. If $f$ denotes the potential of the
  exact structure, then $f^{-1}(0)$ carries a canonical one-form of degree $-1$,
  so the induced $(-1)$-shifted contact structure on $f^{-1}(0)$ is given by a
  global one-form. Its contact line bundle is trivial, the
  symplectification is $f^{-1}(0) \times \Gm$, and the sheaf produced by Theorem
  \ref{thm:monodromic_sheaf} on that symplectification is the
  Donaldson--Thomas sheaf of $X$ together with its internal monodromy operator,
  which is locally the monodromy of the sheaf of vanishing cycles.

  Safronov also notes that not every $(-1)$-shifted contact stack arises in
  this way. Beyond the exact
  case the construction requires the descent of $\theta$ across charts for a
  nontrivial contact line bundle, Proposition \ref{prop:theta_glues}, together
  with the twist of Definition \ref{def:orientation}, which is trivial when
  $\cL^{\otimes (\vdim L + 1)}$ is. The sign $T = -\theta$ is already visible in
  the exact case, being computed on a chart where the torsor is trivial.
\end{remark}

Six signs occur in Sections \ref{sec:perverse}--\ref{sec:joyce}, on
a chart with
base $U$ of dimension $n$ and homogeneous potential $\widetilde{f}
= t \cdot s$
on $\widetilde{U} = U \times \Gm$. We collect them, with the place
each is fixed.
\[
  \begin{array}{lll}
    T_{U, \widetilde{f}} & \text{monodromy of } \mathcal{PV}_{U,
    \widetilde{f}}
    \text{ along the } \Gm\text{-orbits}
    & \text{Definition \ref{def:twisted}} \\
    \theta_{U, \widetilde{f}} = (-1)^n T_{U, \widetilde{f}}
    & \text{its dimensional twist on the chart}
    & \text{Definition \ref{def:twisted}} \\
    \varepsilon_U = (-1)^{n+1}
    & \text{monodromy of the orientation local system } Q_{U, \widetilde{f}}
    & \text{Proposition \ref{prop:sign_character}} \\
    \varpi = (-1)^n \varepsilon_U \equiv -1
    & \text{the global ratio } T = \varpi \cdot \theta
    & \text{Proposition \ref{prop:sign_character}} \\
    \delta = n + \rk E \equiv \vdim L
    & \text{parity defect of a Legendrian presentation}
    & \text{Proposition \ref{prop:parity_formula}} \\
    \Lambda
    & \text{the quadratic Kummer twist, } (B_N, T_N) \mapsto (B_N, -T_N)
    & \text{Definition \ref{def:slicepair}}
  \end{array}
\]
Thus $T = -\theta$ globally, and $\theta \circ \mu_L = (-1)^{\delta_L} \mu_L$
with $\delta_L \equiv \vdim L$. The grading correction of
\S\ref{subsec:gradedjoyce} absorbs the second sign by twisting the orientation
datum by $\Lambda^{\otimes \vdim L}$.

\section{The Parity Defect and the Contact Joyce Conjecture}
\label{sec:joyce}

To define $\ell$-adic integration and composition functors for Legendrian
correspondences we require a contact analogue of Joyce's conjecture
\cite[Conjecture 5.18]{AmorimBassat}. We first construct the local
fundamental class of a Legendrian and compute the eigenvalue of the twisted
monodromy operator on it (\S\ref{subsec:localmodel}). This eigenvalue is
$(-1)^{\vdim L}$, and the exponent is the parity defect
(\S\ref{subsec:parity}). Both parities occur already on the $A_1$ chart, and an
odd Legendrian can carry a nonzero local class, so the resulting obstruction is
not vacuous (\S\ref{subsec:nonvacuous}). The conjecture is therefore formulated
for Legendrians carrying a graded orientation, which absorbs the sign
(\S\ref{subsec:gradedjoyce}).

\subsection{The Local Model of a Legendrian}
\label{subsec:localmodel}

On a contact Darboux chart a Legendrian
acquires a presentation in the sense of Joyce and Safronov, and the
computation of the eigenvalue takes place there.

\begin{proposition}[Local Verification for Contact
  Schemes]\label{prop:local_verification}
  Let $X \simeq \Delta\mathrm{loc}(s)$ be a $-1$-shifted contact
  derived scheme in
  Darboux form over a chart with base of dimension $n$, and let
  $\varphi \colon L
  \to X$ be a proper oriented Legendrian with local presentation $(V,
    \Psi, E, q,
  \sigma)$ as in Theorem \ref{thm:JS_local}. There exists a morphism
  $\mu_L \colon \F_L[\vdim L] \to
  \varphi^! \mathcal{P}_X$ in $D^b_c(L_{et}, \F)$,
  independent of the choices in its construction, satisfying
  \[ T \circ \mu_L = (-1)^{\rk E + n + 1}\, \mu_L \enskip \text{and}
  \quad \theta \circ \mu_L = (-1)^{n + \rk E}\, \mu_L. \enskip (\text{\S\ref{subsec:parity} for parity of exponent})\]
\end{proposition}

\begin{proof}
  The proof proceeds in four steps, carried out on the symplectification and
  descended in the last.

  \vspace{5pt}
  \textbf{\textit{Step 1: An equivariant local presentation upstairs.}}
  We claim that, \'etale-locally on $L$, the presentation $(V, \Psi, E, q,
  \sigma)$ may be chosen so that $q(\sigma) = - s \circ \Psi$ and
  $\vdim L = 2\dim V - \dim U - \rk E$, and so that its pullback along the
  torsor
  \[ \widetilde{V} = V \times \Gm, \quad
    \widetilde{\Psi} = \Psi \times \mathrm{id}, \quad
    \widetilde{E} = E \times \Gm, \quad
    \widetilde{q} = t \cdot q, \quad
  \widetilde{\sigma} = \sigma \times \mathrm{id} \]
  is a $\Gm$-equivariant local presentation of the Lagrangian
  $\widetilde{\varphi} \colon \widetilde{L} \to \widetilde{X} \simeq
  \dCrit(t \cdot s)$, with $\Gm$ acting through the torsor factor and
  $\widetilde{q}(\widetilde{\sigma}) = -\, t \cdot (s \circ \Psi)$.

Recall that by Definition \ref{def:legendrian} the symplectification
  $\widetilde{\varphi} \colon \widetilde{L} \to \widetilde{X}$ is a
  $\Gm$-equivariant Lagrangian \cite[proof of Theorem
  4.1]{IzbudakBerktav2026}. Apply Theorem
  \ref{thm:JS_local} to $\widetilde{\varphi}$ with target $\dCrit(t
  \cdot s)$. The
  choices entering its proof are a smooth morphism, a vector bundle with a
  non-degenerate quadratic form, and a section. Each may be made
  $\Gm$-equivariantly, as a rational $\Gm$-representation is the direct sum of its weight spaces, so that a
  $\Gm$-stable subspace has a $\Gm$-stable complement and an invariant element
  lifts invariantly along a surjection. Because $\Gm$ acts freely
  through the torsor factor, an
  equivariant presentation is the pullback along $\widetilde{X} \to X$ of a
  presentation $(V, \Psi, E, q, \sigma)$ of $L$ over the base,
  \[
    \begin{tikzcd}[row sep=large, column sep=large]
      \widetilde{E} \arrow[d] &
      \widetilde{V} \arrow[l, "\widetilde{\sigma}"'] \arrow[r,
      "\widetilde{\Psi}"]
      \arrow[d] \arrow[dr, phantom, "\lrcorner", very near start] &
      \widetilde{U} \arrow[d, "\Gm"] \\
      E & V \arrow[l, "\sigma"] \arrow[r, "\Psi"'] & U ,
    \end{tikzcd}
  \]
  with the quadratic form scaled by the weight of the potential. The
  potential of $\dCrit(t \cdot s)$
  is $t \cdot s$, so $\widetilde{q} = t \cdot q$ and
  \[ \widetilde{q}(\widetilde{\sigma}) = t \cdot q(\sigma)
  = -\, t \cdot (s \circ \Psi) = -\, (t\cdot s) \circ \widetilde{\Psi}, \]
  which is the defining relation of that theorem upstairs. Taking
  $\Gm$-invariants recovers $q(\sigma) = - s \circ \Psi$ downstairs, and
  $\vdim L = 2\dim V - \dim U - \rk E$.

  \vspace{5pt}
  \textbf{\textit{Step 2: Thom--Sebastiani for the homogeneous quadratic
  form.}}
  Let $\widetilde{z}$ denote the zero section of $\widetilde{E}$, let $\Lambda$
  be the monodromic sign twist of Definition \ref{def:slicepair}, and let $P_q$
  be the $\mu_2$-torsor of square roots of $\det(q_{ij})$ on $\widetilde{V}$.
  We claim that
  \begin{equation}\label{eq:TSquad}
    \phi_{t \cdot q} \F_{\widetilde{E}} \simeq
    \widetilde{z}_* \F_{\widetilde{V}}[1 - \rk E]
    \otimes_{\mu_2} P_q \otimes \Lambda^{\otimes \rk E} ,
  \end{equation}
  and that the residual monodromic factor carried by the comparison of
  \eqref{eq:TSquad} with the globally oriented sheaf
  $\mathcal{P}_{\widetilde{X}}$ is
  \begin{equation}\label{eq:residual_twist}
    \Lambda^{\otimes (\rk E + n + 1)} ,
  \end{equation}
  whose orbit monodromy is $(-1)^{\rk E + n + 1}$.

The function $t \cdot q$ is a quadratic form in the fiber variables of
  $\widetilde{E}$ with invertible coefficient matrix $t \cdot (q_{ij})$ over
  $\widetilde{V}$. By Thom--Sebastiani and the computation for non-degenerate
  quadratic forms \cite[Theorem 2.13, Example 2.14]{BBDJS}, its vanishing cycles
  are supported on the zero section and are the constant sheaf there, shifted by
  $[1 - \rk E]$ and twisted by the $\mu_2$-local system of square roots of
  \[ \det(t \cdot q_{ij}) = t^{\rk E} \det(q_{ij}). \]
  The factor $t^{\rk E}$ contributes the $\rk E$-th power of the
  quadratic Kummer
  system on the $\Gm$-factor, that is $\Lambda^{\otimes \rk E}$. Its orbit
  monodromy is $(-1)^{\rk E}$. Equivalently, the transversal Milnor fiber of
  $t \cdot q$ is the affine quadric $\{ q = \epsilon/t \}$. The loop is taken
  at fixed $\epsilon$ with $t$ traversing a circle, so that $\epsilon/t$ winds
  once around $0$, and it acts on the quadric through the antipodal map of
  degree $(-1)^{\rk E}$, as in the computation of Section \ref{sec:dt}. The remaining factor $\det(q_{ij})$ contributes the
  torsor $P_q$, which is pulled back from $\widetilde{V}$ and carries
  trivial orbit monodromy. Tate twists are suppressed by the conventions of
  Section \ref{sec:review}. This proves \eqref{eq:TSquad}.

  For \eqref{eq:residual_twist} we must compare two different
  orientations of the
  chart, and it is here that the factor $Q_{U, \widetilde{f}}$
  enters. Recall the
  two data. On the one hand, a derived critical locus carries a \bfem{canonical}
  orientation. For $\widetilde{X} = \dCrit(\widetilde{f})$ with
  $\iota \colon \widetilde{X} \to \widetilde{U}$ one has
  $K_{\widetilde{X}} \simeq \iota^* K_{\widetilde{U}}^{\otimes 2}$, so that
  $R^{\mathrm{can}} := \iota^* K_{\widetilde{U}}$ is a square root
  \cite[Section 5.1]{AmorimBassat}. On the other hand, the
  orientation datum of
  a Legendrian is taken relative to the \bfem{chosen global} square root
  $R_{\widetilde{X}}$. By Definition \ref{def:orientation}, following
  \cite[Definition 5.3]{AmorimBassat}, it is an isomorphism
  \[ \beta_{\widetilde{L}} \colon K_{\widetilde{L}} \longrightarrow
  \widetilde{\varphi}^*\bigl( R_{\widetilde{X}} \bigr) \]
  subject to $\gamma_{\widetilde{X}} \circ \beta_{\widetilde{L}}^{\otimes 2} =
  \alpha_{\widetilde{L}}$, where $\gamma_{\widetilde{X}} \colon
  R_{\widetilde{X}}^{\otimes 2} \to K_{\widetilde{X}}$ and
  $\alpha_{\widetilde{L}}$
  is the canonical isomorphism attached to the Lagrangian structure.

  The computation of \cite[Proposition 5.20]{AmorimBassat} that
  identifies the
  orientation with a trivialization of $\det(E)$ is carried out against
  $R^{\mathrm{can}}$. It uses the chart identifications
  $\det(\LL_L) \simeq \det(T^\vee_U)|_L \otimes \det(E)|_L$ and
  $\det(\LL_X) \simeq \det(T^\vee_U)^{\otimes 2}|_X$, and reads
  $\alpha_{\widetilde{L}}$ off from them. Substituting $R_{\widetilde{X}}$ for
  $R^{\mathrm{can}}$ therefore changes the conclusion by precisely
  the $\mu_2$-torsor
  of isomorphisms $R_{\widetilde{X}} \to R^{\mathrm{can}}$ squaring
  to the canonical
  one, which is the comparison torsor $Q_{U, \widetilde{f}}$ of
  Theorem \ref{thm:BBDJS_sheaf}. Hence what $\beta_{\widetilde{L}}$
  trivializes is not
  the discriminant torsor of $q$ alone but its tensor product with
  $Q_{U, \widetilde{f}}$, and the residual monodromic factor is the product of
  $\Lambda^{\otimes \rk E}$ with the orbit monodromy $(-1)^{n+1}$ of
  $Q_{U, \widetilde{f}}$ computed in Proposition
  \ref{prop:sign_character}, that is
  $\Lambda^{\otimes (\rk E + n + 1)}$.

  \vspace{5pt}
  \textbf{\textit{Step 3: The equivariant fundamental class upstairs.}} Assume now that $L$ is oriented and $\varphi$ proper. We claim there is a
  morphism
  \[ \mu_{\widetilde{L}} \colon \F_{\widetilde{L}}[\vdim \widetilde{L}]
    \longrightarrow \widetilde{\varphi}^! \mathcal{P}_{\widetilde{X}}
  \otimes \Lambda^{\otimes (\rk E + n + 1)} \]
  in $D^b_c(\widetilde{L}_{et}, \F)$, every step of whose construction is
  $\Gm$-equivariant.

The section $\sigma$ induces $\widetilde{\sigma} \colon \widetilde{V} \to
  \widetilde{E}$, which is proper as $\sigma$ is a closed immersion.
  Adjunction for \'etale sheaves gives $\F_{\widetilde{E}} \to
  \widetilde{\sigma}_* \F_{\widetilde{V}}$. We apply $\phi_{t \cdot q}$ to this
  morphism. Because $\widetilde{\sigma}$ is proper, vanishing cycles commute with
  the pushforward \cite[Theorem 2.11(ii)]{BBDJS} by proper base change. Combined with Step 2, this yields
  \[ \widetilde{z}_* \F_{\widetilde{V}}[1 - \rk E] \otimes_{\mu_2} P_q \otimes
    \Lambda^{\otimes \rk E} \longrightarrow \widetilde{\sigma}_*
  \phi_{-\, t \cdot s \circ \widetilde{\Psi}} \F_{\widetilde{V}} . \]
  We compose with the canonical identification $\phi_{-g} \simeq
  \phi_{g}$ induced
  by $\epsilon \mapsto -\epsilon$ on the base of the vanishing cycles.
  Multiplication by a unit lifts to every Kummer cover $\epsilon = \delta^m$
  and normalizes its deck group $\mu_m$, acting on it by conjugation. Because $\mu_m$ is abelian that action is trivial, so the induced automorphism of
  $\pi_1^{\mathrm{t}} \simeq \hat{\mathbb{Z}}(1)$ is the identity and the
  identification commutes with $T$. Inversion $\epsilon \mapsto \epsilon^{-1}$
  would interchange $T$ and $T^{-1}$, and is not an automorphism of the
  henselian trait. The
  inverse Thom--Sebastiani isomorphism introduces a cohomological shift of
  $[\rk E - 1]$, which, combined with the relative dimension of the smooth
  morphism $\widetilde{\Psi}$, recovers the virtual dimension shift
  $\vdim \widetilde{L}$. 
  
  Applying adjunction along the Lagrangian morphism
  $\widetilde{\varphi}$ as in \cite[Proposition 5.20]{AmorimBassat} yields
  $\mu_{\widetilde{L}}$, with target twisted by the residual factor
  \eqref{eq:residual_twist} of Step 2.

  Each step is $\Gm$-equivariant. The adjunction unit $\F_{\widetilde{E}} \to
  \widetilde{\sigma}_* \F_{\widetilde{V}}$, the commutation of
  $\phi_{t \cdot q}$
  with the proper pushforward $\widetilde{\sigma}_*$, and the smooth pullback
  along $\widetilde{\Psi}$ are all pulled back along the torsor by Step 1. The only step interacting nontrivially with the
  fiberwise monodromy is the identification \eqref{eq:TSquad}, whose monodromic
  content, after comparison with the globally oriented sheaf
  $\mathcal{P}_{\widetilde{X}}$, is recorded by the factor
  $\Lambda^{\otimes (\rk E + n + 1)}$ of \eqref{eq:residual_twist}.

  \vspace{5pt}
  \textbf{\textit{Step 4: Descent along the torsor.}} By Step 1, the symplectification
  \[\widetilde{L} \to \widetilde{X} \simeq \dCrit(t \cdot s)\] carries the
  $\Gm$-equivariant local presentation $(\widetilde{V}, \widetilde{\Psi},
  \widetilde{E}, t \cdot q, \widetilde{\sigma})$, and by Step 3 there is a morphism
  \[ \mu_{\widetilde{L}} \colon \F_{\widetilde{L}}[\vdim \widetilde{L}]
    \longrightarrow \widetilde{\varphi}^! \mathcal{P}_{\widetilde{X}}
  \otimes \Lambda^{\otimes (\rk E + n + 1)} \]
  whose construction is $\Gm$-equivariant throughout. It remains to descend
  $\mu_{\widetilde{L}}$ along the torsor.

  Write $A := \F_L[\vdim L]$. As $\F_{\widetilde{L}} \simeq p_L^* \F_L$ and
  $\vdim \widetilde{L} = \vdim L + 1$, the source of $\mu_{\widetilde{L}}$ is
  $p_L^* A[1]$, which is monodromic with slice pair $(A, \mathrm{id})$ by Remark
  \ref{rem:pullback_slice}. By the construction of Theorem
  \ref{thm:monodromic_sheaf}, the slice pair of $\mathcal{P}_{\widetilde{X}}$ is
  $(\mathcal{P}_X, T)$. 
  
  Symplectification of $\varphi$ is base change along $p$,
  \[
    \begin{tikzcd}[row sep=large, column sep=large]
      \widetilde{L} \arrow[r, "\widetilde{\varphi}"] \arrow[d, "p_L"']
      \arrow[dr, phantom, "\lrcorner", very near start] &
      \widetilde{X} \arrow[d, "p"] \\
      L \arrow[r, "\varphi"'] & X ,
    \end{tikzcd}
  \]
  so $\widetilde{\varphi}$ is a Cartesian morphism of principal $\Gm$-bundles in
  the sense of Lemma \ref{lem:monodromic_descent}(3), which therefore shows that
  $\widetilde{\varphi}^! \mathcal{P}_{\widetilde{X}}$ is monodromic
  with slice pair
  $(\varphi^! \mathcal{P}_X, \varphi^! T)$, and tensoring with
  $\Lambda^{\otimes (\rk E + n + 1)}$ multiplies the automorphism by
  $(-1)^{\rk E + n + 1}$ by Definition \ref{def:slicepair}. The target of
  $\mu_{\widetilde{L}}$ therefore has slice pair
  \[ \bigl(\varphi^! \mathcal{P}_X,\; S \bigr), \qquad
  S := (-1)^{\rk E + n + 1}\, \varphi^! T. \]

  Applying the slice map of Lemma \ref{lem:monodromic_descent}(2), we define
  \[ \mu_L := \mathrm{sl}(\mu_{\widetilde{L}}) \in
  \Hom_{L}\bigl(\F_L[\vdim L],\, \varphi^! \mathcal{P}_X\bigr). \]
  By exactness of the descent sequence, $\mu_L$ lies in the $S$-invariants, that
  is $S \circ \mu_L = \mu_L$. Substituting the value of $S$ and multiplying both
  sides by the sign $(-1)^{\rk E + n + 1}$, which is an involution, we obtain
  \[ T \circ \mu_L = (-1)^{\rk E + n + 1}\, \mu_L , \]
  the first relation of the statement. For the second, Proposition
  \ref{prop:sign_character} gives $\theta = -T$, so
  \[ \theta \circ \mu_L = -\, T \circ \mu_L = (-1)^{\rk E + n + 2}\, \mu_L
  = (-1)^{n + \rk E}\, \mu_L . \]
  Since a change of
  trivialization of the torsor modifies the slice identifications by
  a power of $S$,
  which fixes $\mu_L$ by the displayed relation, the class $\mu_L$ is
  independent of
  the choices made.
\end{proof}

\begin{remark}[The twist is necessary]\label{rem:twist_necessary}
  The factor $\Lambda^{\otimes \rk E}$ cannot be removed. Let
  $s(x) = x^2$ on $\mathbb{A}^1$ and let $L$ be the Legendrian point, modeled by
  $V = \mathbb{A}^1$, $\Psi = \mathrm{id}$, $E$ of rank $1$. The restriction of
  $\mathcal{P}_{\widetilde{X}} = \phi_{tx^2}$ to its support, the
  $t$-axis orbit, is
  a shift of the quadratic Kummer local system $\mathscr{K}$, whose
  orbit monodromy
  is $-\mathrm{id}$. As $H^\bullet_{\mathrm{et}}(\Gm, \mathscr{K})
  = 0$, every
  morphism from a constant complex to
  $\widetilde{\varphi}^! \mathcal{P}_{\widetilde{X}}$ vanishes, so an untwisted
  fundamental class is necessarily zero. After tensoring with
  $\Lambda \simeq \mathscr{K}$ the relevant cohomology group is nonzero, so the
  twisted construction is not obstructed for this reason. The same computation
  shows that the untwisted operator $T$ acts by $-1$ on the stalk of
  $\mathcal{P}_X$ in this chart, so no nonzero morphism can be
  $T$-invariant. The
  relevant operator is the twisted one.
\end{remark}

\subsection{The Parity Defect}
\label{subsec:parity}

The exponent in the local verification is $n + \rk E$, which depends on the
presentation through both of its terms. Its parity does not, and it is the
parity that controls whether a class can be invariant.

\begin{definition}[Parity defect]\label{def:parity_defect}
  Let $\varphi \colon L \to X$ be a Legendrian with a local presentation
  $(V, \Psi, E, q, \sigma)$ over a contact Darboux chart with base of
  dimension $n$, as in Theorem \ref{thm:JS_local}. The
  \bfem{parity defect} of the presentation is
  $\delta := n + \rk E \in \mathbb{Z}/2$.
\end{definition}

The defect is defined through a choice of presentation, and the next
proposition identifies it with an invariant of the Legendrian.

\begin{proposition}[The parity defect is the parity of the virtual dimension]
  \label{prop:parity_formula}
  For every local presentation of a Legendrian $\varphi \colon L \to X$,
  \[ \delta \;\equiv\; \vdim L \pmod 2 . \]
  In particular $\delta$ is independent of the chart and of the
  presentation, and
  depends only on the connected component of $L$. We write $\delta_L$.
\end{proposition}

\begin{proof}
  By the virtual dimension formula \eqref{eq:vdimJS} of Theorem
  \ref{thm:JS_local},
  \[ \vdim L = 2 \dim V - \dim U - \rk E , \]
  so, solving for the rank of the bundle,
  \[ \rk E = 2 \dim V - \dim U - \vdim L . \]
  Substituting this into Definition \ref{def:parity_defect} and using
  $\dim U = n$,
  \[ \delta = n + \rk E = n + \bigl( 2 \dim V - n - \vdim L \bigr)
  = 2 \dim V - \vdim L . \]
  Reducing modulo $2$ removes the term $2 \dim V$, leaving
  $\delta \equiv \vdim L$. The quantities $\dim V$ and $\rk E$ depend on the
  presentation and $n$ on the chart, while $\vdim L$ depends on neither, so the
  reduction is independent of all of them.
\end{proof}

Proposition \ref{prop:local_verification} gives $\theta \circ \mu_L =
(-1)^{n + \rk E} \mu_L$, and $n + \rk E$ is the parity defect of Definition
\ref{def:parity_defect}, so the local class satisfies
\[ \theta \circ \mu_L = (-1)^{\delta_L}\, \mu_L = (-1)^{\vdim L}\, \mu_L . \]

Note that the computation is consistent with the stabilizations of Proposition
\ref{prop:theta_glues}. A chart stabilization $(U', s') = (U \times
  \mathbb{A}^k,
s \boxplus q_U)$ with $q_U$ non-degenerate of rank $k$ is accompanied by
\[ V' = V \times \mathbb{A}^k, \quad \Psi' = \Psi \times \mathrm{id}, \quad
  E' = E \oplus \cO^{k}, \quad q' = q \oplus (-q_U), \quad
\sigma' = (\sigma, w) , \]
which satisfies $q'(\sigma') = q(\sigma) - q_U(w) = -(s \boxplus q_U) \circ
\Psi'$. Here $n$, $\dim V$ and $\rk E$ each increase by $k$, so that
\[ \vdim L' = 2(\dim V + k) - (n + k) - (\rk E + k) = 2 \dim V - n - \rk E
= \vdim L \]
is unchanged, and so is $\delta$. Example \ref{ex:A1_parity} shows that both
values of $\delta_L$ are realized on a single contact chart, so the defect
cannot be removed by a choice of chart.

\subsection{The Defect is Not Vacuous}
\label{subsec:nonvacuous}

The defect would carry no information if the classes it obstructs vanished for
other reasons. We compute the perverse sheaf of a Morse--Bott chart, exhibit
both parities on the $A_1$ chart, and produce an odd Legendrian whose local
class is nonzero.

\begin{lemma}[The perverse sheaf of a Morse--Bott chart]\label{lem:mb_chart}
  Let $U = \mathbb{A}^n$ with coordinates $x_1, \dots, x_n$ and let
  $s = q(x_1, \dots, x_k)$ be a non-degenerate quadratic form in the first $k$
  coordinates, $1 \leq k \leq n$. Write $X = \Delta\mathrm{loc}(s)$ for the
  associated contact Darboux chart. Then
  \[ X^{\mathrm{cl}} \simeq \mathbb{A}^{n-k}, \qquad
  \mathcal{P}_X \simeq \F_{\mathbb{A}^{n-k}}[\, n-k \,], \]
  and the operators of Theorem \ref{thm:monodromic_sheaf} act on it
  by the scalars
  \[ T_{U, \widetilde{f}} = (-1)^k, \qquad \varepsilon_U = (-1)^{n+1}, \qquad
  T = (-1)^{n+k+1}, \qquad \theta = (-1)^{n+k} . \]
\end{lemma}

\begin{proof}
  The symplectified potential is $\widetilde{f} = t \cdot q(x_1, \dots, x_k)$ on
  $\widetilde{U} = \mathbb{A}^n \times \Gm$. Differentiating,
  $\partial_{x_i} \widetilde{f} = t \,\partial_{x_i} q$ for $i \leq k$,
  $\partial_{x_i} \widetilde{f} = 0$ for $i > k$, and
  $\partial_t \widetilde{f} = q$. As $t$ is invertible and $q$ is
  non-degenerate, the first group of equations forces $x_1 = \cdots = x_k = 0$,
  and the last is then automatic. Hence
  \[ \mathrm{Crit}(\widetilde{f}) = \{ x_1 = \cdots = x_k = 0 \}
  \simeq \mathbb{A}^{n-k} \times \Gm . \]
  Taking $\Gm$-quotients gives $X^{\mathrm{cl}} \simeq \mathbb{A}^{n-k}$.

  Transversally to this locus $\widetilde{f}$ is a non-degenerate
  quadratic form of
  rank $k$ with coefficient matrix $t \cdot q_{ij}$, so by Example
  \ref{ex:quadratic} the complex $\mathcal{PV}_{U, \widetilde{f}}$ is
  the perverse constant sheaf on $\mathrm{Crit}(\widetilde{f})$, twisted by the
  $\mu_2$-system of square roots of $\det(t \, q_{ij}) = t^k \det(q_{ij})$, and
  $T_{U, \widetilde{f}} = (-1)^k$. Proposition \ref{prop:sign_character} gives
  $\varepsilon_U = (-1)^{n+1}$ for contact orientation data, so that
  \[ T = T_{U, \widetilde{f}} \, \varepsilon_U = (-1)^{k} (-1)^{n+1}
    = (-1)^{n+k+1} , \qquad
  \theta = (-1)^n T_{U, \widetilde{f}} = (-1)^{n+k} , \]
  the second by Definition \ref{def:twisted}. These satisfy $\theta = -T$ as
  required. Finally $\mathrm{Crit}(\widetilde{f})$ is smooth of dimension
  $n-k+1$, so the perverse constant sheaf on it is
  $\F[\,n-k+1\,]$, and taking the slice of Definition \ref{def:slicepair}, which
  shifts by $[-1]$, gives $\mathcal{P}_X \simeq \F_{\mathbb{A}^{n-k}}[\,n-k\,]$.
\end{proof}

\begin{example}[Both parities occur over the $A_1$ chart]\label{ex:A1_parity}
  Let $U = \mathbb{A}^1 = \Spec \K[x]$ and $s = x^2$, so that $n =
  \dim U = 1$. By
  Lemma \ref{lem:KT_model} the contact chart is $X = \Delta\mathrm{loc}(s) =
  \Spec(A)$, where $A$ is generated over $\K[x]$ by $y, z$ in degree $-1$ with
  $d(y) = 2x$ and $d(z) = x^2$. By Theorem
  \ref{lem:chart_sympl} its symplectification is $\widetilde{X} \simeq
  \dCrit(t \cdot x^2)$ over $\widetilde{U} = \mathbb{A}^1 \times
  \Gm$. We exhibit
  two Legendrians in $X$ with different parity defects.

  \emph{(a) The reduced point.} Take
  \[ V = \mathbb{A}^1, \quad \Psi = \mathrm{id}, \quad E = \cO_V, \quad
  q(u) = -u^2, \quad \sigma = x . \]
  The form $q$ is non-degenerate, and substituting $\sigma$ into $q$ gives
  $q(\sigma) = -x^2 = -\, s \circ \Psi$, so $(V, \Psi, E, q, \sigma)$ is a local
  presentation in the sense of Theorem \ref{thm:JS_local}. 
  
  The
  section $\sigma$ is
  a regular element of $\K[x]$, so its Koszul complex is a resolution and the
  derived zero locus is the reduced point $Z(\sigma) = \Spec \K[x]/(x)$. By
  \eqref{eq:vdimJS},
  \[ \vdim L = 2 \cdot 1 - 1 - 1 = 0 . \]
  Since $\det E = \cO_V$ is trivial, $L$ is orientable by Theorem
  \ref{thm:JS_local}, and $\varphi$ is proper. The parity defect is
  \[ \delta = n + \rk E = 1 + 1 \equiv 0 \pmod 2 . \]

  \emph{(b) The odd-dimensional point.} Take
  \[ V = \mathbb{A}^1, \quad \Psi = \mathrm{id}, \quad E = \cO_V^{\oplus 2},
  \quad q(u_1, u_2) = u_1 u_2, \quad \sigma = (x, -x) . \]
  The hyperbolic form $q$ is non-degenerate, and substituting $\sigma$ gives
  \[ q(\sigma) = x \cdot (-x) = -x^2 = -\, s \circ \Psi , \]
  so this too is a local presentation. Its derived zero locus is the Koszul
  algebra $\K[x]\langle \xi_1, \xi_2 \rangle$ with $d(\xi_1) = x$ and
  $d(\xi_2) = -x$, where $\xi_1, \xi_2$ have degree $-1$. Substituting
  $\eta := \xi_1 + \xi_2$ gives $d(\eta) = x + (-x) = 0$, while
  $\K[x]\langle \xi_1 \rangle$ with $d(\xi_1) = x$ is a Koszul resolution of
  $\K$. Rearranging the tensor factors accordingly,
  \[ Z(\sigma) \;\simeq\; \Spec\bigl( \Lambda(\eta) \bigr), \qquad
  \deg \eta = -1, \quad d \eta = 0 . \]
  We obtain the cotangent complex $\LL_L \simeq [\, E^\vee \to T^\vee_U \,]|_L$ in
  degrees $-1, 0$, of rank $1 - 2 = -1$, in agreement with \eqref{eq:vdimJS}:
  \[ \vdim L = 2 \cdot 1 - 1 - 2 = -1 . \]
  Since $\det E = \det(\cO_V^{\oplus 2}) = \cO_V$ is trivial, $L$ is again
  orientable, and $\varphi$ is proper, the classical truncation being
  the origin.
  The parity defect is
  \[ \delta = n + \rk E = 1 + 2 \equiv 1 \pmod 2 . \]

  One checks that both presentations satisfy the hypotheses of Theorem \ref{thm:JS_local} over the
  same contact chart, and their defects agree with Proposition
  \ref{prop:parity_formula}, since $\vdim L = 0$ in case (a) and
  $\vdim L = -1$ in
  case (b). By Propositions \ref{prop:local_verification} and
  \ref{prop:parity_formula} the local class of (b) satisfies
  $\theta \circ \mu_L = - \mu_L$, while that of (a) is $\theta$-invariant.
  Thus both parities are realized over one chart, and no choice of
  chart can remove
  the defect.

\end{example}

Over the chart of Example \ref{ex:A1_parity} the defect is not yet
witnessed by a
nonzero class. Applying Lemma \ref{lem:mb_chart} with $n = k = 1$ gives
$X^{\mathrm{cl}}$ the reduced point, $\mathcal{P}_X \simeq \F$ concentrated in
degree $0$, and
\[ T = (-1)^{1+1+1} = -1, \qquad \theta = (-1)^{1+1} = +1 . \]
As $L^{\mathrm{cl}}$ is also the reduced point in both presentations,
\[ \Hom_L\bigl( \F_L[\vdim L],\, \varphi^! \mathcal{P}_X \bigr)
  = \Hom_{\mathrm{pt}}\bigl( \F[\vdim L],\, \F \bigr)
  =
  \begin{cases} \F & \vdim L = 0, \\ 0 & \vdim L = -1,
\end{cases} \]
since $\Hom(\F[-1], \F) = \mathrm{Ext}^1_{\F}(\F, \F) = 0$. In case
(b) the class
$\mu_L$ therefore vanishes and the eigenvalue relation holds
vacuously. Indeed no
Legendrian over this chart can witness it, since $X^{\mathrm{cl}}$ is a
point, $L$ is
affine and $\varphi$ is proper, so $L^{\mathrm{cl}}$ is finite and
$\varphi^! \mathcal{P}_X$ is concentrated in degree $0$, and the
$\Hom$-group vanishes unless $\vdim L = 0$. Example \ref{ex:odd_nonzero} works
over a chart with positive-dimensional critical locus, where the obstruction
becomes visible.

\begin{example}[An odd Legendrian with a nonvanishing
  class]\label{ex:odd_nonzero}
  Let $U = \mathbb{A}^2 = \Spec \K[x_1, x_2]$ and $s = x_1^2$, so
  that $n = 2$. By
  Lemma \ref{lem:KT_model} the chart is $X = \Delta\mathrm{loc}(s) =
  \Spec(A)$ with
  $A$ generated over $\K[x_1, x_2]$ by $y_1, y_2, z$ in degree $-1$ and
  \[ d(y_1) = \frac{\partial s}{\partial x_1} = 2 x_1, \qquad
  d(y_2) = \frac{\partial s}{\partial x_2} = 0, \qquad d(z) = s = x_1^2 , \]
  so that $H^0(A) = \K[x_1,x_2]/(2x_1, x_1^2) = \K[x_2]$ and
  $X^{\mathrm{cl}} = \mathbb{A}^1_{x_2}$.

  \emph{The perverse sheaf.} This is the Morse--Bott chart of Lemma
  \ref{lem:mb_chart} with $n = 2$ and $k = 1$, so
  \[ \mathcal{P}_X \simeq \F_{\mathbb{A}^1}[1], \qquad
  T = (-1)^{2+1+1} = +1, \qquad \theta = (-1)^{2+1} = -1 . \]

  \emph{The Legendrian.} Take
  \[ V = \mathbb{A}^2, \quad \Psi = \mathrm{id}, \quad E = \cO_V, \quad
  q(u) = -u^2, \quad \sigma = x_1 . \]
  Substituting, $q(\sigma) = -x_1^2 = -\, s \circ \Psi$, so this is a local
  presentation in the sense of Theorem \ref{thm:JS_local}. The
  section $\sigma$ is
  regular, so $Z(\sigma) = \Spec \K[x_1,x_2]/(x_1) =
  \mathbb{A}^1_{x_2}$ is reduced,
  $\varphi^{\mathrm{cl}}$ is the identity of $\mathbb{A}^1$ and so proper, and
  $\det E = \cO_V$ is trivial so $L$ is orientable. By \eqref{eq:vdimJS},
  \[ \vdim L = 2 \cdot 2 - 2 - 1 = 1 , \]
  which is odd, and the parity defect is $\delta = n + \rk E = 2 + 1
  \equiv 1$, in
  agreement with Proposition \ref{prop:parity_formula}.

  \emph{The $\Hom$-group.} As $\varphi^{\mathrm{cl}}$ is an isomorphism,
  $\varphi^! \mathcal{P}_X = \mathcal{P}_X = \F_{\mathbb{A}^1}[1]$, and we obtain
  \[ \Hom_L\bigl( \F_L[\vdim L],\, \varphi^! \mathcal{P}_X \bigr)
    = \Hom_{\mathbb{A}^1}\bigl( \F[1], \F[1] \bigr)
  = H^0_{\mathrm{et}}(\mathbb{A}^1, \F) = \F , \]
  a one-dimensional space, on which $\theta$ acts by $-1$ through
  post-composition.
  Consequently
  \[ \Hom^{\theta = +1} = 0, \qquad \Hom^{\theta_L = +1} = \F , \]
  where $\theta_L := (-1)^{\vdim L} \theta = -\theta$. This operator
  is introduced
  in Remark \ref{rem:graded_equiv} below, where the grading that produces it is
  defined. The class is moreover nonzero by Proposition
  \ref{prop:odd_nonvanishing}.
\end{example}

\begin{proposition}[Nonvanishing of the local
  class]\label{prop:odd_nonvanishing}
  In the situation of Example \ref{ex:odd_nonzero}, the class $\mu_L$ of
  Proposition \ref{prop:local_verification} is nonzero, and generates the
  one-dimensional space $\Hom_L(\F_L[\vdim L], \varphi^! \mathcal{P}_X)$.
\end{proposition}

\begin{proof}
  We first reduce to the symplectification. Taking
  $A = \F_L[1]$ and $B = \varphi^! \mathcal{P}_X =
  \F_{\mathbb{A}^1}[1]$, the slice
  automorphism of Proposition \ref{prop:local_verification} is
  $S = (-1)^{\rk E + n + 1} T = (-1)^4 (+1) = +1$, and the descent sequence of
  Lemma \ref{lem:monodromic_descent}(2) has kernel term
  \[ \Hom_L\bigl( A, B[-1](-1) \bigr)_S =
  H^{-1}_{\mathrm{et}}(\mathbb{A}^1, \F) = 0 , \]
  so the slice map $\mathrm{sl}$ is an isomorphism onto $\Hom_L(A,
  B)^S = \F$. Hence
  $\mu_L \neq 0$ as soon as $\mu_{\widetilde{L}} \neq 0$, which we now verify.

  The construction in Step 3 of the proof of that proposition follows the chain of \cite[Proposition
  5.20]{AmorimBassat}, whose three steps
  we check in turn. We work upstairs, where by Step 1 of that proof the
  data are
  \[ \widetilde{U} = \mathbb{A}^2 \times \Gm, \quad
    \widetilde{f} = t x_1^2, \quad \widetilde{V} = \widetilde{U}, \quad
    \widetilde{\Psi} = \mathrm{id}, \quad \widetilde{E} =
    \cO_{\widetilde{V}}, \quad
  \widetilde{q}(u) = - t u^2, \quad \widetilde{\sigma} = x_1, \]
  so that $\dim \widetilde{U} = \dim \widetilde{V} = 3$ and
  $\widetilde{q}(\widetilde{\sigma}) = -t x_1^2 = - \widetilde{f}$.

  First the loci. Since $t$ is invertible on $\Gm$, the ideal
  $(2 t x_1, x_1^2)$ equals $(x_1)$, so $\mathrm{Crit}(\widetilde{f}) = \{ x_1 = 0 \}$, while $\widetilde{f}^{-1}(0)$ is cut out by $(x_1^2)$ and
  carries a nilpotent. The same holds for
  $(\widetilde{q} \circ \widetilde{\sigma})^{-1}(0)$ against
  $\widetilde{L} = \widetilde{\sigma}^{-1}(0) = \{x_1 = 0\}$. 
  
  The
  \'etale site is
  insensitive to nilpotents, so the maps $i$ and $j$ of \cite[Proposition
  5.20]{AmorimBassat} induce equivalences on constructible sheaves, and
  $\Psi_0 = \mathrm{id}$ because $\widetilde{\Psi} = \mathrm{id}$.

  The first step is a chain of isomorphisms by loc.\ cit., using
  \[\F_{0_{\widetilde{E}}} \simeq \phi_{\widetilde{q}}
  \F_{\widetilde{E}}[\rk E - 1] \otimes \Lambda^{\otimes \rk E},\] which is
  \eqref{eq:TSquad} read backwards, the discriminant torsor being trivialized
  by the orientation.

  The second step applies $s_0^* \phi_{\widetilde{q}}$ to the unit
  $\F_{\widetilde{E}} \to \widetilde{\sigma}_* \F_{\widetilde{V}}$. Since
  $\widetilde{q}$ is non-degenerate on the fibers of $\widetilde{E}$,
  the complex
  $\phi_{\widetilde{q}} \F_{\widetilde{E}}$ is supported on the zero
  section. And
  since $\widetilde{\sigma}$ vanishes on $\widetilde{L}$, the map $s_0$ factors
  through the closed embedding $\widetilde{\sigma}$, and therefore
  $s_0^* \widetilde{\sigma}_* = (-)|_{\widetilde{L}}$. Both sides are
  rank-one systems on $\widetilde{L}$ twisted by $\Lambda$, and it suffices to
  check that the map is an isomorphism on stalks. Fix $p \in \widetilde{L}$ and
  work in the transversal slice with coordinates $(x_1, u)$, in which
  $\widetilde{q} = -t u^2$ and the graph of $\widetilde{\sigma}$ is the line
  $\{ u = x_1 \}$. The critical locus of $\widetilde{q}$ in this slice is the
  $x_1$-axis, and the graph is transverse to it, since the fiberwise
  differential
  of $\widetilde{\sigma}$ is $d x_1 \neq 0$. The Milnor fiber of
  $\widetilde{q}$ at
  $p$ is $\{ -t u^2 = \epsilon \}$, a pair of lines $u = \pm \sqrt{-\epsilon/t}$
  parallel to the $x_1$-axis, and the graph meets each of them in exactly one
  point. So the inclusion of the Milnor fiber of
  $\widetilde{q}|_{\mathrm{graph}} = -t x_1^2$ into that of $\widetilde{q}$ is a
  homotopy equivalence, and the induced map on the rank-one vanishing
  cohomologies
  is an isomorphism.

  The third step is the composition of the base change
  $\Psi_{0!} \phi_{\widetilde{f} \circ \widetilde{\Psi}} \to
  \phi_{\widetilde{f}} \widetilde{\Psi}_!$ with the trace
  $\delta_{\widetilde{\Psi}}$. Both are identities because
  $\widetilde{\Psi} = \mathrm{id}$, and the accompanying shift
  $[\, 2 \dim \widetilde{U} - 2 \dim \widetilde{V} \,]$ is trivial since the two
  dimensions are equal.

  The composite is an isomorphism, in particular nonzero, and since
  adjunction is a bijection on morphism spaces its adjoint
  $\mu_{\widetilde{L}}$ is
  nonzero as well. Hence $\mu_L \neq 0$.
\end{proof}

Combining Example \ref{ex:odd_nonzero} with Proposition
\ref{prop:odd_nonvanishing}, the ungraded formulation of the conjecture, which
demands $\theta \circ \mu_L = \mu_L$, is satisfiable on that
Legendrian only by
$\mu_L = 0$, whereas the graded formulation of Conjecture
\ref{conj:joyce} admits
the generator of a one-dimensional space. This is the precise sense
in which the
orientation datum must be graded.

\subsection{Graded Orientations and the Conjecture}
\label{subsec:gradedjoyce}

The obstruction is therefore in the orientation datum rather than in the class.
Twisting that datum by the parity of the virtual dimension absorbs the sign and
leaves a formulation that is satisfiable in both parities.

\begin{definition}[Graded contact orientation]\label{def:graded_orientation}
  Let $Z$ be a $-1$-shifted contact derived Artin stack with contact orientation
  data $R_Z$ as in Definition \ref{def:orientation}, and let
  $\underline{\Lambda}$ denote the $\mu_2$-torsor on $\widetilde{Z}$
  underlying the
  monodromic sign twist $\Lambda$ of Definition \ref{def:slicepair}. For a
  Legendrian morphism $\varphi \colon L \to Z$ write
  \[ \mathcal{O}_\varphi := \underline{\mathrm{Isom}}^{\,\square}
  \bigl( K_{\widetilde{L}},\, \widetilde{\varphi}^*(p^* R_Z) \bigr) \]
  for the $\mu_2$-torsor of isomorphisms whose tensor square is the canonical
  isomorphism $(K_{\widetilde{L}})^{\otimes 2} \simeq
  \widetilde{\varphi}^*(K_{\widetilde{Z}})$, so that a trivialization of
  $\mathcal{O}_\varphi$ is an orientation in the sense of Definition
  \ref{def:orientation}. A \bfem{graded orientation} of $\varphi$ is a
  trivialization of the twisted torsor
  \[ \mathcal{O}_\varphi \otimes_{\mu_2}
  \underline{\Lambda}^{\otimes \vdim L} . \]
  Since $\underline{\Lambda}^{\otimes 2}$ is canonically trivial,
  only the parity
  of $\vdim L$ enters. For $\vdim L$ even a graded orientation is an
  orientation,
  and for $\vdim L$ odd the two notions differ by the quadratic
  Kummer twist along
  the $\Gm$-orbits.
\end{definition}

\begin{remark}\label{rem:graded_equiv}
  Equivalently, one may retain Definition \ref{def:orientation} and modify the
  operator. For an oriented Legendrian $\varphi \colon L \to X$ set
  \[ \theta_L := (-1)^{\vdim L}\, \varphi^! \theta , \]
  the \bfem{Legendrian twisted operator}. A class is
  $\theta$-invariant for a graded
  orientation if and only if it is $\theta_L$-invariant for the underlying orientation. The two formulations differ by the orbit monodromy
  $(-1)^{\vdim L}$ of $\underline{\Lambda}^{\otimes \vdim L}$. The sign is
  multiplicative under composition of correspondences. In general
  \[ \vdim \bigl( M \times^h_{B} N \bigr) = \vdim M + \vdim N - \vdim B , \]
  so the parity is additive precisely when $\vdim B$ is even. This holds in both
  situations where the grading is composed. For $B$ a $-1$-shifted symplectic
  stack one has $\vdim B = 0$, as noted in
  \cite[Conjecture 5.22(c)]{AmorimBassat}, and for $B$ a
  $0$-shifted symplectic
  stack the non-degeneracy $\TT_B \simeq \LL_B$ gives $h^i(\TT_B) \simeq
  h^{-i}(\TT_B)^\vee$, so these have equal dimension at every point. Therefore
  $\vdim B \equiv \dim H^0(\TT_B) \pmod 2$, and $H^0(\TT_B)$ carries a
  non-degenerate alternating form, so $\vdim B$ is even. It follows that
  \[(-1)^{\vdim M} (-1)^{\vdim N} = (-1)^{\vdim (M \times^h_{B} N)}\]
  in both cases,
  and the enrichment of Definition \ref{def:Hmon} is unaffected.
\end{remark}

\begin{conjecture}[Contact Joyce Conjecture]\label{conj:joyce}
  Let $(X, \cL, \alpha)$ be an oriented $-1$-shifted contact derived stack over
  $\K$ and let $\varphi \colon L \to X$ be a proper Legendrian equipped with a
  graded orientation in the sense of Definition \ref{def:graded_orientation}.
  There exists a morphism in $D_c^b(L_{et}, \F)$
  \[ \mu_L \colon \F_L[\vdim L] \to \varphi^! \mathcal{P}_X, \]
  where $\vdim L$ denotes the virtual dimension of $L$, invariant under the
  twisted monodromy operator, $\theta \circ \mu_L = \mu_L$. Equivalently, for a
  Legendrian carrying an ordinary orientation, $\theta_L \circ \mu_L
  = \mu_L$ in the
  notation of Remark \ref{rem:graded_equiv}. Its local models in equivariant
  Darboux charts are those of Proposition \ref{prop:local_verification}.
\end{conjecture}

This is the contact analogue of the Symplectic Joyce Conjecture for
Lagrangians
\cite[Conjecture 5.18]{AmorimBassat}. Two features distinguish the contact
setting. First, the relevant operator is the twisted one. Invariance under the untwisted operator $T$ is too strong. In the local model of Remark \ref{rem:twist_necessary} every $T$-invariant morphism vanishes. Second, the
orientation datum must be graded by the virtual dimension. Legendrians of both
parities occur over a single contact chart by Example
\ref{ex:A1_parity}, and by
Example \ref{ex:odd_nonzero} the ungraded formulation forces $\mu_L = 0$ on an
odd-dimensional Legendrian whose $\Hom$-group is one-dimensional, while the
graded formulation admits a generator there. Neither phenomenon has
a counterpart
in the symplectic case, where the perverse sheaf carries no
monodromy operator.

\begin{remark}[Local verification, graded form]\label{rem:local_graded}
  In the situation of Proposition \ref{prop:local_verification}, that
  proposition gives $\theta \circ \mu_L =
  (-1)^{n + \rk E} \mu_L$ and Proposition \ref{prop:parity_formula}
  identifies $n + \rk E$ with $\vdim L$ modulo $2$, so
  \[ \theta \circ \mu_L = (-1)^{\vdim L}\, \mu_L . \]
  Multiplying both sides by $(-1)^{\vdim L}$, which is an involution, and
  using the definition $\theta_L = (-1)^{\vdim L} \varphi^! \theta$ of Remark
  \ref{rem:graded_equiv} gives $\theta_L \circ \mu_L = \mu_L$. Conjecture
  \ref{conj:joyce}, in its graded formulation, therefore holds locally for
  derived contact schemes.
\end{remark}

\begin{remark}[Global Gluing Obstruction]\label{rem:gluing}
  The local maps $\mu_L$ exist in the derived category $D_c^b(L_{et},
  \F)$ for each \'etale Darboux chart. Establishing this as a global
  theorem requires gluing these local morphisms. The object
  $\F_L[\vdim L]$ is a constructible complex rather than a perverse
  sheaf. Thus, the maps $\mu_L$ do not satisfy local descent. Gluing
  these local maps into a global morphism is obstructed by classes
  residing in the first extension group
  $\mathrm{Ext}^1_{L_{et}}(\F_L[\vdim L], \varphi^! \mathcal{P}_X)$
  within the Lisse-\'etale topos \cite[Remark 5.21]{AmorimBassat}.
  Explicitly, since local choices of $\mu_L$ exist, the obstruction
  to patching them on overlaps resides in the \v{C}ech cohomology
  group $\check{H}^1(L_{et}, \mathcal{H}om(\F_L[\vdim L], \varphi^!
  \mathcal{P}_X))$, which maps to the global extension group via the
  local-to-global spectral sequence. The conjecture posits the
  existence of a global choice bypassing these obstructions.
\end{remark}

\section{Perverse Linearizations of the Legendrian Categories}
\label{sec:linearization}

We linearize the non-linear 2-categories reviewed in Section \ref{sec:review}
using the six-functor formalism for $\ell$-adic sheaves \cite{LO2,
LZ}, under the
following standing assumptions.

\begin{setup}[Standing assumptions]\label{setup:joyce}
  Throughout this section we assume the symplectic Joyce conjecture of
  Amorim--Ben-Bassat in both of its forms, together with a monodromic
  refinement:
  \begin{enumerate}
    \item[(J1)] For every proper oriented Lagrangian $\varphi \colon
      L \to X$ into an
      oriented $-1$-shifted symplectic derived Artin stack there
      exists a fundamental
      class $\mu_L \colon \F_L[\vdim L] \to \varphi^! \mathcal{P}_X$
      \cite[Conjecture 5.18]{AmorimBassat}.
    \item[(J2)] The classes of (J1) are compatible with
      Thom--Sebastiani external
      products and with compositions of Lagrangian correspondences, in the form
      required by the pull-push functors of \cite[Section 6]{AmorimBassat}
      \cite[Conjecture 5.22]{AmorimBassat}.
    \item[(J3)] (\emph{Monodromic refinement.}) When the Lagrangians
      and the ambient
      stack carry $\Gm$-actions scaling the symplectic form with
      weight $1$, so that
      the perverse sheaves are monodromic and carry the twisted
      operator $\theta$ of
      Proposition \ref{prop:theta_glues}, the classes of (J1) and the
      isomorphisms of
      (J2) can be chosen $\theta$-invariant, after grading the
      orientation datum by the
      virtual dimension as in Definition
      \ref{def:graded_orientation}, and compatible
      with the monodromic structures. In particular the
      compatibilities of (J2) hold
      in $D^b_{\mathrm{mon}}$, and the K\"unneth and Thom--Sebastiani
      isomorphisms
      entering them are $\theta$-equivariant.
  \end{enumerate}
  All results of this section, and Theorem \ref{thm:C}, are conditional on
  (J1)--(J3).
\end{setup}

Assumption (J3) does not follow from (J1)--(J2). The attractor case of
(J1)--(J2) is a theorem of Kinjo--Park--Safronov \cite[Theorem 7.23]{KPS}, but
it is proved for a $\Gm$-action preserving both the d-critical structure and
the orientation, and the corresponding hyperbolic localization statement fails
for weight-one actions. Neither hypothesis holds here. The action scales the
potential with weight one, and the orientation local system has monodromy
$(-1)^{n+1}$ along the orbits by Proposition \ref{prop:sign_character}, so it
is not preserved when $n$ is even.

The attractor case is therefore not
evidence for the monodromic refinement (J3), and the support for (J3) is the
unconditional local verification below. Proposition
\ref{prop:local_verification} and Remark \ref{rem:local_graded} verify its
local model unconditionally for symplectifications of Legendrians in derived
schemes, once the orientation datum is graded by the virtual dimension in the
sense of Definition \ref{def:graded_orientation}. Example \ref{ex:odd_nonzero}
shows that this grading is not a normalization but is forced, in that the
ungraded requirement annihilates the class on odd-dimensional Legendrians.

\begin{definition}\label{def:Hmon}
  Let $M$ be a monodromic complex on a $\Gm$-space arising from an
  oriented shifted
  contact stack, equipped with the twisted operator $\theta$ of Proposition
  \ref{prop:theta_glues}. The \bfem{monodromic hypercohomology}
  \[ \mathbb{H}^\bullet_{\mathrm{mon}}(Y, M) := \big( \mathbb{H}^\bullet_{et}(Y,
  M), \; \mathbb{H}^\bullet(\theta) \big) \]
  is the \'etale hypercohomology together with the automorphism
  induced by $\theta$,
  an object of the symmetric monoidal category of graded $\F$-vector spaces
  equipped with an automorphism. This category is the enrichment of the
  2-categories below.
\end{definition}

\begin{lemma}\label{lem:intersection_orientation}
  Let $N_0, N_1 \to L_{01}$ be Legendrians with graded orientations into a
  $0$-shifted contact derived stack. The orientation induced on the
  $-1$-shifted symplectic
  intersection $\widetilde{N}_{01} = \widetilde{N}_0
  \times^h_{\widetilde{L}_{01}}
  \widetilde{N}_1$ by the construction of \cite[Section 5]{AmorimBassat} is
  $\Gm$-equivariant. So $\mathcal{P}_{\widetilde{N}_{01}}$ carries
  the twisted
  operator $\theta$ of Proposition \ref{prop:theta_glues}, and Definition
  \ref{def:Hmon} applies.
\end{lemma}

\begin{proof}
  Since $\widetilde{N}_i = N_i \times_{L_{01}} \widetilde{L}_{01}$ and
  $\widetilde{N}_{01} = N_{01} \times_{L_{01}} \widetilde{L}_{01}$
  are pulled-back
  torsors, the relative cotangent complexes along all torsor projections are
  equivariantly trivialized by $d_{\DR}t/t$ as in the proof of Proposition
  \ref{prop:orientation_obstruction}, so every canonical bundle entering the
  construction is the pullback of its counterpart on the base, with
  the weight-$0$
  equivariant structure, and likewise the chosen square roots. The orientations of $N_0$ and $N_1$ are, by Definition \ref{def:orientation}, equivariant data on
  $\widetilde{N}_0$ and $\widetilde{N}_1$ of weight $0$. The Amorim--Ben-Bassat
  orientation of the Lagrangian intersection is built from these
  square roots and
  the canonical isomorphisms of determinant lines of the fiber-product cotangent
  sequence, all of which are pulled back along the torsor. A tensor product of
  weight-$0$ equivariant isomorphisms is weight-$0$ equivariant.

  The grading is additive modulo $2$. By Definition
  \ref{def:graded_orientation} the gradings attached to the Legendrians $N_0$
  and $N_1$ are $\underline{\Lambda}^{\otimes \vdim N_0}$ and
  $\underline{\Lambda}^{\otimes \vdim N_1}$, so their tensor product twists
  $\mathcal{O}$ by $\underline{\Lambda}^{\otimes(\vdim N_0 + \vdim N_1)}$.
  Since $\vdim \widetilde{N}_i = \vdim N_i + 1$, this agrees modulo $2$ with
  $\underline{\Lambda}^{\otimes(\vdim \widetilde{N}_0 + \vdim
  \widetilde{N}_1)}$, and we compute with the latter.
  The intersection satisfies
  \[ \vdim \widetilde{N}_{01} = \vdim \widetilde{N}_0 + \vdim \widetilde{N}_1
  - \vdim \widetilde{L}_{01} , \]
  and $\vdim \widetilde{L}_{01}$ is even, since $\widetilde{L}_{01}$ is
  $0$-shifted symplectic (Remark \ref{rem:graded_equiv}). As
  $\underline{\Lambda}^{\otimes 2}$ is canonically trivial, the twist
  above equals
  $\underline{\Lambda}^{\otimes \vdim \widetilde{N}_{01}}$, which is the grading
  attached to the intersection. The induced datum is an orientation of
  $\widetilde{N}_{01}$ twisted by
  $\underline{\Lambda}^{\otimes \vdim \widetilde{N}_{01}}$, the grading of
  Definition \ref{def:graded_orientation} read on the intersection rather than
  on a Legendrian morphism.
\end{proof}

\begin{theorem}\label{thm:linear_category}
  Assume Setup \ref{setup:joyce}. For a $1$-shifted contact derived
  stack $X$, there exists a bicategory $\LFc(X)$ enriched over graded
  $\mathbb{F}$-vector spaces equipped with an automorphism, where:
  \begin{itemize}
    \item \textbf{Objects:} Legendrian morphisms $\varphi \colon L
      \to X$ equipped with graded orientations (Definition
      \ref{def:graded_orientation}).
    \item \textbf{1-Morphisms:} For objects $L_0$ and $L_1$,
      1-morphisms are proper Legendrian correspondences with graded
      orientations, given by Legendrian morphisms $\varphi_N \colon N
      \to L_{01}$ into the $0$-shifted contact derived intersection
      $L_{01} = L_0 \times_X^h L_1$.
    \item \textbf{2-Morphisms:} The mapping spaces between
      1-morphisms $\varphi_{N_0} \colon N_0 \to L_{01}$ and
      $\varphi_{N_1} \colon N_1 \to L_{01}$ are defined as the
      monodromic \'etale hypercohomology of the monodromic perverse
      sheaf on the intersection of their derived symplectifications
      $\widetilde{N}_{01} = \widetilde{N}_0
      \times_{\widetilde{L}_{01}}^h \widetilde{N}_1$, which is a
      $-1$-shifted symplectic space:
      \[ \Hom_{\LFc(X)}(N_0, N_1) :=
        \mathbb{H}^\bullet_{\mathrm{mon}}(\widetilde{N}_{01},
      \mathcal{P}_{\widetilde{N}_{01}}[-\vdim \widetilde{N}_0]). \]
  \end{itemize}
\end{theorem}

\begin{proof}
  We construct the vertical composition $a_2 \odot a_1$. Given three
  parallel 1-morphisms $\varphi_{N_i} \colon N_i \to L_{01}$ for
  $i=0, 1, 2$ in $\LFc(X)_1(L_0, L_1)$, the derived triple
  intersection \[N_{012} = N_0 \times_{L_{01}}^h N_1 \times_{L_{01}}^h
  N_2\] lacks a natural contact structure. To perform intersection
  calculus, we pass to their derived symplectifications. The
  symplectification $\widetilde{N}_{012}$ embeds as a proper
  Lagrangian inside the $-1$-shifted product symplectic space
  $\widetilde{N}_{01}^- \times \widetilde{N}_{12}^- \times
  \widetilde{N}_{02}$ \cite[Theorem 2.14]{AmorimBassat}. This
  yields a pull-push correspondence diagram
  \begin{center}
    \begin{tikzcd}[row sep=large, column sep=large]
      & \widetilde{N}_{012} \arrow[dl, "\tilde{\pi}_{\mathrm{in}}"']
      \arrow[dr, "\tilde{\pi}_{\mathrm{out}}"] & \\
      \widetilde{N}_{01} \times \widetilde{N}_{12} & & \widetilde{N}_{02}
    \end{tikzcd}
  \end{center}
  where $\tilde{\pi}_{\mathrm{in}} = \tilde{\pi}_{01} \times
  \tilde{\pi}_{12}$. By Thom--Sebastiani, Theorem \ref{thm:TS},
  the external tensor product $\mathcal{P}_{\widetilde{N}_{01}}
  \boxtimes \mathcal{P}_{\widetilde{N}_{12}}$ represents the perverse
  sheaf on the domain symplectic space. By Assumptions (J2) and (J3)
  of Setup \ref{setup:joyce}, there is a
  $\theta$-invariant fundamental characteristic class, compatible
  with the monodromic
  structures, \[\mu_{\widetilde{N}_{012}} \colon
  \tilde{\pi}_{\mathrm{in}}^* (\mathcal{P}_{\widetilde{N}_{01}}
  \boxtimes \mathcal{P}_{\widetilde{N}_{12}}) [\vdim
  \widetilde{N}_{012}] \to \tilde{\pi}_{\mathrm{out}}^!
  \mathcal{P}_{\widetilde{N}_{02}}.\]

  The symplectic form has weight 1 under the $\mathbb{G}_m$-action on
  the symplectification, so the local superpotential carries weight
  1. The resulting perverse sheaf of vanishing cycles is monodromic
  with respect to the $\mathbb{G}_m$-action and does not descend to
  an equivariant perverse sheaf on the contact quotient. We thus compute
  the composition in the derived category of
  $\mathbb{G}_m$-monodromic sheaves
  $D_{\mathrm{mon}}^b(\widetilde{N}_{012, \mathrm{\acute{e}t}},
  \mathbb{F})$. The external tensor product of the perverse sheaves
  via the $\ell$-adic K\"unneth formula has
  an initial combined cohomological shift of $[-\vdim \widetilde{N}_0
  - \vdim \widetilde{N}_1]$ and the fundamental class
  $\mu_{\widetilde{N}_{012}}$ introduces a dimension shift
  corresponding to $[\vdim \widetilde{N}_{012}]$. Since the
  Lagrangians $\widetilde{N}_i$ have virtual dimension $d$ and the
  ambient symplectic space $\widetilde{L}_{01}$ has virtual dimension
  $2d$, the triple intersection has virtual dimension $3d - 2(2d) =
  -d$. Evaluating the fundamental class shifts the cohomological
  degree by $[-\vdim \widetilde{N}_{012}] = [d] = [\vdim
  \widetilde{N}_1]$. Integration along the fibers of the proper
  $\mathbb{G}_m$-equivariant projection $\tilde{\pi}_{\mathrm{out}}$
  applies the derived pushforward $R\tilde{\pi}_{\mathrm{out},!}$,
  preserving the cohomological degree. This cancels the cohomological
  shift $[-\vdim \widetilde{N}_1]$ from the input, yielding the
  output shift of $[-\vdim \widetilde{N}_0]$ on the target space
  $\widetilde{N}_{02}$.

  Using the adjunction $R\tilde{\pi}_{\mathrm{out}, !}
  \tilde{\pi}_{\mathrm{out}}^! \to \mathrm{id}$ and the fact that
  $R\tilde{\pi}_{\mathrm{out}, !} \simeq R\tilde{\pi}_{\mathrm{out},
  *}$ by the properness of $\tilde{\pi}_{\mathrm{out}}$, this characteristic class induces a
  monodromic \'etale pull-push integration functor \cite{AmorimBassat}
  \[ \Phi_{\mu} \colon
    \mathbb{H}^\bullet_{\mathrm{mon}}(\widetilde{N}_{01} \times
      \widetilde{N}_{12}, \mathcal{P}_{\widetilde{N}_{01}} \boxtimes
      \mathcal{P}_{\widetilde{N}_{12}}[-\vdim \widetilde{N}_0 - \vdim
    \widetilde{N}_1]) \to
    \mathbb{H}^\bullet_{\mathrm{mon}}(\widetilde{N}_{02},
  \mathcal{P}_{\widetilde{N}_{02}}[-\vdim \widetilde{N}_0]). \]
  The composition of $a_1 \in
  \mathbb{H}^\bullet_{\mathrm{mon}}(\widetilde{N}_{01},
  \mathcal{P}_{\widetilde{N}_{01}}[-\vdim \widetilde{N}_0])$ and $a_2
  \in \mathbb{H}^\bullet_{\mathrm{mon}}(\widetilde{N}_{12},
  \mathcal{P}_{\widetilde{N}_{12}}[-\vdim \widetilde{N}_1])$ is
  defined via this integration functor:
  \[ a_2 \odot a_1 := \Phi_{\mu} (a_1 \boxtimes a_2). \]

  Note that the intersection correspondence map $\tilde{\pi}_{\mathrm{in}}$ is
  a representable, proper morphism. Also, the projection
  $\tilde{\pi}_{\mathrm{out}}$ is proper, following from the
  properness of the Legendrian 1-morphisms (it is the Cartesian lift
    of the proper base map $\pi_{\mathrm{out}}$ across principal
  $\mathbb{G}_m$-bundles). 
  
  The derived \'etale pushforward
  $R\tilde{\pi}_{\mathrm{out},!} \simeq
  R\tilde{\pi}_{\mathrm{out},*}$ preserves constructibility of the
  sheaves. The pull-push integration functor is defined by geometric
  operations within the
  monodromic derived category. These commute with the glued
  automorphism $\theta$ of
  Proposition \ref{prop:theta_glues}, and the fundamental class is
  $\theta$-invariant
  by (J3). So $\Phi_\mu$ is a morphism of the enrichment category
  of Definition
  \ref{def:Hmon}. Proper base change
  ensures that the integration behaves compositionally across derived
  fiber products.

  To verify the axioms of a bicategory, we check the Categorical
  Interchange Law for horizontal ($*$) and vertical ($\odot$)
  composition \cite[Lemma 6.6]{AmorimBassat}. This requires
  verifying that $(b_2 \odot b_1) * (a_2 \odot a_1) =
  (-1)^{|a_2||b_1|} (b_2 * a_2) \odot (b_1 * a_1)$. This follows by
  evaluating the $2 \times 2$ grid of morphisms using derived fiber
  products of vertical triple intersections \cite[Section
  6]{AmorimBassat}. The pull-push correspondence diagrams for
  horizontal and vertical compositions form derived Cartesian
  squares. Because derived fiber products avoid Tor-independence
  assumptions, proper base change shows
  that the derived pushforward along proper vertical projections
  commutes with pullback along horizontal embeddings. This provides
  the natural isomorphisms required to satisfy the interchange law.
\end{proof}

The identity $1$-morphism on an object $\varphi \colon L \to X$ is
the diagonal
Legendrian $\Delta_L \colon L \to L_{LL} = L \times_X^h L$ of
\cite{IzbudakBerktav2}, whose symplectification is the diagonal Lagrangian of
\cite[Proposition 3.8]{AmorimBassat}. The underlying non-linear bicategory
structure is not constructed here.

The composition operations descend to the
contact quotients and satisfy the coherence conditions by \cite[Theorem
6.1]{IzbudakBerktav2}, and what is added above is the linearization of the
$2$-morphism spaces. Horizontal composition, the unitors, and the associators
are obtained from the same pull-push operations applied to the corresponding
derived fiber products, following \cite[Section 6]{AmorimBassat}, with the
six operations taken in the monodromic derived category. The coherence axioms
reduce to proper base change and to (J2)--(J3). The proof above records the two verifications
that differ from
the symplectic case. These are vertical composition through the
symplectification, and the
interchange law.

\begin{theorem}[Global Categorified Contact 2-Category]\label{thm:global_2cat}
  Assume Setup \ref{setup:joyce}. There exists a linear weak
  2-category $\mathit{LLeg}_0$ of Legendrian correspondences. Its
  objects are $0$-shifted contact derived stacks, its 1-morphisms are
  proper spans $X_1 \leftarrow N \rightarrow X_2$ carrying graded
  orientations (Definition \ref{def:graded_orientation}) and defining
  $\mathbb{G}_m$-equivariant Lagrangian correspondences between their
  symplectifications, and its 2-morphisms are the monodromic
  $\ell$-adic hypercohomologies \cite{AmorimBassat}
  \[ \mathit{LLeg}_0(N_0, N_1)_2 :=
    \mathbb{H}^\bullet_{\mathrm{mon}}(\widetilde{N}_{01},
  \mathcal{P}_{\widetilde{N}_{01}}[-\vdim \widetilde{N}_0]). \]
\end{theorem}
\begin{proof}
  Write $\widetilde{X}_i$ for the symplectification of $X_i$, a $0$-shifted
  symplectic derived stack carrying a $\Gm$-action scaling its form with weight
  $1$ by Theorem \ref{thm:symplectification}. By
  \eqref{eq:contact_product} the contact product $X_1 \circledast X_2$ is a
  $0$-shifted contact derived stack with symplectification
  \[ B_{12} := \widetilde{X}_1^- \times \widetilde{X}_2 , \]
  and a $1$-morphism $X_1 \leftarrow N \rightarrow X_2$ is a Legendrian
  morphism $N \to X_1 \circledast X_2$, its symplectification
  $\widetilde{N} \to B_{12}$ being the associated $\Gm$-equivariant Lagrangian.
  The construction is that of Theorem \ref{thm:linear_category} with
  $X_1 \circledast X_2$ in place of $L_{01}$. We record the points at which the
  virtual dimensions differ.

  By Remark \ref{rem:graded_equiv} the
  virtual dimension of a $0$-shifted symplectic stack is even. Write
  $\vdim \widetilde{X}_i = 2 a_i$, so $\vdim B_{12} = 2(a_1 + a_2)$. For a
  Lagrangian $\widetilde{N} \to B_{12}$ the fiber sequence
  $\TT_{\widetilde{N}} \to \widetilde{N}^* \TT_{B_{12}} \to
  \LL_{\widetilde{N}}$ has ranks $\vdim \widetilde{N}$ and
  $\vdim \widetilde{N}$ on the outer terms, so
  \[ \vdim \widetilde{N} = \tfrac12 \vdim B_{12} = a_1 + a_2 =: d , \]
  the same for every $1$-morphism between the fixed pair $X_1, X_2$. Hence
  \[\vdim \widetilde{N}_{01} = 2d - \vdim B_{12} = 0,\] so $\widetilde{N}_{01}$
  is $-1$-shifted symplectic as asserted, and the triple intersection
  $\widetilde{N}_{012} = \widetilde{N}_0 \times_{B_{12}} \widetilde{N}_1
  \times_{B_{12}} \widetilde{N}_2$ satisfies
  \[\vdim \widetilde{N}_{012} = 3d - 2 \vdim B_{12} = -d.\]

  Since $X_1 \circledast X_2$ is a $0$-shifted contact derived stack and the
  $\widetilde{N}_i$ are the associated $\Gm$-torsors, Lemma
  \ref{lem:intersection_orientation} applies to it directly. The orientation
  induced on $\widetilde{N}_{01}$ by \cite[Section 5]{AmorimBassat} is
  $\Gm$-equivariant and graded, so $\mathcal{P}_{\widetilde{N}_{01}}$ carries
  the operator $\theta$ of Proposition \ref{prop:theta_glues} and Definition
  \ref{def:Hmon} applies.

  Vertical composition is the argument of Theorem \ref{thm:linear_category}. The
  triple intersection $\widetilde{N}_{012}$ embeds as a proper Lagrangian in
  $\widetilde{N}_{01}^- \times \widetilde{N}_{12}^- \times \widetilde{N}_{02}$
  by \cite[Theorem 2.14]{AmorimBassat}, giving the correspondence
  $\tilde{\pi}_{\mathrm{in}}, \tilde{\pi}_{\mathrm{out}}$. 
  
  The K\"unneth
  formula gives the external product
  $\mathcal{P}_{\widetilde{N}_{01}} \boxtimes
  \mathcal{P}_{\widetilde{N}_{12}}$, the $\theta$-invariant fundamental class of
  (J2) and (J3) contributes the shift attached to the triple intersection, and
  $R\tilde{\pi}_{\mathrm{out},!} \simeq R\tilde{\pi}_{\mathrm{out},*}$ by
  properness preserves cohomological degree. The
  three shifts are
  \begin{align*}
    [-\vdim \widetilde{N}_0 - \vdim \widetilde{N}_1] &= [-2d] , \\
    [-\vdim \widetilde{N}_{012}] &= [d] , \\
    [-2d] + [d] &= [-d] = [-\vdim \widetilde{N}_0] ,
  \end{align*}
  the last of which is the normalization in the statement.

  Horizontal composition is the composition of Lagrangian correspondences. Given
  spans $X_1 \leftarrow N \rightarrow X_2$ and $X_2 \leftarrow M \rightarrow
  X_3$, the symplectifications are Lagrangian in $B_{12}$ and in $B_{23}$, and
  \cite[Corollary 2.13]{AmorimBassat} equips
  $\widetilde{N} \times_{\widetilde{X}_2} \widetilde{M} \to B_{13}$ with a
  Lagrangian structure. The diagonal $\Gm$-actions on $B_{12}$ and $B_{23}$
  restrict to the same action on the shared factor $\widetilde{X}_2$, so the
  fiber product is taken equivariantly and the composite is again a
  $\Gm$-equivariant Lagrangian correspondence, and properness is inherited from
  that of the two spans. The identity $1$-morphism on $X$ is the span with
  symplectification the diagonal Lagrangian
  $\Delta_{\widetilde{X}} \to \widetilde{X}^- \times \widetilde{X}$ of
  \cite[Corollary 2.19]{AmorimBassat}, and the unit laws are
  \cite[Proposition 3.8]{AmorimBassat}.

  For the interchange law of \cite[Lemma 6.6]{AmorimBassat},
  \[(b_2 \odot b_1) * (a_2 \odot a_1) = (-1)^{|a_2||b_1|}
  (b_2 * a_2) \odot (b_1 * a_1),\] the
  pull-push diagrams for horizontal and vertical composition form derived
  Cartesian squares, and derived fiber products require no Tor-independence
  hypothesis, so proper base change shows that
  pushforward along the proper vertical projections commutes with pullback along
  the horizontal embeddings. 
  
  All the operations involved are the six operations
  in the monodromic derived category and commute with $\theta$, which is
  $\theta$-invariant on the classes by (J3), so the resulting natural
  isomorphisms are morphisms of the enrichment of Definition \ref{def:Hmon}. The
  remaining coherence data are those of the non-linear category $Leg_0$ of
  \cite[Corollary 8.1]{IzbudakBerktav2}, whose composition operations descend to
  the contact quotients and satisfy the coherence conditions by \cite[Theorem
  6.1]{IzbudakBerktav2}. The construction above linearizes the $2$-morphism
  spaces and leaves that structure untouched.
\end{proof}
\vspace{1em}

\begin{theorem}[Perverse linearization of $Leg_0$]\label{thm:linearization_functor}
  Assume Setup \ref{setup:joyce}. Let $Leg_0^{\mathrm{pr}}$ denote the
  sub-$2$-category of $Leg_0$ \cite[Corollary 8.1]{IzbudakBerktav2} with the
  same objects, whose $1$- and $2$-morphisms are proper and carry graded
  orientations in the sense of Definition \ref{def:graded_orientation}. No
  orientation of the objects is required, the orientation of
  $\widetilde{N}_{01}$ being supplied by Lemma
  \ref{lem:intersection_orientation}. There is a $2$-functor
  \[ F \colon Leg_0^{\mathrm{pr}} \longrightarrow \mathit{LLeg}_0 \]
  which is the identity on objects and on $1$-morphisms, and which sends a
  $2$-morphism $\varphi \colon P \to N_{01}$ to the image of its fundamental
  class,
  \[ F(P) \in \mathbb{H}^{\,\vdim \widetilde{N}_0 - \vdim \widetilde{P}}_{
  \mathrm{mon}}\bigl( \widetilde{N}_{01},
  \mathcal{P}_{\widetilde{N}_{01}}[-\vdim \widetilde{N}_0] \bigr) . \]
  Each $F(P)$ is $\theta$-invariant, and $F$ is monoidal for the contact product
  \eqref{eq:contact_product}. This is the contact analogue of
  \cite[Theorem 6.11]{AmorimBassat}.
\end{theorem}

Properness is required at two levels, and the two requirements are distinct.
On $2$-morphisms it makes the fundamental class of (J1) pushable. Without it
$R\widetilde{\varphi}_!$ and $R\widetilde{\varphi}_*$ differ and the unit class
does not lie in compactly supported hypercohomology. This is the restriction
imposed in \cite[Section 6]{AmorimBassat}, where the $2$-morphisms of
$\mathfrak{Symp}^{or}_c$ are taken proper for the same reason.

On $1$-morphisms it makes composition in $\mathit{LLeg}_0$ defined. Here $\widetilde{N}_{012} \to \widetilde{N}_{02}$ is base changed from
$\widetilde{N}_1 \to \widetilde{L}_{01}$ and must be proper for
$R\tilde{\pi}_{\mathrm{out},!} \simeq R\tilde{\pi}_{\mathrm{out},*}$. The
hypotheses of Theorems \ref{thm:linear_category} and \ref{thm:global_2cat}
therefore carry properness at the level of $1$-morphisms, and
Theorem \ref{thm:linearization_functor} adds it at the level of
$2$-morphisms.

\begin{proof}
  $Leg_0^{\mathrm{pr}}$ is a sub-$2$-category. Horizontal and vertical
  composition are derived fiber products, and properness is stable under base
  change and under composition, so a composite of proper morphisms is proper. The identities, associators and unitors are graphs of
  Legendreomorphisms and are therefore proper as well. Graded orientations are
  preserved by vertical composition by Lemma \ref{lem:intersection_orientation}
  and by horizontal composition by \cite[Lemma 5.5 and Lemma
  5.7]{AmorimBassat}, applied to the symplectifications. This is the contact
  form of the corresponding statement for $\mathfrak{Symp}^{or}_c$ in
  \cite[Section 6]{AmorimBassat}.

  A $2$-morphism of $Leg_0^{\mathrm{pr}}$ from $N_0$ to $N_1$ is a Legendrian
  $\varphi \colon P \to N_{01}$, where
  $N_{01} = N_0 \times^h_{X_1 \circledast X_2} N_1$ is the $-1$-shifted contact
  intersection of Theorem \ref{thm:review_intersection}. Its symplectification
  $\widetilde{\varphi} \colon \widetilde{P} \to \widetilde{N}_{01}$ is a
  $\Gm$-equivariant Lagrangian in the $-1$-shifted symplectic stack
  $\widetilde{N}_{01}$, proper because $\varphi$ is, and graded oriented by
  hypothesis.

  By (J1) and (J3) it therefore carries a $\theta$-invariant fundamental class
  \[ \mu_{\widetilde{P}} \colon \F_{\widetilde{P}}[\vdim \widetilde{P}]
  \longrightarrow \widetilde{\varphi}^{\,!}
  \mathcal{P}_{\widetilde{N}_{01}} . \]
  Since $\widetilde{\varphi}$ is proper we have
  $R\widetilde{\varphi}_! \simeq R\widetilde{\varphi}_*$, so the adjunction
  $R\widetilde{\varphi}_! \widetilde{\varphi}^{\,!} \to \mathrm{id}$ turns
  $\mu_{\widetilde{P}}$ into
  $R\widetilde{\varphi}_* \F_{\widetilde{P}}[\vdim \widetilde{P}] \to
  \mathcal{P}_{\widetilde{N}_{01}}$, and on hypercohomology into
  \[ \mathbb{H}^{\bullet + \vdim \widetilde{P}}(\widetilde{P}, \F)
  \longrightarrow \mathbb{H}^{\bullet}(\widetilde{N}_{01},
  \mathcal{P}_{\widetilde{N}_{01}}) . \]
  We set $F(P)$ to be the image of the unit $1 \in \mathbb{H}^0(\widetilde{P},
  \F)$, which lies in
  $\mathbb{H}^{-\vdim \widetilde{P}}(\widetilde{N}_{01},
  \mathcal{P}_{\widetilde{N}_{01}})$, that is in degree
  $\vdim \widetilde{N}_0 - \vdim \widetilde{P}$ of the $2$-morphism space of
  Theorem \ref{thm:global_2cat}. All the operations used commute with $\theta$
  by Proposition \ref{prop:theta_glues}, and $\mu_{\widetilde{P}}$ is
  $\theta$-invariant by (J3), so $F(P)$ is fixed by
  $\mathbb{H}^\bullet(\theta)$ and is a morphism of the enrichment of
  Definition \ref{def:Hmon}.

  For functoriality, vertical composition of $2$-morphisms in
  $Leg_0^{\mathrm{pr}}$ is the derived fiber product of the corresponding
  Legendrian spans, and its symplectification is the triple intersection
  $\widetilde{P}_{012}$ of the proof of Theorem \ref{thm:global_2cat}. The
  compatibility of the classes of (J1) with composition of Lagrangian
  correspondences is (J2), refined monodromically by (J3), so
  $F(P_2 \odot P_1) = F(P_2) \odot F(P_1)$ with $\odot$ the pull-push
  composition $\Phi_\mu$ constructed there. Horizontal composition is the same
  argument applied to the composite correspondences of \cite[Corollary
  2.13]{AmorimBassat}, and the interchange law holds on both sides by
  \cite[Theorem 6.1]{IzbudakBerktav2} and by the computation in Theorem
  \ref{thm:global_2cat}. Identities go to identities because the diagonal
  Legendrian has fundamental class the unit, by (J2) applied to
  \cite[Proposition 3.8]{AmorimBassat}.

  Monoidality is Thom--Sebastiani. For contact products the symplectifications
  multiply, and Theorem \ref{thm:TS} together with the K\"unneth formula gives
  $\mathcal{P}_{\widetilde{N} \times \widetilde{N}'} \simeq
  \mathcal{P}_{\widetilde{N}} \boxtimes \mathcal{P}_{\widetilde{N}'}$
  compatibly with the monodromy operators. The fundamental classes are
  multiplicative under external products by (J2), so $F$ carries the contact
  product to the tensor product of the enrichment.
\end{proof}

\section{Contact DT Invariants and the
\texorpdfstring{$\ell$}{l}-adic Behrend Function}
\label{sec:dt}

Let $Z$ be a proper, oriented $-1$-shifted contact derived Artin stack over
$\K$. Invariants extracted through the symplectification alone are
inadequate for
structural reasons. The $\Gm$-action on $\widetilde{Z}$ is free, so
$\chi_{et, c}(\widetilde{Z}) = 0$, and the pushforward
$Rp_* \mathcal{P}_{\widetilde{Z}}$ retains only the generalized
$1$-eigenspace of the monodromy, by the triangle of Lemma
\ref{lem:monodromic_descent}(1). It vanishes outright whenever $1$ is not an
eigenvalue.

This is the case in both examples computed below. The twisted
operator $\theta$ acts by $+1$ there, so by Proposition
\ref{prop:sign_character}
the geometric monodromy $T = -\theta$ acts by $-1$. The trace of the
twisted monodromy operator $\theta$ of Theorem
\ref{thm:monodromic_sheaf} retains the information lost by the pushforward. We
define the \bfem{contact Donaldson--Thomas invariant} as the alternating sum of
its traces on the \'etale hypercohomology of the perverse sheaf $\mathcal{P}_Z$:
\[ \mathrm{DT}(Z) := \sum_{i \in \mathbb{Z}} (-1)^i \mathrm{Tr}\left(
    \theta \mid
\mathbb{H}^i_{et}(Z, \mathcal{P}_Z) \right) \in \overline{\mathbb{Q}}_\ell. \]
In classical Donaldson-Thomas theory over $\mathbb{C}$, the global
hypercohomology invariant is computed by integrating a local
constructible Behrend function \cite[Section 1.2]{Behrend}. We
formulate the $\ell$-adic contact analogue over $\K$.

\begin{definition}\label{def:behrend}
  We define the \bfem{$\ell$-adic contact Behrend function}
  $\nu^{\text{contact}}_Z \colon Z(\K) \to
  \overline{\mathbb{Q}}_\ell$ as the pointwise alternating sum of the
  traces of the operator $\theta$ of Theorem
  \ref{thm:monodromic_sheaf} evaluated at geometric points:
  \[ \nu^{\text{contact}}_Z(x) := \sum_{i \in \mathbb{Z}} (-1)^i
  \mathrm{Tr}\left( \theta_x \mid \mathcal{H}^i(\mathcal{P}_Z)_x \right). \]
\end{definition}

\begin{theorem}[$\ell$-adic Contact Topological
  Integration]\label{thm:integration}
  Let $Z$ be a proper, oriented $-1$-shifted contact derived stack
  over $\K$. The contact DT invariant is equal to the $\ell$-adic
  topological integration of the contact Behrend function:
  \[ \mathrm{DT}(Z) = \chi_{et, c}(Z, \nu^{\text{contact}}_Z) =
    \sum_{c \in \overline{\mathbb{Q}}_\ell} c \cdot \chi_{et, c}\left(
  (\nu^{\text{contact}}_Z)^{-1}(c) \right), \]
  where $\chi_{et, c}$ denotes the stacky \'etale Euler
  characteristic with compact support.
\end{theorem}

\begin{proof}
  Because $Z$ is proper, the global \'etale hypercohomology coincides
  with the compactly supported \'etale hypercohomology
  $\mathbb{H}^\bullet_{et, c}(Z, \mathcal{P}_Z)$.

  Fix a finite stratification of $Z$ by locally closed substacks
  $j_c \colon S_c \hookrightarrow Z$ refining both a stratification on which the
  cohomology sheaves of $\mathcal{P}_Z$ are locally constant and the level sets
  of $\nu^{\mathrm{contact}}_Z$. All strata within one level set carry stalkwise
  trace $c$, so it suffices to treat one stratum per value. Applying the
  alternating sum of the traces of $\theta$ to both sides, the
  evaluation splits as
  a sum over the strata. Because $j_c^* \mathcal{P}_Z$ has locally
  constant cohomology sheaves with a tame action of $\theta$, it
  represents a class in the Grothendieck ring of $\ell$-adic local
  systems equipped with an automorphism. As the compactly supported
  \'etale Euler characteristic $\chi_{et,c}(S_c, -)$ is additive
  under distinguished triangles, we can evaluate it on this
  Grothendieck class. As $\theta$ is an endomorphism of local systems its
  generalized eigenvalues are locally constant, so each cohomology sheaf
  $\mathcal{H}^i(j_c^* \mathcal{P}_Z)$ splits on $S_c$ into the generalized
  eigen-local-systems $L_{i, \lambda}$ of $\theta$. Because $\K$ has
  characteristic zero every $\ell$-adic local system on $S_c$ is tame, so
  $\chi_{et, c}(S_c, L_{i, \lambda}) = \rk(L_{i, \lambda}) \cdot
  \chi_{et, c}(S_c, \overline{\mathbb{Q}}_\ell)$. Summing over $i$ and
  $\lambda$ with signs and eigenvalues, and using that the stalkwise alternating
  trace $\sum_i (-1)^i \sum_\lambda \lambda \rk(L_{i, \lambda})$ is the
  constant $c$, the alternating sum of the traces of $\theta$
  on the compactly supported hypercohomology of $S_c$ with
  coefficients in $j_c^* \mathcal{P}_Z$ evaluates to the scalar
  product $c \cdot \chi_{et, c}(S_c, \overline{\mathbb{Q}}_\ell)$.
  Therefore, the sum reduces to $\sum_{c \in
  \overline{\mathbb{Q}}_\ell} c \cdot \chi_{et, c}\left(
  (\nu^{\text{contact}}_Z)^{-1}(c) \right)$. This establishes the
  topological integration formula over $\K$.
\end{proof}

When $Z$ is in addition Deligne--Mumford, the formula descends to the coarse
moduli space $\pi \colon Z^{\mathrm{cl}} \to \bar{Z}$ of its classical
truncation \cite{KeelMori}. The contact Behrend function is constant on
isomorphism classes of geometric points, so it descends to a constructible
function $\bar{\nu}$ on $\bar{Z}$. The map $\pi$ is proper and induces a
homeomorphism of underlying topological spaces \cite{KeelMori}, so the level
sets of $\nu^{\mathrm{contact}}_Z$ and of $\bar{\nu}$ correspond and
constructibility descends.

By proper base change
the stalk of $R\pi_* \overline{\mathbb{Q}}_\ell$ at a geometric point
$\bar{x}$ is the cohomology of the fiber, a gerbe over the finite group
$\Gamma_{\bar{x}}$. Since $H^{>0}(\Gamma_{\bar{x}},
\overline{\mathbb{Q}}_\ell) = 0$ we get $R\pi_* \overline{\mathbb{Q}}_\ell
\simeq \overline{\mathbb{Q}}_\ell$, and $R\pi_! = R\pi_*$ by properness, so
$\chi_{et,c}$ agrees on corresponding level sets. Substituting into Theorem
\ref{thm:integration},
\[ \mathrm{DT}(Z) = \sum_{c \in \overline{\mathbb{Q}}_\ell} c \cdot
\chi_{et, c}\big( \bar{\nu}^{-1}(c) \big), \]
an Euler characteristic of constructible subsets of the coarse space. Residual
gerbes contribute through $\chi_{et,c}(B\Gamma, \overline{\mathbb{Q}}_\ell)
= 1$ for finite $\Gamma$.

Classical Donaldson--Thomas theory over $\mathbb{C}$ integrates the Behrend
function against the orbifold measure, weighting a point $x$ of the
coarse space by
$1/|\mathrm{Aut}(x)|$ \cite{Behrend}. That measure differs from the
one computed by
$\ell$-adic hypercohomology. A finite gerbe satisfies $\chi_{et,c}(B\Gamma,
\overline{\mathbb{Q}}_\ell) = 1$, not $1/|\Gamma|$. The invariant
$\mathrm{DT}(Z)$
of this paper is the cohomological one. An orbifold-weighted variant
$\mathrm{DT}^{\mathrm{orb}}(Z)$ may be defined by the weighted integral over
$\bar{Z}$. It is not computed by the trace on $\mathbb{H}^\bullet_{et}(Z,
\mathcal{P}_Z)$ and we do not use it.

\subsection{The \texorpdfstring{$A_k$}{Ak}-Singularity}
\label{subsec:Ak}
Consider a local contact chart representing the intersection of two
Legendrians giving an $A_k$-singularity, with potential $s(x) =
x^{k+1}$ on the affine space $\mathbb{A}^1$. The associated
homogeneous potential on the derived symplectification is
$\widetilde{f}(x, t) = t \cdot x^{k+1}$, and its critical locus is the
$t$-axis $x = 0$.

By the proof of Proposition \ref{prop:behrend_one}, the transversal Milnor
fiber of $\widetilde{f}$ at $(0, t_0)$ is the Milnor fiber $F_0$ of $s$, the
$k + 1$ roots of $\epsilon$ of Example \ref{ex:Ak_classical}, and
$T_{U, \widetilde{f}}$ is the inverse of the Milnor monodromy $T_0$, so it acts
through $x \mapsto \zeta_{k+1}^{-1} x$, a cyclic permutation without fixed
points. By that example, $\mathrm{Tr}(T_{U, \widetilde{f}} \mid
\widetilde{H}^0(F_0, \overline{\mathbb{Q}}_\ell)) = -1$, the inversion being
immaterial for a permutation.

Because $s(x) = x^{k+1}$ defines an isolated singularity on a
1-dimensional domain, the perverse vanishing cycles complex
$\mathcal{P}_Z = \phi^p_s(\overline{\mathbb{Q}}_\ell[1]) =
\phi_s(\overline{\mathbb{Q}}_\ell[1])[-1] =
\phi_{s}(\overline{\mathbb{Q}}_\ell)[0]$ is concentrated in degree 0.
By Definition \ref{def:twisted}, $\theta$ acts on the stalk as
$(-1)^{\dim \mathbb{A}^1} T_{U, \widetilde{f}} = -T_{U, \widetilde{f}}$, where
$T_{U, \widetilde{f}}$ denotes the chartwise vanishing-cycle monodromy computed
above. It is related to the global operator of Theorem
\ref{thm:monodromic_sheaf} by $T = \varepsilon_U T_{U, \widetilde{f}}$ with
$\varepsilon_U = (-1)^{n+1} = +1$ here.

Thus
\[ \nu^{\text{contact}}_Z(0)
  = \mathrm{Tr}\bigl(-T_{U, \widetilde{f}} \mid \tilde{H}^0(F_0,
\overline{\mathbb{Q}}_\ell)\bigr) = 1 . \]
The classical Behrend function evaluates to the Milnor number $k$ by
\S\ref{subsec:milnor}, while the $\ell$-adic contact one evaluates to $1$ for
every $k \geq 1$. The
alternating trace of the cyclic monodromy cancels the contribution of the
multi-sheeted branches. The critical locus is the single point $\{0\}$, so
Theorem \ref{thm:integration} gives
\[ \mathrm{DT}(Z) = \nu^{\text{contact}}_Z(0) \cdot \chi_{et,
c}(\{0\}, \overline{\mathbb{Q}}_\ell) = 1 \cdot 1 = 1. \]

\begin{remark}[Non-degenerate quadratic singularity]\label{rem:quadratic_dt}
  For $s = x_1^2 + \dots + x_n^2$ on $\mathbb{A}^n$ the same
  computation gives $1$
  again. The vanishing cohomology of an ordinary quadratic singularity in $n$
  variables has rank one, and by the Picard--Lefschetz formula \cite[Exp.
  XV]{SGA7II} the generator $\gamma$ acts on it by $(-1)^n$, so
  $T_{U, \widetilde{f}} = (-1)^n$. Over $\mathbb{C}$ this is the
  antipodal map on
  the Milnor fiber $F_0 \simeq S^{n-1}$. Since
  $\mathcal{P}_Z = \phi_s(\overline{\mathbb{Q}}_\ell)[n-1]$ has its
  only stalk in
  degree $0$, equal to $\widetilde{H}^{n-1}(F_0)$, and
  $\theta = (-1)^n T_{U, \widetilde{f}}$, we get
  $\nu^{\mathrm{contact}}_Z(0) = (-1)^n (-1)^n = 1$ and
  $\mathrm{DT}(Z) = 1$. The dimensional twist of $\theta$ cancels the parity
  dependence of the geometric monodromy in every dimension.
\end{remark}

\subsection{The Contact Behrend Function is Constant}
\label{subsec:constant}

Both computations returned $1$, and neither is accidental.

\begin{proposition}\label{prop:behrend_one}
  Let $Z$ be an oriented $-1$-shifted contact derived Artin stack over $\K$ and
  $x \in Z(\K)$. Let $(U, s)$ be a contact Darboux chart around $x$
  with $\dim U =
  n$, and let $F_x$ denote the Milnor fiber of $s$ at $x$. Then
  \[ \nu^{\mathrm{contact}}_Z(x)
    = 1 - \Lambda\bigl(T_{U, \widetilde{f}} \mid H^\bullet(F_x,
  \overline{\mathbb{Q}}_\ell)\bigr) = 1 , \]
  where $\Lambda$ denotes the Lefschetz number. If $Z$ is proper, then
  \[ \mathrm{DT}(Z) = \chi_{et, c}(Z^{\mathrm{cl}},
  \overline{\mathbb{Q}}_\ell). \]
\end{proposition}

\begin{proof}
  We first identify the stalks. Near a point $(x, t_0)$ of $\widetilde{U} = U
  \times \Gm$ the fiber $\widetilde{f}^{-1}(\epsilon)$ is
  $\{(x', t') : t'\, s(x') = \epsilon\}$, and $t'$ is invertible, so it is the
  graph of $t' = \epsilon / s(x')$ over $\{s(x') \neq 0\}$. Projection to $U$
  identifies it with the fiber of $s$ over $\epsilon / t_0$. The local Milnor
  fiber of $\widetilde{f}$ at $(x, t_0)$ is $F_x$. By Definition
  \ref{def:twisted}, $\mathcal{P}_{\widetilde{Z}} =
  \phi^p_{\widetilde{f}}(\overline{\mathbb{Q}}_\ell[n+1]) =
  \phi_{\widetilde{f}}(\overline{\mathbb{Q}}_\ell)[n]$, whose stalk
  in degree $i$
  is $\widetilde{H}^{i+n}(F_x)$, and the slice $\mathcal{P}_Z = u^*
  \mathcal{P}_{\widetilde{Z}}[-1]$ of Definition \ref{def:slicepair} has
  \[ \mathcal{H}^i(\mathcal{P}_Z)_x = \widetilde{H}^{i+n-1}(F_x,
  \overline{\mathbb{Q}}_\ell) , \]
  the stalk of $\phi^p_s(\overline{\mathbb{Q}}_\ell[n])$, as used in
  \S\ref{subsec:Ak} and Remark \ref{rem:quadratic_dt}. Under the same
  identification a
  loop in $t$ at fixed $\epsilon$ moves $\epsilon/t$ around $0$ in the direction
  opposite to a loop in $\epsilon$ at fixed $t$, so $T_{U,
  \widetilde{f}}$ is the
  inverse of the Milnor monodromy of $s$ at $x$.

  Now $\theta$ acts as $(-1)^n T_{U, \widetilde{f}}$
  by Definition \ref{def:twisted}. Reindexing by $j = i + n - 1$, the sign
  $(-1)^i = (-1)^{j}(-1)^{n+1}$ combines with the $(-1)^n$ of $\theta$ to give
  \[ \nu^{\mathrm{contact}}_Z(x)
    = -\sum_j (-1)^j \mathrm{Tr}\bigl(T_{U, \widetilde{f}} \mid
    \widetilde{H}^j(F_x)\bigr)
    = - \Lambda\bigl(T_{U, \widetilde{f}} \mid
  \widetilde{H}^\bullet(F_x)\bigr) . \]
  The monodromy acts trivially on the constant summand of $H^0(F_x)$, so
  $\Lambda(T_{U, \widetilde{f}} \mid H^\bullet(F_x)) = 1 +
  \Lambda(T_{U, \widetilde{f}} \mid \widetilde{H}^\bullet(F_x))$,
  which gives the
  first equality.

  A geometric point of $Z$ lies in $\mathrm{Crit}(s) \cap \{s = 0\}$ by Lemma
  \ref{lem:KT_model}, so $s$ has a critical point with critical value
  $0$ at $x$.
  The Lefschetz number of the Milnor monodromy at such a point vanishes
  \cite{ACampo}. The statement is proved there over $\mathbb{C}$ and
  transfers to
  $\K$ by the Lefschetz principle, the Lefschetz number being an
  integer computed
  from an embedded resolution defined over a subfield finitely generated over
  $\mathbb{Q}$. A quasi-unipotent operator and its inverse have the same
  Lefschetz number, so the inversion found above is immaterial. Hence
  $\Lambda = 0$
  and $\nu^{\mathrm{contact}}_Z \equiv 1$.

  The last assertion follows from Theorem \ref{thm:integration}. The only value
  taken is $c = 1$, whose level set is all of $Z(\K)$.
\end{proof}

So the contact invariant carries no Behrend weighting, in contrast with the
symplectic case, where the weight at an $A_k$ point is $k$. The content of
$\mathrm{DT}$ is entirely in the Euler characteristic of the truncation, and the examples above are proper only when their truncation is a point. Here is one
where it is not.

\begin{example}[A contact DT invariant different from $1$]\label{ex:P1}
  Let $C$ be a smooth proper curve of genus $g$, let $F$ be a line
  bundle on $C$,
  and let $V$ be the total space of $E := F \oplus F^\vee$, with $s \colon V \to
  \mathbb{A}^1$ the hyperbolic form $s(u, v) = \langle u, v \rangle$, a regular
  function on $V$. Then $ds$ vanishes exactly on the zero section, where $s$
  vanishes too, so the contact Darboux scheme $Z = \Delta\mathrm{loc}(s)$ of
  Definition \ref{def:discriminant} has $Z^{\mathrm{cl}} \simeq C$
  and is proper,
  and it is canonically oriented by Remark \ref{rem:canonical_orientation}, $s$
  being a function on the smooth scheme $V$. Transversally to the
  zero section $s$ is a non-degenerate
  quadratic form of rank $2$, so Lemma \ref{lem:mb_chart} applies fiberwise and
  $\nu^{\mathrm{contact}}_Z \equiv 1$, as Proposition \ref{prop:behrend_one}
  predicts. Theorem \ref{thm:integration} then gives
  \[ \mathrm{DT}(Z) = \chi_{et, c}(C, \overline{\mathbb{Q}}_\ell) = 2 - 2g , \]
  which is $2$ for $C = \mathbb{P}^1$. Here the integration formula
  does work that
  a single stalk cannot. The support is positive-dimensional and the
  invariant sees
  its topology.
\end{example}

\subsection{Higher Traces}
\label{subsec:highertraces}

Proposition \ref{prop:behrend_one} says that the first trace of $\theta$ is
blind to the singularity type. The higher traces are not. For $m \geq 1$ set
\[ \mathrm{DT}_m(Z) := \sum_{i \in \mathbb{Z}} (-1)^i
  \mathrm{Tr}\bigl( \theta^m \mid \mathbb{H}^i_{et}(Z, \mathcal{P}_Z) \bigr),
  \qquad
  \nu^{(m)}_Z(x) := \sum_{i \in \mathbb{Z}} (-1)^i
\mathrm{Tr}\bigl( \theta^m_x \mid \mathcal{H}^i(\mathcal{P}_Z)_x \bigr), \]
so that $\mathrm{DT}_1 = \mathrm{DT}$ and $\nu^{(1)} = \nu^{\mathrm{contact}}$.
Each $\theta^m$ is again a tame automorphism of $\mathcal{P}_Z$, so the proof of
Theorem \ref{thm:integration} applies verbatim and
$\mathrm{DT}_m(Z) = \chi_{et, c}(Z, \nu^{(m)}_Z)$.

\begin{example}[The higher traces recover $k$]\label{ex:Ak_higher}
  In the notation of \S\ref{subsec:Ak}, the operator $T_{U, \widetilde{f}}$ is a
  cyclic permutation of the $k+1$ points of $F_0$, so $T^m_{U,
  \widetilde{f}}$ is
  the identity when $(k+1) \mid m$ and a fixed-point-free permutation otherwise.
  Hence
  \[ \mathrm{Tr}\bigl(T^m_{U, \widetilde{f}} \mid \widetilde{H}^0(F_0)\bigr)
    =
    \begin{cases} k, & (k+1) \mid m, \\ -1, & \text{otherwise,}
  \end{cases} \]
  and since $\theta = -T_{U, \widetilde{f}}$ on this chart,
  \[ \mathrm{DT}_m(Z) = \nu^{(m)}_Z(0) =
    \begin{cases} (-1)^m\, k, & (k+1) \mid m, \\
      (-1)^{m+1}, & \text{otherwise.}
  \end{cases} \]
  For $m = 1$ this is $1$, as it must be. For $k \geq 2$ the first $m$ at which
  the invariant
  leaves $\{\pm 1\}$ is $m = k+1$, and its value there is $(-1)^{k+1} k$. The
  sequence $(\mathrm{DT}_m)_{m \geq 1}$ determines $k$, and so separates the
  $A_k$-singularities that $\mathrm{DT}$ alone identifies. For $k =
  1$ the sequence
  is constantly $1$, agreeing with Remark \ref{rem:quadratic_dt},
  where the rank-one
  vanishing cohomology carries $\theta = 1$ and $\theta^m = 1$
  for all $m$.
\end{example}

\section{Applications to Conic Lagrangian Intersections}
\label{sec:applications}

The invariants of Section \ref{sec:dt} were computed on charts. This section
applies them to a class of global objects on which the symplectic invariants
vanish identically. A conic Lagrangian in a cotangent bundle, such as a conormal bundle, the
characteristic variety of a $D$-module, or a component of a nilpotent cone, is
the symplectification of a Legendrian in the projectivized cotangent bundle.
The derived intersection of two of them is $(-1)$-shifted symplectic, and its
Behrend-weighted Euler characteristic is the invariant the symplectic theory
attaches to it. We show that this invariant vanishes identically, for a
structural reason, and that the contact invariant of Theorem
\ref{thm:integration} is what remains. We then compute it for
conormal bundles, decide when the orientation data it requires exists, and
assemble the higher traces into a zeta function.

\subsection{Projectivized Cotangent Bundles}
\label{subsec:projcot}

Throughout this section $M$ denotes a smooth variety over $\K$ of dimension
$d$, that is a smooth separated $\K$-scheme of finite type. Its cotangent
bundle $T^*M$ carries the Liouville form $\lambda$ and the symplectic form
$\omega = d_{\DR}\lambda$, which make it a $0$-shifted symplectic derived
scheme \cite{PTVV}. The group $\Gm$ acts on $T^*M$ by scaling the fibers. In
local coordinates $(x_i, \xi_i)$ one has $\lambda = \sum_i \xi_i\, d_{\DR}x_i$,
so $\lambda$, and with it $\omega$, has weight one. Write
\[ T^*M^\circ := T^*M \setminus 0_M \]
for the complement of the zero section, an open $\Gm$-stable subscheme on which
the action is free, and
\[ \mathbb{P}T^*M := [T^*M^\circ / \Gm] \]
for the quotient. As the action is free with a scheme as quotient, this is the
projectivized cotangent bundle $\mathbb{P}(T^*M) \to M$, a smooth variety of
dimension $2d - 1$ fibered in $\mathbb{P}^{d-1}$ over $M$. By
\cite[Theorem 3.7]{IzbudakBerktav2026}, applied to the weight-one form $\omega$ on
$T^*M^\circ$, the quotient $\mathbb{P}T^*M$ carries a $0$-shifted contact
structure whose symplectification is the tautological torsor $T^*M^\circ \to
\mathbb{P}T^*M$. The contact form is the tautological $1$-form, which at a
point $[\xi]$ sends a tangent vector $v$ to $\langle \xi, d\pi(v) \rangle$.
This is linear in the representative $\xi$, so it takes values in the dual of
the tautological subbundle $\cO(-1)$, and the contact line bundle of
$\mathbb{P}T^*M$ is $\cL = \cO_{\mathbb{P}(T^*M)}(1).$
The symplectification, the bundle of contact forms of \cite[Definition
4.3]{kib1}, is the frame bundle of $\cO(-1) = \cL^\vee$, namely
$T^*M^\circ$ with $\Gm$ scaling covectors. The restriction of $\cL$ to a fiber
$\mathbb{P}^{d-1}$ is $\cO(1)$, which is nontrivial for $d \geq 2$. This is confirmed
in Step 1 of the proof of Proposition \ref{prop:conormal_orientation}, where
the fiber sequence of Definition \ref{def:legendrian} gives $K_{\mathbb{P}T^*M}
= \cL^{\otimes(-d)}$ along any Legendrian of virtual dimension $d - 1$, in
agreement with $K_{\mathbb{P}(T^*M)} = \cO(-d)$. This is the non-exact case, outside the construction of Remark
\ref{rem:exact_case}, and it is the case in which the descent of Section
\ref{sec:perverse} carries content.

By a \bfem{conic Lagrangian} in $T^*M^\circ$ we mean a $\Gm$-equivariant
morphism $\ell \colon \Lambda \to T^*M^\circ$ from a derived scheme with a
$\Gm$-action, equipped with a $\Gm$-equivariant Lagrangian structure. A smooth
closed $\Gm$-stable subvariety $\Lambda \subseteq T^*M^\circ$ which is
Lagrangian in the classical sense is a conic Lagrangian in this sense. A closed
immersion of smooth schemes into a symplectic smooth scheme carries a
Lagrangian structure if and only if it is Lagrangian in the classical sense,
and the structure is then unique \cite[Introduction and \S 3]{PTVV}. The
$\Gm$-action preserves $\omega|_\Lambda$ and therefore, by uniqueness, the
Lagrangian structure.

\begin{lemma}[Conic Lagrangians are symplectifications of Legendrians]
\label{lem:conic_legendrian}
  Let $\ell \colon \Lambda \to T^*M^\circ$ be a conic Lagrangian. Then $\Gm$
  acts freely on $\Lambda$, and for the induced morphism
  \[ \mathbb{P}\ell \colon \mathbb{P}\Lambda := [\Lambda / \Gm]
    \longrightarrow \mathbb{P}T^*M \]
  the map $\Lambda \to \mathbb{P}\Lambda$ is the symplectification of
  $\mathbb{P}\Lambda$, so that
  \[ \Lambda \simeq \mathbb{P}\Lambda \times_{\mathbb{P}T^*M} T^*M^\circ , \]
  and this symplectification carries the $\Gm$-equivariant Lagrangian
  structure of $\ell$. In particular the conclusion of Theorem
  \ref{thm:review_intersection} holds for any two such morphisms, and we call
  them Legendrian.
\end{lemma}

\begin{proof}
  The action of $\Gm$ on $\Lambda$ is the restriction along $\ell$ of the
  action on $T^*M^\circ$, and the stabilizer of a point of $\Lambda$ is
  contained in the stabilizer of its image, which is trivial. Hence the action
  is free, $\Lambda \to [\Lambda/\Gm]$ is a principal $\Gm$-bundle, and $\ell$
  descends to $\mathbb{P}\ell$. The square
  \[
    \begin{tikzcd}
      \Lambda \arrow[r, "\ell"] \arrow[d] & T^*M^\circ \arrow[d] \\
      \mathbb{P}\Lambda \arrow[r, "\mathbb{P}\ell"'] & \mathbb{P}T^*M
    \end{tikzcd}
  \]
  is Cartesian. Indeed both vertical maps are principal $\Gm$-bundles, so
  $\Lambda \to \mathbb{P}\Lambda \times_{\mathbb{P}T^*M} T^*M^\circ$ is an
  equivariant morphism between principal $\Gm$-bundles over
  $\mathbb{P}\Lambda$, and every such morphism is an isomorphism. Hence
  $\Lambda$ is the symplectification of $\mathbb{P}\Lambda$ in the sense of
  Theorem \ref{thm:symplectification}, and it carries the equivariant
  Lagrangian structure of $\ell$ by hypothesis. The proof of \cite[Theorem
  4.1]{IzbudakBerktav2026} uses the Legendrian structures on its inputs only
  through the $\Gm$-equivariant Lagrangian structures on their
  symplectifications, which its first step produces, the remaining steps
  working upstairs and descending by Theorem \ref{thm:symplectification}.
  Its conclusion therefore holds for morphisms whose symplectifications carry
  such structures.
\end{proof}

\begin{remark}[Conic Lagrangians are Legendrian]\label{rem:conic_legendrian}
  The morphism $\mathbb{P}\ell$ is Legendrian in the sense of Definition
  \ref{def:legendrian} as well. The Liouville form satisfies $\lambda =
  \iota_E \omega$ for the Euler vector field $E$ generating the $\Gm$-action,
  and $E$ is tangent to $\Lambda$ because $\ell$ is equivariant. Contracting
  the null-homotopy $\ell^*\omega \simeq 0$ of the Lagrangian structure with
  $E$ gives a null-homotopy $\ell^*\lambda \simeq 0$, which is
  $\Gm$-equivariant and descends to an isotropic structure
  $(\mathbb{P}\ell)^*\alpha \simeq 0$ for the contact form. The fiber sequence
  of Definition \ref{def:legendrian} is the descent along $\Lambda \to
  \mathbb{P}\Lambda$ of the equivalence $\TT_\Lambda \simeq
  \LL_{\Lambda/T^*M^\circ}[-1]$ expressing the non-degeneracy of the
  Lagrangian structure, by the duality argument recorded after
  \cite[Definition 2.39]{IzbudakBerktav2026}. This is the classical fact that a
  conic Lagrangian in a cotangent bundle lies in the zero locus of the
  Liouville form. We do not use it below.
\end{remark}

\subsection{Vanishing of the Symplectic Invariants}
\label{subsec:symp_vanishing}

The symplectification of a $-1$-shifted contact stack is never proper, and
its compactly supported Euler characteristic vanishes, weighted or not.

\begin{lemma}[Invariant Euler characteristics vanish on the symplectification]
\label{lem:torsor_vanishing}
  Let $Z$ be a $-1$-shifted contact derived Artin stack of finite type with
  symplectification $p \colon \widetilde{Z} \to Z$, and let
  $\nu \colon \widetilde{Z}^{\mathrm{cl}}(\K) \to \overline{\mathbb{Q}}_\ell$
  be a constructible function invariant under the $\Gm$-action. Then
  \[ \chi_{et, c}\bigl( \widetilde{Z}^{\mathrm{cl}}, \nu \bigr) :=
    \sum_{c \in \overline{\mathbb{Q}}_\ell} c \cdot
    \chi_{et, c}\bigl( \nu^{-1}(c) \bigr) = 0 . \]
  In particular $\chi_{et, c}(\widetilde{Z}^{\mathrm{cl}}) = 0$, and the
  Behrend-weighted Euler characteristic
  $\chi_{et, c}(\widetilde{Z}^{\mathrm{cl}}, \nu_{\widetilde{Z}})$ vanishes,
  where $\nu_{\widetilde{Z}}$ is the Behrend function of the Artin stack
  $\widetilde{Z}^{\mathrm{cl}}$ \cite[Proposition 4.4]{JoyceSong}.
\end{lemma}

\begin{proof}
  The proof proceeds in three steps. We descend $\nu$ to $Z$, reduce to a
  single principal bundle, and compute its Euler characteristic.

  \vspace{5pt}
  \textbf{\textit{Step 1: Descent of the function.}}
  The map $p$ is a principal $\Gm$-bundle, so the geometric points of
  $Z^{\mathrm{cl}}$ are the $\Gm$-orbits of geometric points of
  $\widetilde{Z}^{\mathrm{cl}}$, and a $\Gm$-invariant function $\nu$ is the
  pullback $\bar{\nu} \circ p$ of a function $\bar{\nu}$ on
  $Z^{\mathrm{cl}}(\K)$. The level set $\bar{\nu}^{-1}(c)$ is the image under
  $p$ of the constructible set $\nu^{-1}(c)$, hence is constructible by
  Chevalley's theorem, so $\bar{\nu}$ is constructible. Choose a finite
  stratification of $Z^{\mathrm{cl}}$ by locally closed substacks $T_\alpha$ on
  each of which $\bar{\nu}$ is constant, with value $c_\alpha$. Then
  \[ \nu^{-1}(c) = p^{-1}\bigl( \bar{\nu}^{-1}(c) \bigr) =
    \bigsqcup_{c_\alpha = c} p^{-1}(T_\alpha) , \]
  and additivity of $\chi_{et, c}$ over stratifications gives
  \[ \chi_{et, c}\bigl( \widetilde{Z}^{\mathrm{cl}}, \nu \bigr) = \sum_\alpha
    c_\alpha \cdot \chi_{et, c}\bigl( p^{-1}(T_\alpha) \bigr) . \]

  \textbf{\textit{Step 2: Reduction to a principal bundle.}}
  The restriction $q \colon S := p^{-1}(T_\alpha) \to T := T_\alpha$ of $p$ is
  a principal $\Gm$-bundle over a finite type Artin stack, being the pullback
  of one. It suffices to show $\chi_{et, c}(S) = 0$ for every such $q$.

  \vspace{5pt}
  \textbf{\textit{Step 3: The Euler characteristic of a principal bundle.}}
  By proper base change the stalk of $Rq_! \overline{\mathbb{Q}}_\ell$ at a
  geometric point $t$ of $T$ is the compactly supported cohomology of the fiber
  $q^{-1}(t) \simeq \Gm$, namely $H^1_c(\Gm, \overline{\mathbb{Q}}_\ell) = \overline{\mathbb{Q}}_\ell$, $H^2_c(\Gm, \overline{\mathbb{Q}}_\ell) = \overline{\mathbb{Q}}_\ell(-1)$ and $H^\bullet_c(\Gm, \overline{\mathbb{Q}}_\ell)=0$ in every other degree. Let $U \to T$ be a smooth atlas. The pullback
  of $q$ to $U$ is a principal $\Gm$-bundle over a scheme, hence Zariski
  locally trivial by Hilbert's Theorem 90, so over an open cover of $U$ the
  pullback of $Rq_! \overline{\mathbb{Q}}_\ell$ is the pullback of
  $R\Gamma_c(\Gm, \overline{\mathbb{Q}}_\ell)$ by proper base change and the
  K\"unneth formula. Lisseness of a sheaf on $T$ may be checked on a smooth
  atlas, so the cohomology sheaves $R^1 q_! \overline{\mathbb{Q}}_\ell$ and
  $R^2 q_! \overline{\mathbb{Q}}_\ell$ are lisse of rank one. Every lisse sheaf
  on $T$ is tame because $\K$ has characteristic zero, so
  $\chi_{et, c}(T, L) = \rk(L) \cdot \chi_{et, c}(T)$ for lisse $L$, as in the
  proof of Theorem \ref{thm:integration}. Additivity of the Euler
  characteristic over the triangles of the canonical filtration of
  $Rq_! \overline{\mathbb{Q}}_\ell$ then gives
  \[ \chi_{et, c}(S) = \sum_{i} (-1)^i \chi_{et, c}\bigl( T, R^i q_!
    \overline{\mathbb{Q}}_\ell \bigr) = - \chi_{et, c}(T) + \chi_{et, c}(T) =
    0 . \]
    
  For the last assertion, the Behrend function of an Artin stack $X$ locally
  of finite type is characterized by the property that $\varphi^*(\nu_X) =
  (-1)^n \nu_W$ for every smooth morphism $\varphi \colon W \to X$ of
  relative dimension $n$ from a finite type scheme
  \cite[Proposition 4.4]{JoyceSong}. For an automorphism $g$ of $X$ the
  function $g^*\nu_X$ satisfies the same property, since $\varphi$ and
  $g \circ \varphi$ are smooth of the same relative dimension, so
  $g^*\nu_X = \nu_X$ by uniqueness. Applied to the automorphisms given by
  the elements of $\Gm(\K)$, this shows that $\nu_{\widetilde{Z}}$ is constant
  on orbits.
\end{proof}

\begin{theorem}[Conic Lagrangian intersections]\label{thm:conic}
  Let $\ell_1 \colon \Lambda_1 \to T^*M^\circ$ and $\ell_2 \colon \Lambda_2
  \to T^*M^\circ$ be conic Lagrangians, and write
  \[ Z := \mathbb{P}\Lambda_1 \times^h_{\mathbb{P}T^*M} \mathbb{P}\Lambda_2,
     \qquad
     \widetilde{Z} := \Lambda_1 \times^h_{T^*M^\circ} \Lambda_2 . \]
  \begin{enumerate}
    \item $Z$ is a $-1$-shifted contact derived stack whose symplectification
      is $\widetilde{Z}$, and $\widetilde{Z}$ is the $-1$-shifted symplectic
      derived intersection of the Lagrangians $\Lambda_1$ and $\Lambda_2$. Its
      contact line bundle is the restriction of $\cL = \cO(1)$ along $Z \to
      \mathbb{P}T^*M$.
    \item If $\widetilde{Z}$ is of finite type, then every $\Gm$-invariant
      constructible function on $\widetilde{Z}^{\mathrm{cl}}$ integrates to
      zero. In particular the Euler characteristic and the Behrend-weighted
      Euler characteristic of the symplectic intersection vanish.
    \item If $Z$ is proper and admits contact orientation data, then
      \[ \mathrm{DT}(Z) = \chi_{et, c}\bigl( \mathbb{P}\Lambda_1^{\mathrm{cl}}
        \times_{\mathbb{P}T^*M} \mathbb{P}\Lambda_2^{\mathrm{cl}} \bigr) , \]
      the compactly supported Euler characteristic of the classical
      intersection of the projectivized Lagrangians.
  \end{enumerate}
\end{theorem}

\begin{proof}
  By Lemma \ref{lem:conic_legendrian} the morphisms $\mathbb{P}\ell_i$ are
  Legendrian, so $Z$ is $-1$-shifted contact by Theorem
  \ref{thm:review_intersection}. Its symplectification is the pullback of the
  torsor $T^*M^\circ \to \mathbb{P}T^*M$ along $Z \to \mathbb{P}T^*M$, and
  fiber products commute with one another, so
  \begin{align*}
    Z \times_{\mathbb{P}T^*M} T^*M^\circ
    &\simeq \bigl( \mathbb{P}\Lambda_1 \times_{\mathbb{P}T^*M} T^*M^\circ \bigr)
      \times_{T^*M^\circ}
      \bigl( \mathbb{P}\Lambda_2 \times_{\mathbb{P}T^*M} T^*M^\circ \bigr) \\
    &\simeq \Lambda_1 \times_{T^*M^\circ} \Lambda_2 = \widetilde{Z} ,
  \end{align*}
  using Lemma \ref{lem:conic_legendrian} for the second line. The derived
  intersection of two Lagrangians in the $0$-shifted symplectic
  $T^*M^\circ$ is $-1$-shifted symplectic by the recollection of
  \S\ref{subsec:symplectic}. For the contact line bundle, the display above
  exhibits $\widetilde{Z} \to Z$ as the pullback of the torsor $T^*M^\circ \to
  \mathbb{P}T^*M$ along $Z \to \mathbb{P}T^*M$. The contact line bundle of a
  $-1$-shifted contact stack is the line bundle associated to its
  symplectification torsor \cite[Definition 4.3]{kib1}, as recalled in
  Proposition \ref{prop:orientation_obstruction}, and the line bundle
  associated to a pullback torsor is the pullback of the associated line
  bundle. Hence the contact line bundle of $Z$ is $\cL|_Z$, which proves
  (1). Assertion (2) is Lemma
  \ref{lem:torsor_vanishing} applied to $Z$. For (3), Theorem
  \ref{thm:integration} and Proposition \ref{prop:behrend_one} give
  $\mathrm{DT}(Z) = \chi_{et, c}(Z^{\mathrm{cl}})$, and the classical
  truncation of a derived fiber product is the classical fiber product of the
  truncations.
\end{proof}

\begin{remark}[Microsupports and the index formula]\label{rem:microlocal}
  The microsupport of a constructible sheaf $F$ on $M$ is a closed conic
  Lagrangian in $T^*M$ \cite[Theorem 8.4.2]{KS}, and so is the characteristic
  variety of a holonomic $D$-module. The Euler characteristic of $F$ is
  computed by the index formula \cite[Corollary 9.5.2]{KS} as the intersection
  number of the characteristic cycle $CC(F)$ \cite[Definition 9.4.1]{KS} with
  the zero section of $T^*M$. The projectivization $T^*M^\circ \to
  \mathbb{P}T^*M$ removes the zero section. Theorem \ref{thm:conic} therefore
  attaches to a pair of microsupports, once they are given $\Gm$-equivariant
  Lagrangian structures, an invariant supported on the part of their
  intersection that the index formula does not see, and on which the
  symplectic Euler characteristic vanishes.
\end{remark}

\begin{remark}[The contact invariant is not a virtual count]
\label{rem:not_virtual}
  The symplectic invariant of a proper $-1$-shifted symplectic stack is
  invariant under deformation, because the Behrend weighting makes it a virtual
  count. The contact invariant is not. By Proposition \ref{prop:behrend_one}
  it is the Euler characteristic of the classical truncation, and Euler
  characteristics jump in families. The two invariants count different things,
  and the following family shows the difference. 
  
  For $s_t = (x^2 - t)^2$ on
  $U = \mathbb{A}^1$, the symplectic invariant of $\dCrit(s_t)$ is the sum of
  the Milnor numbers of the critical points of $s_t$, the Behrend function
  being the Milnor number at an isolated critical point by
  \S\ref{subsec:milnor}. For $t \neq 0$ these are
  $x = 0$ and $x = \pm\sqrt{t}$, all nondegenerate, and for $t = 0$ the single
  point $x = 0$ with $s_0 = x^4$, so the invariant is $3$ for every $t$. The
  contact invariant of $\Delta\mathrm{loc}(s_t)$ counts the critical points on
  the zero level, each with weight $1$ by Theorem \ref{thm:D}. For $t \neq 0$
  these are the two points $x = \pm\sqrt{t}$, and for $t = 0$ the single
  $A_3$-point, so the invariant is $2$ for $t \neq 0$ and $1$ for $t = 0$. The
  critical point at $x = 0$ leaves the zero level as $t$ moves away from $0$,
  and the contact invariant records it. The information the symplectic
  theory encodes in the Behrend function is carried in the contact theory by
  the operator $\theta$ and its higher traces, which distinguish the
  $A_3$-point at $t = 0$ from a nondegenerate one by Example
  \ref{ex:Ak_higher}.
\end{remark}

\subsection{Conormal Bundles}
\label{subsec:conormal}

The conormal bundle of a smooth subvariety is the basic conic Lagrangian, and
intersections of conormal bundles are the objects of microlocal geometry. We
compute their contact invariants.

Let $Y \subseteq M$ be a smooth closed subvariety of codimension $c$. Its
conormal bundle \[N^*Y := \ker\bigl( T^*M|_Y \to T^*Y \bigr)\] is a subbundle of
rank $c$ of $T^*M|_Y$, and its total space is a smooth closed subvariety of
$T^*M$ of dimension $d$, Lagrangian for $\omega$ and stable under $\Gm$. Write
\[N^*Y^\circ := N^*Y \cap T^*M^\circ\] for the complement of the zero section.
By Lemma \ref{lem:conic_legendrian} the projectivization $\mathbb{P}N^*Y =
[N^*Y^\circ / \Gm]$, a $\mathbb{P}^{c-1}$-bundle over $Y$, is a Legendrian in
$\mathbb{P}T^*M$ with symplectification $N^*Y^\circ$. Its virtual dimension is
$\dim Y + c - 1 = d - 1$.

Two smooth closed subvarieties $Y_1, Y_2 \subseteq M$ \bfem{meet cleanly} if
their scheme-theoretic intersection $W := Y_1 \times_M Y_2$ is smooth and
$T_W = T_{Y_1}|_W \cap T_{Y_2}|_W$ inside $T_M|_W$. In that case
$T_{Y_1}|_W + T_{Y_2}|_W$ is a subbundle of $T_M|_W$ of rank
$\dim Y_1 + \dim Y_2 - \dim W$, and the \bfem{excess} of the intersection is
\[ e := d - \dim Y_1 - \dim Y_2 + \dim W , \]
the corank of that subbundle. The intersection is transverse exactly when
$e = 0$.

\begin{proposition}[Contact invariants of conormal intersections]
\label{prop:conormal}
  Let $M$ be a smooth proper variety of dimension $d$, let $Y_1, Y_2 \subseteq
  M$ be smooth closed subvarieties of codimensions $c_1, c_2$ meeting cleanly
  along $W$ with excess $e$, and let
  \[ Z := \mathbb{P}N^*Y_1 \times^h_{\mathbb{P}T^*M} \mathbb{P}N^*Y_2 . \]
  Then $Z$ is a proper $-1$-shifted contact derived scheme, its classical
  truncation is the projective bundle $\mathbb{P}(E) \to W$ of the excess
  bundle
  \[ E := N^*Y_1|_W \cap N^*Y_2|_W \subseteq T^*M|_W , \]
  a vector bundle of rank $e$, and the symplectic invariants of Theorem
  \ref{thm:conic}(2) vanish. If $Z$ admits contact orientation data, then
  \[ \mathrm{DT}(Z) = e \cdot \chi_{et, c}(W) . \]
  In particular $\mathrm{DT}(Z) = 0$ for a transverse intersection, and for
  the self-intersection $Y_1 = Y_2 = Y$ of a subvariety of codimension $c$
  \[ \mathrm{DT}\bigl( \mathbb{P}N^*Y \times^h_{\mathbb{P}T^*M}
    \mathbb{P}N^*Y \bigr) = c \cdot \chi_{et, c}(Y) . \]
\end{proposition}

\begin{proof}
  The proof proceeds in three steps. We identify the excess bundle, compute
  the classical truncation, and evaluate the invariant.

  \vspace{5pt}
  \textbf{\textit{Step 1: The excess bundle.}}
  At a point $w \in W$, a covector $\xi \in T^*_w M$ lies in $N^*_w Y_1 \cap
  N^*_w Y_2$ if and only if it annihilates both $T_w Y_1$ and $T_w Y_2$, that
  is if and only if it annihilates $T_w Y_1 + T_w Y_2$. Hence $E$ is the
  annihilator of the subbundle $T_{Y_1}|_W + T_{Y_2}|_W$ of $T_M|_W$, whose
  rank is $\dim Y_1 + \dim Y_2 - \dim W$ by cleanness, so $E$ is a subbundle of
  $T^*M|_W$ of rank $e$. In particular the two subbundles $N^*Y_1|_W$ and
  $N^*Y_2|_W$ of $T^*M|_W$ have an intersection of constant rank $e$, and their
  sum has constant rank $c_1 + c_2 - e$, so both are subbundles.

  \vspace{5pt}
  \textbf{\textit{Step 2: The classical truncation.}}
  The classical truncation of a derived fiber product is the classical fiber
  product of the truncations, so $Z^{\mathrm{cl}} = \mathbb{P}N^*Y_1
  \times_{\mathbb{P}T^*M} \mathbb{P}N^*Y_2$. We compute it in two stages. The
  projection $\mathbb{P}T^*M \to M$ restricts to $\mathbb{P}N^*Y_i \to Y_i$,
  so the fiber product maps to $Y_1 \times_M Y_2 = W$, and over $W$ it is the
  fiber product of $\mathbb{P}(N^*Y_1|_W)$ and $\mathbb{P}(N^*Y_2|_W)$ over
  $\mathbb{P}(T^*M|_W)$. Let $A, B$ be subbundles of a vector bundle $V$ over
  $W$ such that $A \cap B$ and $A + B$ are subbundles. Then the
  scheme-theoretic intersection of $\mathbb{P}(A)$ and $\mathbb{P}(B)$ inside
  $\mathbb{P}(V)$ is $\mathbb{P}(A \cap B)$. This is local on $W$, and locally
  $V$ admits a trivialization adapted to the flag $A \cap B \subseteq A, B
  \subseteq A + B \subseteq V$, in which $A$ and $B$ are constant subspaces.
  The statement then reduces to the fact that two linear subspaces of
  projective space, each cut out by linear equations, intersect
  scheme-theoretically in the projectivization of the intersection of the
  corresponding vector subspaces, which is reduced. By Step 1 we may apply this
  with $A = N^*Y_1|_W$, $B = N^*Y_2|_W$ and $V = T^*M|_W$, and we obtain
  \[ Z^{\mathrm{cl}} = \mathbb{P}(E) \longrightarrow W , \]
  a Zariski locally trivial fibration with fiber $\mathbb{P}^{e-1}$, which is
  empty when $e = 0$. It is proper, as $W$ is closed in the proper $M$ and
  $\mathbb{P}(E) \to W$ is projective. A derived stack is proper when its
  classical truncation is, so $Z$ is proper. It is a derived scheme, being a
  fiber product of derived schemes.

  \vspace{5pt}
  \textbf{\textit{Step 3: The invariant.}}
  Theorem \ref{thm:conic} applies, its conic Lagrangians being the inclusions
  of $N^*Y_1^\circ$ and $N^*Y_2^\circ$. Assertion (2) gives the vanishing of the
  symplectic invariants. For (3) we need $\chi_{et, c}(\mathbb{P}(E))$. The
  fibration $\mathbb{P}(E) \to W$ is Zariski locally trivial, so the argument
  of Step 3 of Lemma \ref{lem:torsor_vanishing}, with the fiber
  $\mathbb{P}^{e-1}$ in place of $\Gm$, shows that its pushforward has lisse
  cohomology sheaves whose ranks are those of $H^\bullet(\mathbb{P}^{e-1})$,
  and
  \[ \chi_{et, c}(\mathbb{P}(E)) = \chi(\mathbb{P}^{e-1}) \cdot
    \chi_{et, c}(W) = e \cdot \chi_{et, c}(W) . \]
  The final assertions are the cases $e = 0$ and $Y_1 = Y_2$, for which
  $W = Y$ and $E = N^*Y$ has rank $e = c$.
\end{proof}

For a curve $Y = C$ of genus $g$ on a surface, the self-intersection formula
gives $\chi_{et, c}(C) = 2 - 2g$, the value of Example \ref{ex:P1}. The chart
of that example is the local model of the situation, its contact line bundle
being trivial, whereas here the contact line bundle of $Z$ is the restriction
of $\cO(1)$, which for $c = 1$ is the normal bundle $N_C$.

The orientation hypothesis in Proposition \ref{prop:conormal} is not
automatic, and the twist of Definition \ref{def:orientation} decides it.

\begin{proposition}[Orientability of conormal self-intersections]
\label{prop:conormal_orientation}
  Let $Y \subseteq M$ be a smooth closed subvariety of codimension $c$ in a
  smooth variety of dimension $d$, and let $Z$ be the derived self-intersection
  of $\mathbb{P}N^*Y$ in $\mathbb{P}T^*M$. Then the canonical bundle of $Z$ is
  \[ K_Z \simeq \bigl( K_Y \otimes \det N_Y \bigr)^{\otimes 2}\big|_Z \otimes
    \cL^{\otimes (d - 2c)} , \]
  where $N_Y$ is the normal bundle and $\cL = \cO(1)|_Z$ is the contact line
  bundle. Consequently
  \begin{enumerate}
    \item if $d$ is even, $Z$ admits contact orientation data, with the square
      root $(K_Y \otimes \det N_Y)|_Z \otimes \cL^{\otimes (d/2 - c)}$ of
      $K_Z$,
    \item if $d$ is odd and $c \geq 2$, $Z$ admits no contact orientation data,
    \item if $d$ is odd and $c = 1$, so that $Z^{\mathrm{cl}} = Y$ and $\cL =
      N_Y$, contact orientation data exists if and only if $N_Y$ is a square
      in $\mathrm{Pic}(Y)$.
  \end{enumerate}
\end{proposition}

\begin{proof}
  The proof proceeds in two steps. We compute $K_Z$ from the fiber product,
  and read off the square roots.

  \vspace{5pt}
  \textbf{\textit{Step 1: The canonical bundle.}}
  Write $L := \mathbb{P}N^*Y$, $X := \mathbb{P}T^*M$ and $f \colon L \to X$
  for the Legendrian, so that $Z = L \times^h_X L$. The cotangent complex of a
  derived fiber product sits in the cofiber sequence
  \[ f^*\LL_X|_Z \longrightarrow \LL_L|_Z \oplus \LL_L|_Z \longrightarrow
    \LL_Z , \]
  so taking determinants
  \[ K_Z \simeq K_L^{\otimes 2}|_Z \otimes f^*K_X^{-1}|_Z . \]
  The two factors are computed separately. For the first, $L \to Y$ is the
  projective bundle of the rank $c$ bundle $N^*Y$, and the relative Euler
  sequence
  \[ 0 \longrightarrow \cO_L \longrightarrow \pi^*N^*Y \otimes \cO(1)
    \longrightarrow \TT_{L/Y} \longrightarrow 0 \]
  gives $\det \TT_{L/Y} = \pi^*\det(N^*Y) \otimes \cO(c)$, hence
  \[ K_L = \pi^*K_Y \otimes \det(\TT_{L/Y})^{-1} = \pi^*\bigl( K_Y \otimes
    \det N_Y \bigr) \otimes \cO(-c) . \]
  For the second, the fiber sequence of Definition \ref{def:legendrian} for
  the Legendrian $f$ in the $0$-shifted contact stack $X$ reads
  \[ \cO_L \longrightarrow \LL_{L/X} \otimes f^*\cL[-1] \longrightarrow
    \TT_L , \]
  and taking determinants, with $r := \rk \LL_{L/X} = \vdim L - \vdim X = -d$
  and $\det \LL_{L/X} = K_L \otimes f^*K_X^{-1}$,
  \[ \bigl( K_L \otimes f^*K_X^{-1} \otimes f^*\cL^{\otimes r}
    \bigr)^{-1} = K_L^{-1} , \]
  that is $f^*K_X \simeq f^*\cL^{\otimes r} = \cL^{\otimes(-d)}|_L$. As a
  check, $K_X = \cO(-d)$ by the Euler sequence for $\mathbb{P}(T^*M)$ and
  $\cL = \cO(1)$, in agreement. Substituting,
  \[ K_Z \simeq \pi^*\bigl( K_Y \otimes \det N_Y \bigr)^{\otimes 2}\big|_Z
    \otimes \cO(-2c)|_Z \otimes \cO(d)|_Z , \]
  which is the displayed formula, $\cO(1)|_Z$ being $\cL$.

  \vspace{5pt}
  \textbf{\textit{Step 2: Square roots.}}
  If $d$ is even then $d - 2c$ is even, and the bundle in (1) is a square root
  of $K_Z$, which proves (1). Suppose $d$ is odd and $c \geq 2$, and let $R$ be
  a square root of $K_Z$. The fiber of $Z^{\mathrm{cl}} = \mathbb{P}N^*Y \to
  Y$ over a point of $Y$ is $F \simeq \mathbb{P}^{c-1}$, on which $K_Y \otimes
  \det N_Y$ is trivial and $\cL$ restricts to $\cO(1)$, so $K_Z|_F \simeq
  \cO(d - 2c)$ with $d - 2c$ odd, and $R|_F$ would be a square root of an odd
  power of $\cO(1)$ on $\mathbb{P}^{c-1}$. As $\mathrm{Pic}(\mathbb{P}^{c-1})
  = \mathbb{Z} \cdot \cO(1)$ for $c \geq 2$, no such $R$ exists, which proves
  (2). If $c = 1$ then $\mathbb{P}N^*Y = Y$, the tautological subbundle is
  $N^*Y$ and $\cL = (N^*Y)^\vee = N_Y$, so
  \[ K_Z \simeq \bigl( K_Y \otimes N_Y \bigr)^{\otimes 2}\big|_Z \otimes
    N_Y^{\otimes (d - 2)} \]
  with $d - 2$ odd. If $N_Y \simeq A^{\otimes 2}$ in $\mathrm{Pic}(Y)$, then
  $N_Y^{\otimes (d-2)} = (A^{\otimes (d-2)})^{\otimes 2}$ and $K_Z$ has the
  square root $(K_Y \otimes N_Y)|_Z \otimes A^{\otimes (d-2)}|_Z$. Conversely a
  square root of $K_Z$ restricts to $Z^{\mathrm{cl}} = Y$, where it is a square
  root of $(K_Y \otimes N_Y)^{\otimes 2} \otimes N_Y^{\otimes (d-2)}$, so
  $N_Y^{\otimes (d-2)}$ is a square in $\mathrm{Pic}(Y)$, and since $d - 2$ is
  odd so is $N_Y$. This proves (3).
\end{proof}

\begin{remark}[The equivariant computation]\label{rem:equivariant_check}
  The formula for $K_Z$ can also be obtained upstairs. On $T^*M^\circ$ the
  fiber coordinates have weight one, so the determinant of the relative
  cotangent complex of $T^*M \to M$ carries weight $d$ and $K_{T^*M^\circ}
  \simeq \cO\langle d \rangle$, the trivial bundle with $\Gm$ acting in weight
  $d$, while the same computation for $N^*Y \to Y$ gives $K_{N^*Y^\circ} \simeq
  \pi^*(K_Y \otimes \det N_Y) \otimes \cO\langle c \rangle$. The fiber product
  formula then gives $K_{\widetilde{Z}} \simeq \pi^*(K_Y \otimes \det
  N_Y)^{\otimes 2} \otimes \cO\langle 2c - d \rangle$ as equivariant bundles.
  The tautological section of $p^*\cL^\vee$ has weight one, so
  $\cO\langle 1 \rangle$ descends to $\cL^\vee$ and $\cO\langle 2c - d
  \rangle$ to $\cL^{\otimes (d - 2c)}$, and by Proposition
  \ref{prop:orientation_obstruction} the descent of $K_{\widetilde{Z}}$ is
  $K_Z$. This recovers the formula of Step 1 and fixes the sign of the weight
  convention.
\end{remark}

The parity condition in (2) is the parity defect of Section \ref{sec:joyce}
read on a global object. It shows that the obstruction of Theorem
\ref{thm:B} is not confined to charts.

\subsection{The Monodromy Zeta Function}
\label{subsec:zeta}

The higher traces of $\theta$ assemble into a single generating function.

\begin{definition}[Contact zeta function]\label{def:zeta}
  Let $Z$ be a proper oriented $-1$-shifted contact derived Artin stack. Its
  \bfem{contact zeta function} is the formal power series
  \[ \zeta_Z(t) := \exp\Bigl( \sum_{m \geq 1} \mathrm{DT}_m(Z)\,
    \frac{t^m}{m} \Bigr) \in \overline{\mathbb{Q}}_\ell[[t]] , \]
  where $\mathrm{DT}_m(Z) = \sum_i (-1)^i \mathrm{Tr}\bigl( \theta^m \mid
  \mathbb{H}^i_{et}(Z, \mathcal{P}_Z) \bigr)$ are the higher traces of
  \S\ref{subsec:highertraces}. For a geometric point $x$ of $Z$ the
  \bfem{local zeta function} is
  \[ \zeta_{Z, x}(t) := \prod_{i \in \mathbb{Z}} \det\bigl( 1 - t\,\theta_x
    \mid \mathcal{H}^i(\mathcal{P}_Z)_x \bigr)^{(-1)^{i+1}} . \]
\end{definition}

\begin{proposition}[Rationality and stratification]\label{prop:zeta}
  Let $Z$ be a proper oriented $-1$-shifted contact derived Artin stack.
  \begin{enumerate}
    \item $\zeta_Z(t)$ is the rational function
      \[ \zeta_Z(t) = \prod_{i \in \mathbb{Z}} \det\bigl( 1 - t\,\theta \mid
        \mathbb{H}^i_{et}(Z, \mathcal{P}_Z) \bigr)^{(-1)^{i+1}} . \]
    \item There is a finite stratification of $Z^{\mathrm{cl}}$ by locally
      closed substacks $S_\alpha$ on which the local zeta function is constant,
      with value $\zeta_\alpha(t)$, and for every such stratification
      \[ \zeta_Z(t) = \prod_\alpha \zeta_\alpha(t)^{\chi_{et, c}(S_\alpha)} . \]
  \end{enumerate}
\end{proposition}

\begin{proof}
  For a finite dimensional vector space $V$ with an automorphism $A$ one has
  the identity of formal power series
  \[ \exp\Bigl( \sum_{m \geq 1} \mathrm{Tr}(A^m)\, \frac{t^m}{m} \Bigr) =
    \det(1 - tA)^{-1} . \]
  Indeed, writing $A$ in triangular form with eigenvalues $\lambda_1, \dots,
  \lambda_r$ gives $\mathrm{Tr}(A^m) = \sum_j \lambda_j^m$, so the exponent on
  the left is $-\sum_j \log(1 - \lambda_j t)$, and exponentiating gives
  $\prod_j (1 - \lambda_j t)^{-1} = \det(1 - tA)^{-1}$. The hypercohomology
  groups $\mathbb{H}^i_{et}(Z, \mathcal{P}_Z)$ are finite dimensional and
  vanish for almost all $i$, as $Z$ is proper and $\mathcal{P}_Z$ is
  constructible. Applying the identity to each of them with the sign
  $(-1)^i$,
  \[ \sum_{m} \mathrm{DT}_m(Z)\, \frac{t^m}{m} = \sum_i (-1)^i \sum_m
    \mathrm{Tr}\bigl( \theta^m \mid \mathbb{H}^i \bigr) \frac{t^m}{m} =
    - \sum_i (-1)^i \log \det\bigl( 1 - t\,\theta \mid \mathbb{H}^i \bigr) , \]
  and exponentiating proves (1).

  For (2), Theorem \ref{thm:integration} applied to the automorphism $\theta^m$
  in place of $\theta$, as recorded in \S\ref{subsec:highertraces}, gives
  $\mathrm{DT}_m(Z) = \chi_{et, c}(Z^{\mathrm{cl}}, \nu^{(m)}_Z)$ with
  \[ \nu^{(m)}_Z(x) = \sum_i (-1)^i \mathrm{Tr}\bigl( \theta_x^m \mid
    \mathcal{H}^i(\mathcal{P}_Z)_x \bigr) . \]
  Fix a finite stratification of $Z^{\mathrm{cl}}$ by connected locally closed
  substacks $S_\alpha$ on which the cohomology sheaves
  $\mathcal{H}^i(\mathcal{P}_Z)$ are lisse, as in the proof of Theorem
  \ref{thm:integration}. On such a stratum $\theta$ is an automorphism of a
  lisse sheaf, so its characteristic polynomial on the stalks is locally
  constant, and the local zeta function, which is determined by these
  characteristic polynomials, is constant on $S_\alpha$ with value
  $\zeta_\alpha(t)$. The same holds for every refinement of this
  stratification, which proves the existence assertion. On such a
  stratification every $\nu^{(m)}_Z$ is constant on $S_\alpha$, with value
  $\nu^{(m)}_\alpha$, and
  \[ \mathrm{DT}_m(Z) = \sum_\alpha \nu^{(m)}_\alpha \cdot
    \chi_{et, c}(S_\alpha) . \]
  Applying the identity of the first paragraph to the stalks at a point of
  $S_\alpha$ gives $\sum_m \nu^{(m)}_\alpha\, t^m / m = \log \zeta_\alpha(t)$,
  so
  \[ \sum_m \mathrm{DT}_m(Z)\, \frac{t^m}{m} = \sum_\alpha
    \chi_{et, c}(S_\alpha) \log \zeta_\alpha(t) , \]
  and exponentiating proves (2).
\end{proof}

\begin{example}[The zeta function of an $A_k$-point]\label{ex:Ak_zeta}
  For the contact Darboux chart $Z = \Delta\mathrm{loc}(x^{k+1})$ of
  \S\ref{subsec:Ak}, whose classical truncation is a point, the sheaf
  $\mathcal{P}_Z$ is concentrated in degree $0$ with stalk
  $\widetilde{H}^0(F_0)$ of dimension $k$. The operator $T_{U, \widetilde{f}}$
  is the inverse of the cyclic monodromy $T_0$ of Example
  \ref{ex:Ak_classical}, whose characteristic polynomial on
  $\widetilde{H}^0(F_0)$ is $(x^{k+1} - 1)/(x - 1)$, a polynomial whose roots
  are closed under inversion, and $\theta = -T_{U, \widetilde{f}}$ by Example
  \ref{ex:Ak_higher}. Using $\prod_{\zeta^{k+1} = 1}(1 - u\zeta) = 1 -
  u^{k+1}$ with $u = -t$,
  \[ \det\bigl( 1 - t\theta \mid \widetilde{H}^0(F_0) \bigr) =
     \frac{\prod_{\zeta^{k+1} = 1} (1 + t\zeta)}{1 + t} =
     \frac{1 - (-t)^{k+1}}{1 + t} , \]
  and Proposition \ref{prop:zeta}(1) gives
  \[ \zeta_Z(t) = \frac{1 + t}{1 - (-t)^{k+1}} . \]
  For $k = 1$ this is $1/(1 - t)$, in agreement with $\mathrm{DT}_m \equiv 1$.
  For $k = 2$ the expansion of $\log \zeta_Z(t) = \log(1 + t) - \log(1 + t^3)$
  returns $\mathrm{DT}_1 = 1$, $\mathrm{DT}_2 = -1$ and $\mathrm{DT}_3 = -2$,
  in agreement with Example \ref{ex:Ak_higher}. Comparing with the monodromy
  zeta function $\zeta_{f,0}(t) = (1 - t)/(1 - t^{k+1})$ of Example
  \ref{ex:Ak_classical},
  \[ \zeta_Z(t) = \zeta_{f, 0}(-t) , \]
  the substitution $t \mapsto -t$ being the sign twist $\theta = -T_{U,
  \widetilde{f}}$. Proposition \ref{prop:zeta}(2) is in this sense a global
  form of A'Campo's formula \cite{ACampo}. By Proposition \ref{prop:zeta}(2), a proper oriented
  $-1$-shifted contact stack whose
  transversal singularity type along a stratum $S_\alpha$ is $A_{k_\alpha}$ has
  \[ \zeta_Z(t) = \prod_\alpha \Bigl( \frac{1 + t}{1 - (-t)^{k_\alpha + 1}}
    \Bigr)^{\chi_{et, c}(S_\alpha)} . \]
\end{example}

\subsection{Moduli Spaces}
\label{subsec:moduli_app}

The variety $M$ of \S\ref{subsec:projcot} may itself be a moduli space. Smooth
projective moduli spaces of stable bundles of coprime rank and degree on a
curve, of stable sheaves on a surface, and Hilbert schemes of points on a
surface all satisfy the hypotheses of Proposition \ref{prop:conormal}, and
their cotangent bundles are the corresponding moduli of Higgs bundles or Higgs
sheaves on the stable locus. For a smooth projective moduli space of stable
sheaves on a K3 or abelian surface the dimension is even, so by Proposition
\ref{prop:conormal_orientation}(1) every conormal self-intersection admits
contact orientation data, and Proposition \ref{prop:conormal} computes its
invariant unconditionally. The fixed locus of a finite order automorphism of
$M$, and any smooth Brill--Noether locus, furnish smooth closed subvarieties
$Y$ whose conormal self-intersections have invariant $c \cdot
\chi_{et, c}(Y)$.

The stack $\mathrm{Bun}_G(C)$ of $G$-bundles on a smooth projective curve $C$
is the stacky case. The global nilpotent cone $\mathcal{N}ilp \subseteq
T^*\mathrm{Bun}_G(C)$ is a Lagrangian substack, in the sense that its pullback
to any smooth presentation of $\mathrm{Bun}_G(C)$ is Lagrangian \cite[Main
Theorem]{GinzburgNilpotent}, \cite[Theorem 2.10.4]{BD}, and it is stable under $\Gm$.
A $\Gm$-stable Lagrangian substack of a cotangent stack is a union of conormal
bundles to reduced irreducible closed substacks \cite[Remark (iii) after
Theorem 2.10.2]{BD}, so the irreducible components of $\mathcal{N}ilp$ are
closures of conormals. 

For $PSL_n$ some of the corresponding substacks are
described in \cite[\S\S 3.8--3.9]{Laumon}, and for $GL_n$ the components of
the semistable locus are described combinatorially in \cite{Bozec}. The projectivization
$P\mathrm{Higgs}_G(C) = \mathbb{P}T^*\mathrm{Bun}_G(C)$ is $0$-shifted
contact \cite[Corollary 5.4]{IzbudakBerktav2026}, and \cite[Corollary
5.6]{IzbudakBerktav2026} equips the derived intersection of the projectivizations
of two components, once these are given $\Gm$-equivariant Lagrangian
structures, with a $-1$-shifted contact structure. Lemma
\ref{lem:torsor_vanishing} applies to that intersection as stated, its
argument being insensitive to whether the base is a scheme, and shows that the
Euler characteristic and the Behrend-weighted Euler characteristic of the
corresponding intersection of components of $\mathcal{N}ilp$ vanish for every
$G$ and every $C$. The contact invariant requires properness, which fails on
$\mathrm{Bun}_G(C)$ and holds after restriction to a finite type open substack
containing the intersection, when one exists. Its evaluation by the conormal
formula of Proposition \ref{prop:conormal} requires in addition the extension
of that proposition to smooth Artin stacks. The conormal of a smooth closed
substack carries a canonical Lagrangian structure by \cite{Calaque2019}, so this
extension is a matter of properness alone. We leave both aside.

\section{Concluding Remarks}
\label{sec:conclusion}

In this paper we have attached $\ell$-adic invariants to $-1$-shifted contact
derived Artin stacks by working on the symplectification and descending along
the structural $\Gm$-action. Theorem \ref{thm:monodromic_sheaf} produces the
perverse sheaf $\mathcal{P}_Z$ together with its twisted monodromy operator
$\theta$, and the central difference from the symplectic setting is contained
in Proposition \ref{prop:sign_character}. The $\Gm$-action does not preserve
the canonical orientation of a contact Darboux chart, the character of its
action being $(-1)^{n+1}$, so the geometric monodromy and the twisted operator
differ by the global sign $T = -\theta$. The pushforward to $Z$ retains only one
generalized eigenspace of $\theta$, and the pair $(\mathcal{P}_Z, \theta)$
retains both. The sign has two consequences. The first is the parity
defect of Theorem \ref{thm:B}. A nonzero Legendrian class can be
$\theta$-invariant only when its virtual dimension is even, both parities occur over a single
chart by Example \ref{ex:A1_parity}, and the contact Joyce conjecture is
therefore stated for graded orientation data. Theorem
\ref{thm:linearization_functor} shows that, under the assumptions of Setup
\ref{setup:joyce}, the non-linear $2$-category of proper Legendrian correspondences linearizes to
the enriched $2$-categories of Theorems \ref{thm:linear_category} and
\ref{thm:global_2cat}.

The second is the content of Sections \ref{sec:dt} and
\ref{sec:applications}. The contact
Behrend function is identically $1$ by Proposition \ref{prop:behrend_one}, a
consequence of A'Campo's theorem, so the invariant of Theorem
\ref{thm:integration} is the compactly supported Euler characteristic of the
classical truncation, and the information that the symplectic theory encodes
in the Behrend function is carried instead by $\theta$. Its higher traces
distinguish the $A_k$-singularities by Example \ref{ex:Ak_higher} and assemble
into the zeta function of Proposition \ref{prop:zeta}, a global form of
A'Campo's formula. Theorem \ref{thm:conic} then identifies a class of objects
on which the symplectic invariant vanishes identically and the contact
invariant does not. The derived intersection of two conic Lagrangians in a cotangent
bundle has vanishing symplectic invariant for a structural reason, the free
$\Gm$-action, while its contact invariant computes the Euler characteristic of
the projectivized intersection, with the value $e \cdot \chi_{et, c}(W)$ for
conormal bundles meeting cleanly with excess $e$ by Proposition
\ref{prop:conormal}. Proposition \ref{prop:conormal_orientation} reads the
parity defect on these objects, the self-intersection of the conormal of a
subvariety of codimension at least $2$ in a variety of dimension $d$
admitting contact orientation data exactly when $d$ is even. The examples of Remark \ref{rem:not_virtual} show that the contact
invariant counts singular points of a hypersurface rather than critical points
of a function, and that this is what distinguishes it from a virtual count.

\section{Outlook}
\label{sec:outlook}

Three threads are left open by the constructions above. The first is the
conjecture on which Sections \ref{sec:joyce} and \ref{sec:linearization} depend,
the second is the supply of examples beyond the local models computed here, and
the third is the relation of the grading to the twists appearing elsewhere in
the subject.

\subsection{Toward the Contact Joyce Conjecture}
The main problem left open here is Conjecture \ref{conj:joyce}, and three
ingredients are missing from a proof.

The first is canonicity. Proposition \ref{prop:local_verification}
constructs the
local class from a presentation and shows it is unchanged under a change of
trivialization of the torsor, but not that it is determined by $L$ alone. Were
the symplectic classes of (J1) known to be unique in their $\Hom$-groups, or
pinned by a normalization, the monodromic refinement (J3) would follow. The
operator $\theta$ arises from the $\Gm$-action, which preserves the
symplectic form
up to weight and preserves the orientation, so it fixes any class characterized
by data it preserves. Establishing (J3) would then reduce to a canonicity
statement for (J1) together with Remark \ref{rem:local_graded}.

The second is gluing. The local morphisms of Proposition
\ref{prop:local_verification} are maps of constructible complexes rather than of
perverse sheaves, so they do not descend by sheaf-theoretic gluing, and the
obstruction is the extension class of Remark \ref{rem:gluing}.

The third is the symplectic input. Assumptions (J1)--(J2) are theorems in the
attractor case \cite{KPS}, whose proof proceeds through relative
perverse sheaves
and hyperbolic localization. A general construction of Lagrangian classes has
been announced by Khan, Kinjo, Park and Safronov. If that construction is
compatible with monodromic structures, then (J3), and with it Theorem
\ref{thm:C}, becomes unconditional.

\subsection{Moduli examples}

Section \ref{sec:applications} applies the theory to conic Lagrangian
intersections in cotangent bundles of smooth proper varieties, including
smooth projective moduli spaces, and to the nilpotent cone of
$\mathrm{Bun}_G(C)$ through \cite[Corollary 5.6]{IzbudakBerktav2026}. Two things
are missing. The conormal formula of Proposition \ref{prop:conormal} is
stated for varieties. Its extension to smooth Artin stacks is available in
principle, the conormal of a smooth closed substack carrying a canonical
Lagrangian structure by \cite{Calaque2019}, but requires a finite type
restriction for properness. The invariant of the nilpotent cone intersection requires a
finite type restriction on which it is proper, and a computation of the
excess along the intersection of two components. The components are
closures of conormals to closed substacks of $\mathrm{Bun}_G(C)$
\cite[Remark (iii) after Theorem 2.10.2]{BD}, so the second point is a
question about those substacks, which \cite{Bozec} describes combinatorially
for $GL_n$ on the semistable locus.

\subsection{The parity defect and orientation data}
Proposition \ref{prop:parity_formula} identifies the parity defect of a
Legendrian with the parity of its virtual dimension, and Example
\ref{ex:odd_nonzero} shows that the resulting grading is not a normalization. On
an odd-dimensional Legendrian the ungraded requirement annihilates a class that
is otherwise a generator of a one-dimensional space.

We do not know whether this has a symplectic counterpart. The operator $\theta$
exists only in the presence of the weight-one $\Gm$-action, so the phenomenon is
invisible in the setting of \cite{AmorimBassat}. A half-twist of the same
shape does appear in the motivic and Hodge-theoretic treatments of vanishing
cycles, where it is written $\mathbb{L}^{1/2}$ and absorbed into the
definition of
the monodromic Grothendieck ring \cite{BJM}. Whether Definition
\ref{def:graded_orientation} is a shadow of that construction, and whether the
graded orientations of a fixed Legendrian form a torsor under a group other than
$\mu_2$, are concrete questions we leave open.

\subsection{Relation to microlocal sheaf theory}
Over $\mathbb{C}$, the Nadler--Zaslow theorem identifies the Fukaya
category of a
cotangent bundle with constructible sheaves on the base \cite{NadlerZaslow}. The
2-category $\LFc(X)$ is built from constructible complexes attached to
Legendrians, so the comparison is suggestive, but we make no claim of a
correspondence. Our $2$-morphism spaces are hypercohomologies of vanishing-cycle
sheaves rather than morphism complexes in a Fukaya category, and no functor
between the two is constructed here. Identifying a candidate functor, even for a
$1$-shifted cotangent stack, would be the first step toward such a comparison.

\section*{Acknowledgments}
The author thanks Kadri \.{I}lker Berktav for many discussions on
shifted contact
geometry and for collaboration on the prior work this paper builds upon.

\bibliographystyle{amsalpha}
\bibliography{refss}

\end{document}